\documentclass[12pt,reqno]{amsart}

\usepackage[usenames,dvipsnames]{color}
\usepackage{amsmath,amssymb,amsthm,graphicx,mathrsfs,url,latexsym,enumerate}
\usepackage{color}
\usepackage{a4wide}
\usepackage{ifthen}
\usepackage{verbatim}
\usepackage{fancyhdr}
\usepackage{url}
\usepackage{bbm}
\usepackage[sans]{dsfont}
\usepackage{MnSymbol}

\usepackage{bm}

\usepackage{amsbsy}
\usepackage[colorlinks=true,linkcolor=blue,citecolor=green]{hyperref}
\usepackage{wrapfig}
\usepackage{tikz}
\usepackage{geometry}

\definecolor{bleu}{RGB}{27,88,145}
\definecolor{mauve}{RGB}{138,20,79}

\newcommand{\rev}[1]{{#1}}

\renewcommand{\Re}{\operatorname{Re}}

\newcommand{\supp}{\operatorname{supp}}

\def\<{\langle}
\def\>{\rangle}

\newcommand{\be}{\begin{equation}}
\newcommand{\ee}{\end{equation}}
\newcommand{\bes}{\begin{equation*}}
\newcommand{\ees}{\end{equation*}}

\newcommand{\dom}{\operatorname{dom}}
\newcommand{\SigD}{\Sigma_{\rm D}}
\newcommand{\SigN}{\Sigma_{\rm N}}

\newcommand{\HsD}{H^1_{0,\SigD}(\mathscr D)}

\newcommand{\ve}{\varepsilon}

\def\Re {{\rm \,Re\;}}

\def\dim {{\rm \; dim  \;}}

\numberwithin{equation}{section}
\numberwithin{figure}{section}

\newtheorem*{theorem*}{Theorem}
\newtheorem*{definition*}{Definition}
\newtheorem*{assumption*}{Assumption}

\newtheorem{theorem}{Theorem}
\newtheorem{problem}{Problem}

\newtheorem{definition}{Definition}[section]

\newtheorem{lemma}[definition]{Lemma}
\newtheorem{proposition}[definition]{Proposition}
\newtheorem{remark}[definition]{Remark}

\newcommand{\ass}[1]{\hyperref[ass.#1]{\rm\textbf{[#1]}}}

\def\Lr{{\mathscr L}}

  \title[Entropic metastability and quasi-stationary distributions]{Entropic metastability for drifted reflected Brownian motion: quasi-stationary distributions, regularity, and the  exit position distribution}

\author[L. Carillo]{Louis Carillo}
\address{L. Carillo, CERMICS, CNRS, ENPC, Institut Polytechnique de Paris, Marne-la-Vallée, France. MATHERIALS project-team, Inria Paris, France.}
\email{louis.carillo@enpc.fr.}

\author[L. Delande]{Loïs Delande}
\address{L. Delande,  Laboratoire AGM - Analyse, Géométrie, Modélisation, CY Cergy Paris Université.}
\email{lois.delande@cyu.fr}

\author[B. Nectoux]{Boris Nectoux}
\address{B. Nectoux, Laboratoire de Mathématiques Blaise Pascal, Université Clermont Auvergne}
\email{boris.nectoux@uca.fr}

\begin{document} 
 \maketitle

\begin{abstract}Motivated by entropic metastability and the narrow escape problem, we investigate quasi-stationary distributions for reflected, drifted elliptic diffusions under mixed absorbing-reflecting boundary conditions on a bounded domain. We address two problems, which are solved completely independently of each other. First, using purely probabilistic tools, we prove the existence and uniqueness of a quasi-stationary distribution under minimal regularity assumptions on the coefficients (the case of an unbounded domain is also analyzed). Second, using spectral theory and the theory of strongly continuous semigroups on $L^p$ spaces, we study the analytical properties of the quasi-stationary distribution. We relate it to the principal eigenfunction of the adjoint of the associated (non-symmetric) infinitesimal generator with mixed absorbing-reflecting boundary conditions, and we establish its regularity up to the interface between the two boundary parts. We also fully characterize the law of the exit point when the process is initially distributed according to its quasi-stationary distribution. In particular, we work on domains in which the absorbing (Dirichlet) and reflecting (Neumann) parts of the boundary meet tangentially, without forming any non-trivial angle, a configuration that raises genuine regularity difficulties for the solutions of the associated mixed boundary value problem.
  \medskip

\medskip

\end{abstract}


 \section{Setting}
 
 \subsection{Purpose of this work}
 
  \subsubsection{Entropic metastability}
  \label{sec.NE}

Entropic metastability occurs when a stochastic process is trapped not by energetic potential wells, but by geometric bottlenecks that severely constrain access to exit routes. Unlike classic metastability driven by energetic barriers, entropic barriers constitute a fundamental paradigm in molecular dynamics and physical chemistry, where spatial confinement governs the long-time behavior \cite{ZhouZwanzig1991,Zwanzig1992,LeSt16}. In $2$ and $3$ dimensional biological media, a key manifestation of entropic trapping is the narrow escape problem \cite{HoSc04, HoSc14, HoSc15}, which models cellular receptors temporarily confined within membrane corrals. This geometric framework forms the precise entropic metastable setting addressed in this work.

In the narrow escape problem, the corral is modeled as a bounded domain $\mathscr{D} \subset \mathbb{R}^d$ ($d=2,3$)\rev{. Let us mention right away that we work in an arbitrary dimension $d\ge 2$ throughout this work}. The receptor follows a Brownian motion reflected on the boundary $\partial \mathscr{D}$, except at $N$ small disjoint escape windows $\Sigma_i \subset \partial \mathscr{D}$ where it is absorbed. While existing literature extensively investigates the asymptotic regime as the window sizes vanish ($|\Sigma_i| \to 0$) \cite{HoSc04, AmKaKaLe12, NuTzTz21, Fa24, CaNoLe26}, this work focuses on entropic metastability for fixed escape windows $\Sigma_i$.

 To analyze this system at fixed window sizes, we deploy the framework of quasi-stationary distributions.  
Successfully applied to energetic metastability \cite{DiLeLeNe16, DiLeLeNe20, lelievre2025eyring}, the quasi-stationary approach exploits the property that a process reaches local equilibrium long before escaping. This motivates the investigation of the exit event when the process is initially distributed according to its quasi-stationary distribution \cite{DiLeLeNe16}. This framework allows the probabilistic properties of the exit event to be rigorously analyzed using partial differential equations and spectral theory \cite{BN2017,DiLeLeNe19}, once the quasi-stationary distribution is explicitly characterized in terms of the principal eigenfunction of the generator in a suitable $L^p$ space, under appropriate boundary conditions.
 Extending this spectral machinery to entropic barriers requires establishing the existence and regularity of the principal eigenvector of the generator under mixed absorbing-reflecting boundary conditions. To lay the theoretical foundations for this approach, we address two core questions (Problems \ref{Pb.1} and \ref{Pb.2} below). Before formulating these problems and stating our main results, we specify the geometric setup and stochastic model describing entropic confinement, which generalizes the narrow escape framework to arbitrary dimensions.

\subsubsection{Reflected Brownian motion with boundary absorption}\label{sec.RBMabs}
Let $(\Omega, \mathcal{F}, (\mathcal{F}_t)_{t \ge 0}, \mathbb{P})$ be a filtered probability space satisfying the usual conditions, and let $(B_t)_{t \ge 0}$ denote a standard $d$-dimensional \rev{$(\mathcal F_t)_{t\ge 0}$-}Brownian motion ($d \ge 2$). Throughout this work (except in Appendix \ref{app.unbounded}), $\mathscr{D} \subset \mathbb{R}^d$ is assumed to be a bounded domain \rev{(i.e. a nonempty, open, connected, and bounded subset of $\mathbb R^d$)} with at least $C^3$  boundary. Denote by $\mathbf n_i(z)$    
  the   inward-pointing unit normal vector    of $\mathscr D$ at $z\in \partial \mathscr D$.  
          We also denote by $\mathbf 1_A$ the indicator function of a set $A\subset \mathbb R^d$.  
          Let $b: \bar{\mathscr D}\to \mathbb{R}^d$ and $\sigma: \bar{\mathscr D}\to \mathbb{R}^{d \times d}$ be measurable functions such that $\sigma(x)$ is an invertible matrix for every $x \in \bar{\mathscr D}$ (precise assumptions will be specified later).
For $x\in \bar{\mathscr D}$,  we consider    the solution $(X_t,L_t)_{t\ge 0}$   to the  non reversible (normal) reflecting stochastic differential equations over $\bar{\mathscr D}$ drifted by $b$:
\begin{equation}\label{eq.LR}
 dX_t=  b(X_t)dt+ \sigma(X_t)dB_t + d L_t,
\end{equation}  
  where:   
  \begin{enumerate}
  \item[$\mathfrak 1$.]  $X_t\in \bar{\mathscr D}$ for all $t\ge 0$.
    \item[$\mathfrak 2$.] A.s. $(X_t)_{t\ge 0}$ and $(L_t)_{t\ge 0}$ are continuous, and $(L_t)_{t\ge 0}$ has locally bounded variations.  
  \item[$\mathfrak 3$.]  $L(0)=0$ and for all $t\ge 0$: 
\begin{equation}\label{eq.L=}
|L|_t= \int_0^t \mathbf 1_{X_s\in \partial \mathscr D} d|L|_s \text{ and }   L_t= \int_0^t \mathbf n_i(X_s) d|L|_s,
\end{equation}  
  \end{enumerate}
  where $|L|_t$ denotes the total variation of $L_t$ by time $t$. 

  Let $\Sigma$ be  a subset   of $\partial \mathscr D$\rev{ (only its closure $\bar\Sigma$, which is the absorbing part of the boundary, enters the problem, see Remark~\ref{re.Sigmaclosed} below)}. Set
\begin{equation}\label{eq.Oo}
\mathscr O:= \bar{\mathscr D}\setminus \bar{\Sigma} \ \ \text{ and } \ \  \tau_{\mathscr O}= \inf\{t\ge 0, X_t\notin \mathscr O\}.
\end{equation} 
 Note that $\tau_{\mathscr O}= \inf\{t\ge 0, X_t\in \bar{\Sigma}\}$ and  that $\mathscr O$ is an open subset of $\bar{\mathscr D}$, hence it is a Polish space.   
In the narrow escape problem (Section \ref{sec.NE}), $\Sigma=\cup_{i=1}^N  {\Sigma}_i$. 
 For  $x\in \mathscr O$, we define the killed (outside $\mathscr O$) semigroup 
 \begin{equation}\label{eq:PtO}
     P_t^{\mathscr O}f(x)= \mathbb E_x[f( X_t) \mathbf 1_{t<\tau_{\mathscr O}}],
 \end{equation}
 for every bounded and measurable function $f:\mathscr O\to \mathbb R$. 
We recall the definition of a quasi-stationary distribution adapted to the present setting (see e.g.~\cite{collet2012quasi,meleard2012quasi}).

\begin{definition}\label{QSD}
A quasi-stationary distribution   for the reflected diffusion
$(X_t)_{t\ge 0}$ with absorbing boundary part $\bar \Sigma$ is a probability
measure $\rho$ on $\mathscr O$ such that, for every
$t>0$ and every Borel set $\mathcal A\subset \mathscr O$,
$$
\rho (\mathcal A)
=
\mathbb P_{\rho}
[
X_t\in\mathcal A \, |\, 
t<\tau_{\mathscr O}].
$$
\end{definition}
Roughly speaking,  in this setting, a quasi-stationary distribution  is the invariant probability measure of the process $(X_t)_{t \ge 0}$ conditioned on not having exited the domain $\mathscr{O}$ by time $t$, for all $t \ge 0$.

\subsubsection{Objectives and related literature}
 
The primary objective of this work is to establish the theoretical foundations of the quasi-stationary distribution approach for entropic metastability in the extended narrow escape setting. As detailed previously in Section~\ref{sec.NE}, achieving this requires addressing two core questions regarding this metastable entropic confinement in arbitrary dimensions $d \ge 2$:
 
\begin{problem}\label{Pb.1}
    \textit{(Existence and uniqueness)} Does a quasi-stationary distribution  exist for the reflected drifted diffusion
$(X_t)_{t\ge 0}$ with absorbing boundary part $\bar \Sigma$  (see Definition \ref{QSD})?
\end{problem}

\begin{problem}\label{Pb.2}
    \textit{(Spectral, analytical, and exit properties)} If a quasi-stationary distribution exists, what are its analytical and spectral properties, and can we explicitly characterize the  exit position distributions for the process initiated from the quasi-stationary distribution?
\end{problem}

Problem~\ref{Pb.1} is solved in Theorem~\ref{th.1}, which is the first main result of this work, where we establish the existence and uniqueness of the quasi-stationary distribution under minimal assumptions. Problem~\ref{Pb.2} is addressed in Theorem~\ref{th.2}, the second and last main result of this work, which provides a complete spectral characterization and describes the exit dynamics for the process initiated from the quasi-stationary distribution. Let us also mention that the proof of Theorem~\ref{th.2} provides the existence of a quasi-stationary distribution using tools independent of those employed in Theorem~\ref{th.1}.

Existence and uniqueness of a quasi-stationary distribution   in narrow tubes, a related entropic metastable setting, were recently established in \cite{Hsu25} (see also the recent work \cite{Hsu25b}), utilizing an approach similar to the criteria introduced in \cite{champagnat2017general}.   Our approach to the first question is   different. Indeed, as detailed below, we address Problem \ref{Pb.1} by applying a recent general criterion for the existence and uniqueness of quasi-stationary distributions established in \cite{guillinqsd}, see Section \ref{sec.Gu}. 
Other criteria for the existence and uniqueness of a quasi-stationary distribution can be found in \cite{champagnat2017general,ferre2,CloezHarris,del2023stability}. 
It should also be noted that Problem \ref{Pb.1} was studied in parallel and independently in~\cite{CaLeNo26},   under the assumptions $b \equiv 0$ (no drift) and $\sigma \equiv \sqrt 2\,  \mathrm{Id}_{\mathbb{R}^d}$ (see \eqref{eq.LR}). In that setting the \rev{realization in $L^2(\mathscr D)$ of the differential operator} $\mathscr L$ (see \eqref{eq.Lo}) with mixed boundary conditions  is self-adjoint, and the approach developed there relies crucially on a spectral decomposition along a complete orthonormal system of eigenfunctions spanning the whole space $L^2(\mathscr D)$. This technique is no longer available when $b \not\equiv 0$, since $\mathscr L$ then fails to be symmetric. For this reason, the approach we develop here to address Problem \ref{Pb.2} is necessarily different and more involved, see Section \ref{sec.St2} for more details. It is worth emphasizing that our first main result Theorem~\ref{th.1} meaningfully expands the scope of~\cite{CaLeNo26}.  Indeed, Theorem~\ref{th.1}, apart from requiring minimal regularity assumptions on both $b$ and the diffusion matrix $\sigma$, imposes no regularity conditions on the absorbing set $\Sigma \subset \partial \mathscr{D}$ (no regularity is required on its \rev{relative} boundary $\partial\Sigma$\rev{, see Section~\ref{sec.notations}}). We also mention that the techniques we use to solve Problem \ref{Pb.1} allow to deal with unbounded domains, see the results in Appendix~\ref{app.unbounded} (Theorem \ref{th.unbounded}).




\subsubsection{Notations}\label{sec.notations}
Let $\mathscr E$ \rev{be} a Polish space.
The set $\mathcal B( \mathscr E)$ is the Borel $\sigma$-algebra of $ \mathscr E$, and $b\mathcal B( {\mathscr E})$ is the  space of all bounded and Borel measurable functions $f: {\mathscr E} \to  \mathbb R$ equipped with the sup-norm
 $$\Vert f \Vert_\infty=\sup_{ x\in  {\mathscr E} }\vert f( x)\vert.$$
 If $W:\mathscr E\to [1,+\infty)$ is a continuous function, 
we denote by  $b_{W}\mathcal B(\mathscr E)$   the Banach space of all $\mathcal B(\mathscr E)$-measurable functions on $\mathscr E$ with norm $\|f\|_{b_{W}\mathcal B(\mathscr E)}:=\sup_{{x}\in \mathscr E} {|f({x})|}/{W({x})}< +\infty$.
The space of probability measures over $\mathscr E$ is denoted by $\mathscr P(\mathscr E)$. We also set $\mathscr P_{W}(\mathscr E)=\{\nu\in  \mathscr P(\mathscr E) \text{ s.t. } \nu(W)< +\infty\}$. 
 The set  $ C_b({\mathscr E})$ denotes the space of bounded continuous real-valued functions over ${\mathscr E}$.
 Given an initial distribution $\nu$ on ${\mathscr E}$, we write $\mathbb P_\nu(\cdot)=\int_{{\mathscr E}} \mathbb P_{x}(\cdot) \nu(d{x}) $. 
 The space $C([0,T],\mathscr E)$ of $\mathscr E$-valued continuous paths   defined on $[0,T]$ is equipped with the uniform topology.  
The space $\mathscr P(C([0,T], \mathscr E))$ of probability measures on  $C([0,T],\mathscr E)$ is 
equipped with the weak convergence topology.  
The set $\mathscr B(x,r):=\{y\in \mathbb R^d, |y-x|<r\}$ denotes the open ball of radius $r>0$ centered at $x\in \mathbb R^d$.  \rev{For a subset $A$ of $\partial\mathscr D$, we denote by $\partial A$ its relative boundary in $\partial\mathscr D$.}   We  introduce the operator $\mathscr L$ defined by, for sufficiently smooth function $\phi:\bar{\mathscr D}\to \mathbb R$,
\begin{equation}\label{eq.Lo}
\mathscr L\phi= \frac 12\text{Tr}(\sigma\sigma^*{\rm Hess}\, \phi \, ) +  b\cdot \nabla \phi.
\end{equation}
 \rev{The operator $\mathscr L$ is the differential operator associated with \eqref{eq.LR} in the following sense: by It\=o's formula and \eqref{eq.L=}, for every $\phi\in C^2(\bar{\mathscr D})$ with $\nabla\phi\cdot\mathbf n_i=0$ on $\partial\mathscr D$, $\phi(X_t)-\phi(X_0)-\int_0^t\mathscr L\phi(X_s)\,ds$ is a martingale. }
  We use the standard notation $W^{s,p}(\mathscr D)$ for the Sobolev spaces over $\mathscr D$, with $W^{0,p}(\mathscr D)=L^p(\mathscr D)$ and $W^{s,2}(\mathscr D)=H^s(\mathscr D)$. The inner product in $L^2(\mathscr D)$ is denoted by $\langle\cdot,\cdot\rangle_{L^2(\mathscr D)}$. We denote by $\mathbf n_e=-\mathbf n_i$ is the outward-pointing unit normal to $\mathscr D$.
We also denote the  spectral radius  of a bounded linear operator $T$ on a Banach space by $\mathsf r_{sp}(T)$ which we recall  is the supremum of the absolute values of the elements in its spectrum and satisfies by the  Gelfand's formula, $\mathsf r_{sp}(T) = \lim_{n \to \infty} \Vert T^n\Vert^{1/n}$.

\subsection{Problem \ref{Pb.1}: Existence and uniqueness of a quasi-stationary distribution}

We address the existence and uniqueness of a quasi-stationary distribution (Problem \ref{Pb.1}) in two  low-regularity setting. 

\subsubsection{Lipschitz coefficients $b$ and $\sigma$}\label{sec.Lip} 
The first setting we consider is the following: 
  \begin{enumerate}
\item [{\rm\textbf{[A1]}}\phantomsection\label{ass.A1}] The vector field $b: \bar{\mathscr D}\to \mathbb R^d$  and the diffusion coefficient $\sigma :\bar{\mathscr D}\to \mathbb R^{d\times d}$ are    Lipschitz  functions.
 \end{enumerate} 
   We will also work in this section under  the assumption that $\sigma$ is elliptic in the sense that: 
 \begin{enumerate}
\item [{\rm\textbf{[A2]}}\phantomsection\label{ass.A2}] For all $z\in \bar{\mathscr D}$, $\sigma\sigma^*(z)>0$. 
 \end{enumerate} 
Assumption \ass{A2} is equivalent to requiring that $\sigma(z)$ is invertible for every $z\in\bar{\mathscr D}$.
 
 When  Assumption \ass{A1} above is satisfied,  for every $X_0=x\in \bar{\mathscr D}$, existence and uniqueness of \eqref{eq.LR} are   guaranteed by~\cite[Theorem 3.1]{lionsRSDE}. We denote this process by  $(X_t(x))_{t\ge 0}$ or $X(x)$ (the solution to \eqref{eq.LR} when $X_0=x\in \bar{\mathscr D}$) which  is a strong Markov process~\cite{lionsRSDE}.

\rev{More precisely, under \ass{A1}, \eqref{eq.LR} has a pathwise unique strong solution~\cite[Theorem 3.1]{lionsRSDE}: for every $x\in\bar{\mathscr D}$, the process $(X_t(x),L_t(x))_{t\ge0}$ is defined on the filtered probability space $(\Omega,\mathcal F,(\mathcal F_t)_{t\ge0},\mathbb P)$ of Section~\ref{sec.RBMabs}, it is driven by the Brownian motion $(B_t)_{t\ge 0}$ introduced there, and it is adapted to the ($\mathbb P$-completed) filtration generated by $(B_t)_{t\ge0}$. In particular, all the processes $(X_t(x))_{t\ge0}$, $x\in\bar{\mathscr D}$, are defined on the same probability space and are driven by the same Brownian motion. We denote by $\mathbb P_x$ the law of $(X_t(x))_{t\ge0}$ on $C(\mathbb R_+,\bar{\mathscr D})$, i.e. $\mathbb P_x[(X_t)_{t\ge 0}\in \cdot\,]:=\mathbb P[(X_t(x))_{t\ge 0}\in\cdot\,]$, and by $\mathbb E_x$ the associated expectation ($(\mathbb P_x)_{x\in\bar{\mathscr D}}$ is the strong Markov family associated with \eqref{eq.LR}). Throughout this work, the notation $\mathbb P_x$, $\mathbb E_x$ (and $\mathbb P_\nu$, $\mathbb E_\nu$, see Section~\ref{sec.notations}) is used for quantities which only depend on this law (such as \eqref{eq:PtO} or Definition~\ref{QSD}), whereas the single probability measure $\mathbb P$ (with expectation $\mathbb E$) is used, together with the notation $X(x)$, when several solutions of \eqref{eq.LR} (or of related equations) with different initial conditions are compared on the same probability space, as in the coupling arguments of Section~\ref{sec.proofth1}.}

 \subsubsection{The case when $b$ is just continuous} 
 \label{sec.weak}
We consider another setting where the drift coefficient $b$ is merely continuous, relaxing the global Lipschitz continuity assumed in \ass{A1}, though we   require the diffusion $\sigma$ to be a multiple of $\mathrm{Id}_{\mathbb R^d}$. Namely, we introduce the following assumption
  \begin{enumerate}
\item [{\rm\textbf{[B]}}\phantomsection\label{ass.B}]  $\sigma\equiv \sqrt 2\, \mathrm{Id}_{\mathbb R^d}$ and $b:\bar{\mathscr D}\to \mathbb R^d$ is  continuous. 
 \end{enumerate}
  Note that under \ass{B}, weak existence and weak uniqueness hold for \eqref{eq.LR}.      Weak existence of  \eqref{eq.LR} is ensured by~\cite[Theorem 3.2]{lionsRSDE}.
By \cite[Theorem 1]{KangRamanan2017},  \eqref{eq.LR}  is weakly well posed if and only if the
corresponding submartingale problem is well posed. The latter holds by the results of 
\cite[Section 5]{StroockVaradhan1971} and the solutions  form a strong Markov process. Thus,   under \ass{B}, \eqref{eq.LR}   has a unique weak solution.
\rev{This means that, for every $x\in\bar{\mathscr D}$, there exists a filtered probability space (which a priori depends on $x$) supporting a standard $d$-dimensional Brownian motion and a solution $(X_t(x),L_t(x))_{t\ge0}$ of \eqref{eq.LR} driven by it with $X_0=x$, and that the law $\mathbb P_x$ of $(X_t(x))_{t\ge0}$ on $C(\mathbb R_+,\bar{\mathscr D})$ does not depend on the chosen solution. As above, $\mathbb E_x$ denotes the associated expectation. When, in a proof, we work with a specific weak solution (e.g. in the Girsanov arguments of Proposition~\ref{th.SF2} and Lemma~\ref{le.CT2}), we still denote, with a slight abuse of notation, the underlying filtered probability space by $(\Omega,\mathcal F,(\mathcal F_t)_{t\ge0},\mathbb P)$ and the driving Brownian motion by $(B_t)_{t\ge 0}$. By weak uniqueness, all the quantities appearing in our results only depend on the family $(\mathbb P_x)_{x\in\bar{\mathscr D}}$.}

  The restriction to $\sigma\equiv \sqrt2\,\mathrm{Id}_{\mathbb R^d}$ in \ass{B} comes from our approach. Under \ass{B}, a Girsanov change of measure reduces the strong Feller property of $(X_t)_{t\ge0}$ to that of the reflected Brownian motion. For the latter, we use two known facts: its transition density is continuous on $\bar{\mathscr D}\times\bar{\mathscr D}$~\cite{bass1991some}, and it satisfies the Skorokhod equation~\cite{Hsu2}, which identifies it with the solution of \eqref{eq.LR} when $b\equiv0$. For a general $\sigma$, the Dirichlet form $\int_{\mathscr D} \sigma\sigma^*\nabla\varphi\cdot\nabla\psi\,dx$ on $H^1(\mathscr D)$ corresponds to a reflection along the conormal direction $\sigma\sigma^*\mathbf n_i$, whereas \eqref{eq.LR} reflects along $\mathbf n_i$. These directions coincide when $\sigma\sigma^*$ is a multiple of the identity. Otherwise, the process is an obliquely reflected diffusion, which is in general not symmetric, and the above results are not available in the same form. This already happens for a constant $\sigma$, since the linear change of variables $x\mapsto\sqrt2\,\sigma^{-1}x$ turns the normal reflection into an oblique one. We refer to Section~\ref{sec.ex1} for more details.

Furthermore, working under the two different (and overlapping) sets of hypotheses \ass{A1}-\ass{A2} and \ass{B} adds a valuable methodological richness, as it yields entirely independent proofs of the intermediate results used to establish Theorem~\ref{th.1}.

\subsubsection{Existence and uniqueness of a quasi-stationary distribution}

 The last assumption we consider concerns   the absorbing set $\bar \Sigma$. We naturally  assume that $\bar \Sigma$ is attainable from at least one point, namely:
  \begin{enumerate}
\item [{\rm\textbf{[C]}}\phantomsection\label{ass.C}]  $\exists x\in \mathscr O$ such that $\mathbb P_x[\exists t>0, X_t \in \bar  \Sigma]>0$. 
 \end{enumerate} 
 Assumption \ass{C} is an attainability condition for the solution of the reflected stochastic differential equation. If it does not  hold, then we are concerned with stationary distributions since in this case one has a.s.  $\tau_{\mathscr O}(x)=+\infty$ for all $x\in \mathscr O$.  
 In some simple cases, for instance when $\Sigma$ has non empty interior, it is easy to prove that \ass{C} is satisfied, see Remark~\ref{re.sigma} below.   The first main result of this work is the existence and uniqueness of a quasi-stationary distribution in the narrow escape problem.

\begin{theorem}\label{th.1}
Assume  \ass{A1}-\ass{A2} and  let  $\mathscr D$ be a $C^3$ bounded domain of $\mathbb R^d$, $d\ge 2$  (or assume \ass{B} and let $\mathscr D$ be  a $C^4$ bounded domain of $\mathbb R^d$). Let $\Sigma\subset\partial\mathscr D$ and assume \ass{C}.
  Then,   there exists a unique quasi-stationary distribution  $\bm{\rho}$  in  $\mathscr P(\mathscr O)$  for    the process $(X_t)_{t\ge 0}$ defined in \eqref{eq.LR} with absorbing boundary part $\bar \Sigma$. Moreover:
\begin{enumerate} 
  \item[$\mathfrak 1$.] 
   For all $t>0$, $P^{\mathscr O}_t:b\mathcal B(\mathscr O)\to b\mathcal B(\mathscr O)$ (see \eqref{eq:PtO}) is compact and there exists 
   $$\eta>0$$
      such that $\mathsf r_{sp}(P^{\mathscr O}_t|_{b\mathcal B(\mathscr O)} ) =e^{-\eta t}$, for all $t>0$. 
   Furthermore,  $\bm{\rho}    P^{\mathscr O}_t =e^{-\eta t}\bm{\rho}$, for all $t\ge 0$, and $\bm{\rho} (U)>0$ for all nonempty open subsets $U$ of $ \mathscr O$. In addition,  there is a unique   function $\bm{\varphi} \in C_b ( \mathscr O)$   such that  $\bm{\rho}  (\bm{\varphi} )=1$ and for all $t>0$, $
      P^{\mathscr O}_t \bm{\varphi}= e^{-\eta  t} \bm{\varphi}  \text{ on } \mathscr O,$.  
  Moreover, $\bm{\varphi}>0$   everywhere on $\mathscr O$.
 \item[$\mathfrak 2$] There exist   $c_1>0$, and   $c_2\ge 1$ s.t.  for all  $t>0$ and all   $\nu\in \mathscr P ( \mathscr O)$:
 $$
 \sup_{A\in\mathcal B(\mathscr O)}\big |  \mathbb P_\nu[X_t\in A|t<\tau_{\mathscr O}]-\bm{\rho}   (A)\big |\le c_2 \, e^{-c_1 t} \frac{\nu(\mathbf 1 )}{\nu(\bm{\varphi})}.
 $$
  \item[$\mathfrak 3$.] For all $x\in \mathscr O$, $\mathbb P_x[\tau_{\mathscr O}< +\infty]=1$. 
\end{enumerate} 
 \end{theorem} 
 
\rev{\begin{remark}\label{re.Sigmaclosed}
Only the closure $\bar\Sigma$ of $\Sigma$ enters Theorem~\ref{th.1} (through $\mathscr O=\bar{\mathscr D}\setminus\bar\Sigma$ and $\tau_{\mathscr O}$), and $\bar\Sigma$ is the absorbing part of the boundary. Hence one could equivalently assume from the outset that $\Sigma$ is a closed subset of $\partial\mathscr D$. We keep a general $\Sigma$ because, in Theorem~\ref{th.2} below, $\Sigma$ is taken to be an open subset of $\partial\mathscr D$ with a$C^3$ relative boundary, in order to apply the regularity results for mixed boundary value problems of~\cite{Shamir1968,Savare1997}. In contrast, Theorem~\ref{th.1} requires no regularity assumption on $\bar\Sigma$ (in particular on $\partial\Sigma$).
\end{remark}}
For $X_0$ initially distributed according to $\bm{\rho}$, it is well-known that   the exit time $\tau_{\mathscr O}$   is exponentially  distributed with  parameter $\eta$  (the  absorbing rate) and   is independent of the exit point  $X_{\tau_{\mathscr O}}$ (cf. \cite{collet2012quasi}). The law of $X_{\tau_{\mathscr O}}$ under this initial distribution is analyzed in the next section (see Theorem~\ref{th.2}).  We are also able to treat unbounded domains $\mathscr D$, see Appendix~\ref{app.unbounded} (Theorem \ref{th.unbounded}).

  \begin{remark}
  The proof method of Theorem \ref{th.1} is quite versatile and adapts to other open subsets $\mathscr O$ of $\bar{\mathscr D}$ satisfying appropriate connectivity and reachability properties, see \ass{C4} and \ass{C5} below and the   proofs  of    Propositions \ref{pr.C4} and \ref{pr.C5}.  In addition, when $\mathscr D$ is convex, the regularity assumptions on $\partial\mathscr D$ can be relaxed, as reflected stochastic differential equations are well-posed on convex domains even without boundary smoothness (see~\cite{lionsRSDE,dupuisSDE,marin2004some,saisho,cepa1998problame,pardouxB,Hino}).
  \end{remark}

\rev{\begin{remark}
The boundary behaviour considered here (normal reflection on $\partial\mathscr D\setminus\bar\Sigma$ and absorption on $\bar\Sigma$) is a particular case of the general boundary conditions for multidimensional diffusions studied by Ventcel~\cite{Ventcel1959}, which extend Feller's one-dimensional theory. We do not address these more general boundary conditions here.\end{remark}}

  
  \subsubsection{On the proof of Theorem \ref{th.1}}
  \label{sec.ex1}
Our approach relies on and builds upon the criteria  developed in \cite{guillinqsd},  see \ass{C1}$\to$\ass{C5} below.
In contrast to standard non-reflected stochastic differential equations,  widely treated in the quasi-stationary distribution  literature~\cite{pinsky1985convergence,gong1988killed,le2012mathematical,kolb2012quasilimiting,champagnat2017general,champagnat2018criteria}, our framework must account for the fact that the process \eqref{eq.LR} is reflected over $\partial \mathscr{D}$, which results in an uncountable set of boundary hitting times. This gives rise to several specific technical challenges, which we detail below.

We establish the strong Feller property (this is property \ass{C1})  in two different ways depending on the regularity of $b$: 
\begin{enumerate}
\item[$\mathfrak 1$.] (Proposition~\ref{th.SF}) With a coupling argument when $b$ is Lipschitz (\ass{A1}).  To this end, we must estimate the behavior of the coupling time with respect to the distance between the two initial points, which is achieved using It\=o's formula together with a Girsanov transformation to absorb the coupling drift into the Doléans-Dade exponential. Note that $L_t$ arises directly in It\=o's formula, precise estimates on this boundary term are therefore required throughout the proofs. The argument of \rev{Proposition}~\ref{th.SF} also applies when $\mathscr D$ is convex and possibly unbounded. In that case, \eqref{eq.a} holds with $C_{\partial\mathscr D}=0$  (see Step~1 of the proof of Theorem~\ref{th.unbounded}).
\smallskip

\item[$\mathfrak 2$.] (Proposition~\ref{th.SF2}) When $b$ is merely continuous (\ass{B}), we employ a Girsanov transformation alongside a limiting procedure based on a refined convergence result under uniform integrability from~\cite{liming1999uniqueness}.  The Girsanov approach is particularly effective when information on the reflected Brownian motion is available. In this regard, we use that its transition density is continuous on $\bar{\mathscr D}\times\bar{\mathscr D}$~\cite{bass1991some}, hence bounded for each fixed $t>0$, which yields the required uniform integrability, and that it satisfies the Skorokhod equation~\cite{Hsu2}, which identifies it with the solution of \eqref{eq.LR} when $b\equiv0$. This is where we need $\sigma\equiv \sqrt 2\, \mathrm{Id}_{\mathbb R^d}$.
 
\end{enumerate} 

The Girsanov transformation is also used to prove the weak continuity of the sample paths (namely, property \ass{C2}) when $b$ is merely continuous. In contrast, when $b$ is Lipschitz, direct It\=o calculus together with precise estimates on $L_t$ allow us to bound the supremum   distance (in time) between $(X_t(x))_{t\ge 0}$ and $(X_t(y))_{t\ge 0}$. Of course, when $\sigma\equiv \sqrt 2\, \mathrm{Id}_{\mathbb R^d}$,  the proof for continuous $b$ subsumes the Lipschitz case, however, both proofs possess their own interest and merit separate presentations.  

Assumption \ass{C3} (existence of a Lyapunov function) is automatically satisfied with $W\equiv1$, since the state space $\bar{\mathscr D}$ is compact. To prove the Feller property of the killed semigroup $(P_t^{\mathscr O})_{t\ge0}$ (this is property \ass{C4}), we combine the strong Feller property of $(P_t)_{t\ge0}$ (defined in \eqref{eq:Pt}) with the fact that, the process cannot leave a small ball in a short time. This is established in Proposition~\ref{pr.C4}, which in fact shows that $P_t^{\mathscr O}$ is strongly Feller for every $t>0$. The proof we give uses a test function whose normal derivative vanishes on $\partial\mathscr D$, so that the local-time term disappears from It\=o's formula altogether. 

To prove   irreducibility of the killed process (this is property \ass{C5}), we have to care about the fact that  a trajectory starting on $\partial\mathscr{D}$ is instantaneously and infinitely often pushed back into the interior for arbitrarily small times $t>0$. This is the purpose of Proposition \ref{pr.C5}, which has the advantage of also applying to unbounded domains $\mathscr{D}$.   
To establish that the process starting from $x \in \partial \mathscr{D}$ reaches a given target ball inside $\mathscr{O}$ at a fixed time $t>0$  with positive probability while staying in $\mathscr{O}$, we exploit the fact that it cannot spend significant time on the boundary. Specifically, we use the property that the set $B(x)=\{t \ge 0, X_t(x) \in \partial \mathscr{D}\}$ has Lebesgue measure zero almost surely, alongside the Markov property.
    
We also mention that some of the intermediate results derived along the way are of independent interest.



\subsection{Problem \ref{Pb.2}: Analytical properties of the quasi-stationary distribution and exit position distribution}
 
The aim of this section is to identify the quasi-stationary distribution from Theorem~\ref{th.1} with an eigenfunction of a realization of (the adjoint of) {$-\mathscr L$} subject to mixed boundary conditions, and to study its regularity up to the boundary.  
 
\subsubsection{Main result}

To state our second main result, we introduce the (unbounded) operator $A_2$, defined as the realization of $-\Lr$ (recall that $\Lr$ is the \rev{second-order differential operator associated with \eqref{eq.LR}, see \eqref{eq.Lo}}) on $L^2(\mathscr{D})$ with null Dirichlet boundary conditions on $\SigD:=\bar\Sigma$ and null Neumann boundary conditions on $\SigN:=\partial\mathscr D\setminus\SigD$. To this end, let $\HsD$ denote the closure in $H^1(\mathscr D)$ of $C_c^\infty(\bar{\mathscr D}\setminus \bar{\Sigma})$, i.e.\ of the smooth functions on $\bar{\mathscr D}$ vanishing in a neighborhood of $\bar\Sigma$. Consider now
$$\dom(A_2) =\Big  \{ u   \in \HsD,  \mathscr L u \in L^2(\mathscr D)  \text{ and }  \partial_{\mathbf n_e} u |_{\SigN} = 0\Big  \},\ \text{ $A_2 u = - \mathscr L u$}.$$

By construction (see Section~\ref{sec.H1}), $A_2$ is closed and quasi-sectorial. Moreover, $-A_2$ generates a $C^0$-semigroup on $L^2(\mathscr D)$. We refer to Section~\ref{sec.H1} for a precise definition of and more material on this operator. The second main result of this work is the following.

\begin{theorem}\label{th.2}
Let $\mathscr D\subset\mathbb R^d$ be a $C^\infty$ bounded domain of $\mathbb R^d$, $d\ge2$, and $\Sigma$ be a nonempty $C^3$ open subset of $\partial\mathscr D$ \rev{(i.e. $\Sigma$ is open in $\partial\mathscr D$ and its relative boundary $\partial\Sigma$ in $\partial\mathscr D$, see Section~\ref{sec.notations}, is a $C^3$ submanifold of $\partial\mathscr D$ of dimension $d-2$)}. Assume that $\sigma\equiv\sqrt2\,\mathrm{Id}_{\mathbb R^d}$ and that $b:\bar{\mathscr D}\to\mathbb R^d$ is $C^\infty$. Then, we have the following:
\begin{itemize}
    \item[$\mathfrak 1$.] There exists a triplet $(\lambda,\mathbf u,\mathbf v)$ such that for all $\varepsilon>0$,
    $$\lambda\in\mathbb R,$$
    $$
    \mathbf u\in\dom(A_{2})\cap C^{0,1/2-\varepsilon}(\bar{\mathscr D})\cap C^\infty(\bar{\mathscr D}\setminus\partial\Sigma),
    $$
    $$
    \mathbf v\in\dom(A_{2}^*)\cap C^{0,1/2-\varepsilon}(\bar{\mathscr D})\cap C^\infty(\bar{\mathscr D}\setminus\partial\Sigma),
    $$
    where $\dom(A_2^*)$ is precisely described in Proposition~\ref{prop:domAstar-L2}, and $(\lambda,\mathbf u,\mathbf v)$ is a positive eigentriplet of $A_2$, that is
    $$\mathbf{u},\mathbf{v}>0\text{ in } \mathscr D\quad\text{ and }\quad A_2\mathbf{u} = \lambda\mathbf{u},\ \ A_2^*\mathbf{v} = \lambda\mathbf{v}.
    $$
    
    \item[$\mathfrak 2$.] In addition, $A_2$ and $A_2^*$ are  quasi-sectorial \rev{(roughly speaking, their spectra are contained in a sector, see Proposition~\ref{pr.sectorial} for a precise statement)},  have compact resolvents, and $\lambda$ is dominant: for every eigenvalue $\mu\in\sigma(A_2)\cup\sigma(A_2^*)$,
    $$\mathrm{Re}(\mu)\ge\lambda.$$
    
    \item[$\mathfrak 3$.] Moreover, $\lambda$ is algebraically simple for $A_2$ and $A_2^*$.
    \smallskip
    \item[$\mathfrak 4$.] For some $q>1$,
    $$
    \partial_{\mathbf n_e}\mathbf v \in L^q(\partial\mathscr D).
    $$ 
\end{itemize}

\end{theorem}

Let us emphasize that this result is completely independent of Theorem \ref{th.1} and provides  the crucial regularity we need to make the link between the probabilistic propagator $P_t^\mathscr O$ and the analytic propagator $e^{-tA_2}$.
This is made explicit in the following theorem.

\begin{theorem}\label{th.CorFeyn}
Let the assumptions of Theorem \ref{th.2} hold, we have the following:
\begin{itemize}
    \item[$\mathfrak 1$.] The eigenvalue $\lambda\in\mathbb R$ given by Theorem \ref{th.2} is equal to $\eta>0$ from Theorem \ref{th.1}, and hence is positive.
    \smallskip
    \item[$\mathfrak 2$.] The unique quasi-stationary distribution $\bm{\rho}\in\mathscr P(\mathscr O)$ given by Theorem~\ref{th.1} is absolutely continuous with respect to the Lebesgue measure $dx$, with density proportional to $\mathbf v$:
    $$
    \bm{\rho}(dx) = \frac{\mathbf v(x)}{\int_{\mathscr D}\mathbf v}\,dx.
    $$
    \smallskip
    \item[$\mathfrak 3$.] The right eigenfunction $\bm{\varphi}$ of $P_t^{\mathscr O}$ given in Theorem \ref{th.1} is (up to a normalizing constant) equal to $\mathbf u$. In particular, it enjoys the regularity results of Theorem \ref{th.2}.
    \smallskip
    \item[$\mathfrak 4$.] The law $\mu_{\bm\rho}$ of $X_{\tau_{\mathscr O}}\in\SigD$ under $\mathbb P_{\bm\rho}$ has density $-\partial_{\mathbf n_e}\mathbf v/(\lambda\int_\mathscr D\mathbf v)$ over $\Sigma$.
\end{itemize}
\end{theorem}

Let us also mention that Theorem~\ref{th.2} can be readily extended to the case where $\sigma$ is $C^\infty$ and uniformly elliptic within $\bar{\mathscr D}$. We assume that $\Sigma$ is a relatively open $C^3$ subset of $\partial \mathscr D$, as we will use the results of~\cite{Shamir1968,Savare1997} in the proof of Theorem~\ref{th.2}. Under the assumptions of Theorem \ref{th.2}, the density $\mathbf{v}$ of $\boldsymbol{\rho}$ exhibits higher regularity than $H^1$. Specifically, $\mathbf{v} \in W^{1,q}$ for some $q > 2$  \cite{groger1989aw,Savare1997}. We also emphasize that we assume $\mathscr{D}$ to be $C^\infty$.  This allows us to highlight that the central analytical difficulty in mixed boundary value problems arises not from domain geometry, but from the junction between the Dirichlet and Neumann regions, where derivatives of solutions to Poisson equations inherently lose regularity, even if the data are smooth. \rev{A classical illustration of this phenomenon is provided by the function $u(z) = \mathrm{Re}(\sqrt{z})$ on the upper half-plane $\mathbb{H} = \{z \in \mathbb{C} : \mathrm{Im}(z) > 0\}$. While $u$ is harmonic in $\mathbb{H}$ and smooth up to the boundary away from the origin, it satisfies a mixed boundary condition on $\mathbb{R}$: a homogeneous Neumann condition on $\{x > 0\}$ and a homogeneous Dirichlet condition on $\{x < 0\}$. The transition at the origin causes $\nabla u(x) \sim \vert x\vert ^{-1/2}\notin L^2_{loc}(\mathbb R)$ to blow up, illustrating the loss of derivative regularity at the Dirichlet-Neumann interface.} Yet, this regularity assumption on the domain can be extensively reduced. We do not seek optimality here.

\subsubsection{On the proof of Theorem \ref{th.2}}\label{sec.St2}
  
Recall that $\SigD:= \bar{\Sigma}$ and $\SigN:=\partial\mathscr D\setminus\bar\Sigma$
denote respectively the absorbing (Dirichlet) and reflecting (Neumann)
parts of $\partial\mathscr D$. Since $b\not\equiv0$, the natural
sesquilinear form associated with $-\mathscr L$ on
the space of functions vanishing on $\SigD$ is
non-symmetric. We are therefore outside the scope of self-adjoint spectral theory, and
we instead realize $-\mathscr L$ as a quasi-m-sectorial operator $A_2$ on
$L^2(\mathscr D)$ (Proposition~\ref{pr.domA2}), together with its
adjoint $A_2^*$, which carries the Robin-type  condition $\partial_{\mathbf n_e}v-(b\cdot\mathbf n_e)\,v=0$ on $\SigN$
(Proposition~\ref{prop:domAstar-L2}). 


The core difficulty is that $\SigD$ and $\SigN$ meet along
$\partial\SigD$, where the boundary condition switches type. Even
for the simplest model problem, the gradient of the explicit solution generically blows up at
the interface, however smooth $\SigD,\SigN$ are. This
caps elliptic regularity below $H^2$: only the sharp mixed-problem
estimate of Savaré \cite{Savare1997}, giving $u\in H^{3/2-\varepsilon}(\mathscr D)$
for every $\varepsilon>0$, is available, and it threatens two things we
need later: continuity of eigenfunctions up to $\partial\SigD$, and
integrability of their normal derivative there.

We  show irreducibility of the resolvent of $A_2$ (no non-trivial closed ideal of $L^2(\mathscr D)$ is left invariant) and use  the Krein-Rutman theorem,  the infinite-dimensional
analogue of the Perron-Frobenius theorem for positive matrices,  to deduce 
a principal eigenvalue $\lambda$, with positive
eigenfunction $\mathbf u$ for $A_2$ and positive eigenfunction $\mathbf v$ for the dual
operator $A_2^*$ (Theorem~\ref{th.3}).

Although $A_2$ has a compact resolvent via $H^{3/2-\varepsilon}(\mathscr D)$ and Rellich-Kondrachov, this regularity is insufficient to perform    computations such as integration by parts and It\=o's formula, needed later to prove Theorem \ref{th.CorFeyn}. Since the standard method of \cite[pages 56-57]{Ouhabaz} fails (because $b\not\equiv 0$), we instead apply the heat kernel Gaussian upper bounds of \cite[Theorem 5.4]{OuhabazJOT} to extend the semigroups (and their duals) to $L^p(\mathscr D)$ for $p \in (1,+ \infty)$ (see Proposition~\ref{pr.extensionS}). Taking $p$ sufficiently large then allows us, since $\mathbf u$ and $\mathbf v$ belong to $H^1(\mathscr D)$ and solve homogeneous mixed boundary value problems,  to use Shamir's regularity~\cite{Shamir1968}  results   (although the maximal attainable regularity is Hölder continuous of order $< 1/2$).
  We then apply results from~\cite{OtBr13} on the maximal tangential gradient of solutions to mixed boundary value problems to obtain an $L^q(\partial \mathscr{D})$ bound on $\partial_{\mathbf{n}_e}\mathbf v$ for some $q>1$ (Proposition~\ref{th.Brown}).

\subsubsection{On the proof of Theorem \ref{th.CorFeyn}}\label{sec.St3}

Establishing the continuity of solutions to $e^{-t A_p} f$ up to the boundary $\bar{\mathscr{D}}$ is essential to rigorously justify the Feynman-Kac representation 
$$e^{-t A_p} f(x) = \mathbb{E}_x  [ f(X_t) \mathbf{1}_{t < \tau_{\mathscr{O}}} ]$$
everywhere on $\bar{\mathscr{D}}$ (Proposition~\ref{pr.fin}). We then show, thanks to this formula, that  the normalized eigenfunction $\boldsymbol{\rho} = \mathbf{v} / \int_{\mathscr{D}} \mathbf{v}$ is uniquely identified as the quasi-stationary distribution, and the eigenvalue $\lambda$ is matched with the survival rate $\eta$ from Theorem~\ref{th.1}, which is shown to dominate the spectrum (Proposition~\ref{pr.dominance}). Furthermore, analyzing the regularity of the boundary value $x \mapsto \mathbb{E}_x \left[ h(X_{\tau_{\mathscr{O}}}) \right]$ entails solving an inhomogeneous mixed boundary value problem using the Fredholm alternative. This ultimately allow us to prove  that the exit distribution of $X_{\tau_{\mathscr{O}}}$ under $\boldsymbol{\rho}$ is represented on $\partial \mathscr{D}$ by the density $\theta_{\mathbf{v}} = -\partial_{\mathbf{n}_e} \mathbf{v} / \left( \lambda \int_{\mathscr{D}} \mathbf{v} \right)$ (Proposition~\ref{pr.exitlaw}).

 \section{Existence and uniqueness of a quasi-stationary distribution: proof of Theorem \ref{th.1}}\label{sec.proofth1}
 
 \subsection{Criteria for the existence and uniqueness of a quasi-stationary distribution} 
 \label{sec.Gu}

 In this section we recall the result~\cite[Theorem 2.2]{guillinqsd} we use to prove Theorem \ref{th.1}.  For simplicity of exposition, we restrict our attention to processes with continuous sample paths, as this is the case for the process \eqref{eq.LR} considered here.
More precisely, let $(\mathscr X_t)_{t\ge 0}$ be a time homogeneous Markov process valued in a  Polish  space $\mathscr E$, with continuous sample  paths and satisfying the strong Markov property, which is defined on a filtered probability space.   
For  $x\in \mathscr E$, we define the  (non-killed) semigroup 
\begin{equation}\label{eq:Pt}
    P_tf(x)= \mathbb E_x[f(\mathscr X_t)], \ f\in b\mathcal B( {\mathscr E}).
\end{equation}
Define the following assumptions:

\begin{enumerate}
\item[\textbf{[C1]}\phantomsection\label{ass.C1}] There exists $t_0>0$ such that for each $t\ge t_0$, $P_t$ is strong Feller, i.e. $P_tf$ is continuous on $\mathscr E$ for any $f\in b\mathcal B(\mathscr E)$.

\item[ \textbf{[C2]}\phantomsection\label{ass.C2}]  For all $T>0$, $x\mapsto\mathbb P_{x}[X_{[0,T]}\in \cdot]$  (the law of $X_{[0,T]}:=(\mathscr X_t)_{t\in [0,T]}$) is continuous from $\mathscr E$ to  $\mathscr P(C([0,T], \mathscr E))$. 

\item[ \textbf{[C3]}\phantomsection\label{ass.C3}]  There exist a continuous  function $W : \mathscr E\mapsto[1,+\infty[$, belonging   to the extended domain $D_e(\mathfrak L)$ of the generator of $(\mathscr X_t)_{t\ge 0}$,   two sequences of positive constants $(r_n)_{n\in\mathbb N}$ and $(b_n)_{n\in\mathbb N}$ where $r_n\to +\infty$, and an increasing sequence of compact subsets $(K_n)_{n\in\mathbb N}$ of $\mathscr E$, such that
$$
-\mathfrak L W ({x}) \ge r_n W({x}) - b_n \mathbf 1_{K_n}({x}), \ q.e.
$$ 
\end{enumerate}
For a precise definition of the extended domain and the meaning of \textit{q.e.} (quasi-everywhere), we refer to \cite{guillinqsd,guillinqsd2} and references therein. A detailed treatment is unnecessary for our purposes since our domain $\bar{\mathscr{D}}$ is compact, which guarantees that assumption \ass{C3} holds automatically (see \cite{guillinqsd}). 

Now let $\mathscr O$ be a non-empty open subset of $\mathscr E$. Consider the first exit time of $\mathscr O$, $$
\tau_{\mathscr O}:=\inf\{t\ge0, \mathscr X_t\notin \mathscr O\}.$$  
 For  $x\in \mathscr O$, we define as we did previously the killed (outside $\mathscr O$) semigroup $  P_t^{\mathscr O}f(x)= \mathbb E_x[f(\mathscr X_t) \mathbf 1_{t<\tau_{\mathscr O}}]$, 
 where    $f\in b\mathcal B( {\mathscr O})$.  
Assume the following on the killed process $(\mathscr X_t, 0\le  t< \tau_{\mathscr O})$:
 
\begin{enumerate}
\item[\textbf{[C4]}\phantomsection\label{ass.C4}] For $t\ge 0$, $P_t^{\mathscr O}$ is  Feller, i.e. if $f\in C_b(\mathscr O)$, so is $P_t^{\mathscr O}f$ over $\mathscr O$. 

\item[\textbf{[C5]}\phantomsection\label{ass.C5}] There exists $t_0>0$ such that for all $t\ge t_0$,  ${x}\in {\mathscr O}$, and \rev{$U\subset \mathscr O$} non-empty open,
$
P_t^{\mathscr O}({x},U)>0
$. In addition, there is some $x_0\in \mathscr O$ such that $\mathbb P_{x_0}[\tau_{\mathscr O}< +\infty]>0$.
\end{enumerate}

We refer to \cite[Remark 2.8]{guillinqsd2} for a discussion on these assumptions. 
Note that Definition~\ref{QSD} of a quasi-stationary distribution extends verbatim to this more general framework.
%
We now recall the result of~\cite{guillinqsd} we will use in this work.

 \begin{theorem}\label{thm:main}(\cite[Theorem 2.2]{guillinqsd})
Assume \ass{C1}$\to$\ass{C5}. 
  Then, for all $p>1$,   there exists a unique quasi-stationary distribution  $\bm{\rho}$ for    the process $(\mathscr X_t)_{t\ge 0}$ in $\mathscr P_{W^{1/p}}(\mathscr O)$. Moreover:
\begin{enumerate} 
  \item[$\mathfrak 1$.] 
   For all $t>0$, $P^{\mathscr O}_t:b_{W^{1/p}}\mathcal B(\mathscr O)\to b_{W^{1/p}}\mathcal B(\mathscr O)$ is compact and there exists 
   $$\eta>0$$
      such that $\mathsf r_{sp}(P^{\mathscr O}_t|_{b_{W^{1/p}}\mathcal B(\mathscr O)} ) =e^{-\eta t}$, for all $t>0$.  
   Furthermore,  $\bm{\rho}    P^{\mathscr O}_t =e^{-\eta t}\bm{\rho}$, for all $t\ge 0$, and $\bm{\rho} (U)>0$ for all nonempty open subsets $U$ of $ \mathscr O$. In addition,  there is a unique   function $\bm{\varphi} \in C_{W^{1/p}} ( \mathscr O)$   such that  $\bm{\rho}  (\bm{\varphi} )=1$ and $
      P^{\mathscr O}_t \bm{\varphi}= e^{-\eta  t} \bm{\varphi}  \text{ on } \mathscr O,$ for all $t\ge0$.  
  Moreover, $\bm{\varphi}>0$   everywhere on $\mathscr O$.
 \item[$\mathfrak 2$.] There exist   $c_1>0$ and   $c_2\ge 1$ such that for all  $t>0$ and all   $\nu\in \mathscr P_{W^{1/p}} ( \mathscr O)$:
 $$
 \sup_{A\in\mathcal B(\mathscr O)}\big |  \mathbb P_\nu[\mathscr X_t\in A|t<\tau_{\mathscr O}]-\bm{\rho}   (A)\big |\le c_2 \, e^{-c_1 t} \frac{\nu(W^{1/p} )}{\nu(\bm{\varphi})}.
 $$
  \item[$\mathfrak 3$.] For all $x\in \mathscr O$, $\mathbb P_x[\tau_{\mathscr O}< +\infty]=1$. 
\end{enumerate} 
 \end{theorem}
  \noindent
  
 Let us give some remarks on this theorem. First of all $\bm{\rho}$ and $\eta$ do not depend on $p>1$.  
Moreover, if  $W$ is bounded  over $\mathscr E$, the quasi-stationary distribution $\bm{\rho}$  inside $\mathscr O$ is unique. When $\mathscr E$ is moreover compact, as already explained above, Theorem~\ref{thm:main} is   valid taking $W=\mathbf 1_{\mathscr E}$. Let us also recall that it is well-known that $X_{\tau_{\mathscr O}}$ and $\tau_{\mathscr O}$ are $\mathbb P_{\bm{\rho}}$-independent, and for all $t\ge 0$,
\begin{equation}\label{eq.e-eta}
    \mathbb P_{\bm{\rho}}[t<\tau_{\mathscr O}] = e^{-\eta  t}.
\end{equation}

%
 

 \subsection{Proof of Theorem \ref{th.1}}
 
 Recall that \ass{C3}  is satisfied as the state space $\bar{\mathscr D}$ is  compact. The proof is divided into three steps and consists in showing first that \ass{C1} and \ass{C2} are satisfied  for the process $(X_t)_{t\ge 0}$ (see \eqref{eq.LR}) on the state space $\mathscr E= \bar{\mathscr D}$, and \ass{C4} and \ass{C5} when $\mathscr O= \bar{\mathscr D}\setminus\bar \Sigma$.
  
We  start by recalling the admissible conditions we will need to perform our analysis (see  \eqref{eq.a} and \eqref{eq.b} below) and which hold when $\mathscr D$ is a bounded $C^3$  domain of $\mathbb R^d$~\cite{lionsRSDE}. First of all, there exists $C_{\partial \mathscr D}\ge 0$ such that:  
\begin{equation}\label{eq.a}
\text{For all $z\in\partial \mathscr D$ and all $y\in \bar{\mathscr D}$,
  $(y-z)\cdot \mathbf n_i(z)+C_{\partial \mathscr D}|y-z|^2\ge 0$.}
  \end{equation}
 When $\mathscr D$ is $C^3$,   the vector field $\mathbf n_i$  is $C^2$ in a neighborhood of $\partial \mathscr D$ in $\mathscr D$.  Let $\Theta(z)$ be a   $C^2_b$ extension of $-{\rm dist}(\cdot, \partial \mathscr D)$, so that  $\nabla \Theta(z)= \mathbf n_i(z)$ for all $z \in \partial \mathscr D$. Then, obviously,
      \begin{equation}\label{eq.b}
\text{For all $z\in \partial \mathscr D$, $\nabla \Theta(z) \cdot \mathbf n_i(z) =1$.}
  \end{equation}  

\subsubsection{On Assumptions \ass{C1} and \ass{C2}}

We proceed with a case-by-case analysis for the two settings introduced above (\ass{A1}-\ass{A2} or \ass{B}).

 \begin{proposition}\label{th.SF}
Assume \ass{A1} and \ass{A2}. Let  $\mathscr D$ be a $C^3$ bounded subdomain of $\mathbb R^d$. Then, \rev{for every $x\in \bar{\mathscr D}$ and $t>0$, the law of $X_t(x)$ is absolutely continuous with respect to the Lebesgue measure on $\bar{\mathscr D}$. In particular, $\mathbb P_x[X_t\in\partial\mathscr D]=0$.} In addition, for all $t>0$,  $P_t$ is strongly Feller.
\end{proposition}
\rev{The strong Feller property, proved by a coupling argument, is used to obtain the absolute continuity of the law of $X_t(x)$ for every $t>0$, and not only for a.e. $t>0$.}

 \begin{proof} 
 The proof is divided into several steps. 
 \smallskip
 
\rev{\underline{Step 1}. Absolute continuity for a.e. $t$. Fix $x\in\bar{\mathscr D}$ and $T>0$, and let $\mu(dy\,dt):=\mathbb P_x[X_t\in dy]\,dt$, which is a finite Borel measure on $\Omega_T:=\mathbb R^d\times(0,T)$ carried by $\bar{\mathscr D}\times(0,T)$. Let $A=(a_{ij})_{1\le i,j\le d}:=\frac12\sigma\sigma^*$ on $\bar{\mathscr D}$ (recall \eqref{eq.Lo}), extended by $\mathrm{Id}_{\mathbb R^d}$ outside $\bar{\mathscr D}$. For every nonnegative $\varphi\in C_c^\infty(\Omega_T)$, It\=o's formula and \eqref{eq.L=} give (note that $\varphi$ does not need to vanish near $\partial\mathscr D$)
\begin{align*}
\int_{\Omega_T}\big[\partial_t\varphi+a_{ij}\partial_{x_i}\partial_{x_j}\varphi\big]\,d\mu&=-\mathbb E_x\Big[\int_0^T b\cdot\nabla_x\varphi(X_s,s)\,ds+\int_0^T\nabla_x\varphi(X_s,s)\cdot\mathbf n_i(X_s)\,d|L|_s\Big]\\
&\le \big(T\|b\|_{\infty}+\mathbb E_x\big[|L|_T\big]\big)\sup_{\Omega_T}|\nabla_x\varphi|,
\end{align*}
where $\mathbb E_x[|L|_T]\le 2\|\Theta\|_\infty+T\|\mathscr L\Theta\|_\infty<+\infty$, by It\=o's formula applied to $\Theta(X_t)$ (recall \eqref{eq.b}). By~\cite[Theorem 6.3.1 (i)]{KrylovFP} applied with $\Omega=\mathbb R^d$, the measure $(\det A)^{1/(d+1)}\mu$ is absolutely continuous with respect to $dy\,dt$. Since, by \ass{A2}, $\det A$ is bounded from below by a positive constant on $\bar{\mathscr D}$, so is $\mu$. Hence, for every Borel set $E\subset\bar{\mathscr D}$ of zero Lebesgue measure, $\int_0^T\mathbb P_x[X_t\in E]\,dt=\mu(E\times(0,T))=0$, i.e. $\mathbb P_x[X_t\in E]=0$ for a.e. $t>0$ ($T$ being arbitrary).

\underline{Step 2}. Strong Feller property.}

We now turn to the proof of the strong Feller property of $P_t$ for $t>0$. We   use a coupling argument, which is inspired by the one made   in~\cite{zhang2009exponential} (see also~\cite{zhangM}). The difference with   the computations led in~\cite{zhangM} is that the  boundary terms are not linked with a maximal monotone operator. For that reason, some of our computations differ from those made in~\cite{zhangM} (see especially \eqref{eq.ITO0-}). 
Let $x,y\in \bar{\mathscr D}$ and $\alpha\in (0,1)$. 
Following the first step of the proof of~\cite[Proposition 3.5]{zhangM}, given $X_t = X_t(x)$ solution of \eqref{eq.LR}, we can consider a solution  $(Y_t(y), K_t(y))_{t\ge 0}$ of the following  transformed reflected stochastic differential equation:   
$$
 dY_t=  b(Y_t)dt + |x-y|^\alpha \frac{X_t-Y_t}{|X_t-Y_t|}\mathbf 1_{X_t\neq Y_t} \mathbf 1_{t<\tau}\,dt+ \sigma(Y_t)dB_t + d K_t ,
$$
  where:   
  \begin{enumerate}
  \item[$\mathfrak 1$.] $Y_0=y$, $Y_t\in \bar{\mathscr D}$ for all $t\ge 0$.
    \item[$\mathfrak 2$.] A.s. $(Y_t)_{t\ge 0}$ and $(K_t)_{t\ge 0}$ are continuous, and $(K_t)_{t\ge 0}$ has locally bounded variations.  
  \item[$\mathfrak 3$.]  $K(0)=0$ and for all $t\ge 0$: 
$$
|K|_t= \int_0^t \mathbf 1_{Y_s\in \partial \mathscr D} d|K|_s \text{ and }   K_t= \int_0^t \mathbf n_i(Y_s) d|K|_s,
$$  
  \end{enumerate}
  and where $\tau= \inf\{t\ge 0, X_t=Y_t\}$ is the coupling time.
\rev{Let us emphasize that the process $(X_t,L_t)_{t\ge 0}=(X_t(x),L_t(x))_{t\ge 0}$ is the strong solution of \eqref{eq.LR} on $(\Omega,\mathcal F,(\mathcal F_t)_{t\ge0},\mathbb P)$ driven by $(B_t)_{t\ge 0}$ (see Section~\ref{sec.Lip}), and $(Y_t,K_t)_{t\ge 0}$ is a continuous $(\mathcal F_t)_{t\ge 0}$-adapted solution of the above equation, defined on the same filtered probability space and driven by the same Brownian motion $(B_t)_{t\ge 0}$.}
  Define 
  $$M(z,z')=|x-y|^\alpha\sigma^*(z')(\sigma\sigma^*)^{-1}(z') \frac{z-z'}{|z-z'|}.$$
By Assumptions \ass{A1} and \ass{A2}, $(\sigma\sigma^*)^{-1}(z)$ exists for all $z\in \bar{\mathscr D}$ and is continuous. 
Because $\mathscr D$ is bounded, $\sup_{z\in \bar{\mathscr D}}\Vert (\sigma\sigma^*)^{-1}(z)\Vert_{\rm{HS}}< +\infty$. Hence, there exists $C_\sigma>0$ such that $|M(z,z')|^2\le C_\sigma |x-y|^{2\alpha}$, for all $z,z'\in \bar{\mathscr D}$. 
Note also that $\sigma(z')M(z,z')=|x-y|^\alpha\frac{z-z'}{|z-z'|}$, so that, setting 
$$W_t:= B_t+ \int_0^{t\wedge \tau}M(X_s(x),Y_s(y))\,ds,\qquad t\ge 0,$$
the equation satisfied by $(Y_t)_{t\ge 0}$ rewrites $dY_t=b(Y_t)\,dt+\sigma(Y_t)\,dW_t+dK_t$.
Let $T>0$, consider
$$N_t := \int_0^{t \wedge \tau} M(X_s(x), Y_s(y)) \cdot dB_s, $$
so that $\langle N \rangle_t = \int_0^{t \wedge \tau} \vert M(X_s(x), Y_s(y))\vert ^2 \, ds$. We then define the following Doléans-Dade exponential
$$\mathscr{M}_T := \exp ( -N_T - \frac{1}{2} \langle N \rangle_T ).$$
Since the function $2M$ is bounded, $\mathscr M^\circ_T:=\exp(-2N_T-2\langle N\rangle_T)$ is itself a true martingale with mean $1$, and $\mathscr M_T^2= \mathscr M^\circ_T\,e^{\langle N\rangle_T}$. 
Bounding pointwise $\langle N\rangle_T\le C_\sigma|x-y|^{2\alpha}T$ gives $\mathbb E[\mathscr M_T^2]\le e^{C_\sigma T|x-y|^{2\alpha}}$. 
Therefore, the same computations as those led to get~\cite[Eqs. (11) and (12)]{zhangM} imply that 
  for all $T>0$, there exists $C_T>0$ independent of $x,y$ such that 
\begin{equation}\label{eq.Mt}
\mathbb E[|1-\mathscr M_T|]^2\le C_T|x-y|^{2\alpha}   \text{ and } \mathbb E[(1+\mathscr M_T)\mathbf 1_{T\le \tau}]\le \frac{C_T}{T} \mathbb E[(2T)\wedge \tau].
\end{equation}
By Girsanov's theorem, $(W_t,t\in [0,T])$ is a Brownian motion under $\mathscr M_T\mathbb P$. Since $dY_t=b(Y_t)\,dt+\sigma(Y_t)\,dW_t+dK_t$ and $Y_0=y$, by uniqueness in law of \eqref{eq.LR}, the law of $X_T(y)$ under $\mathbb P$ is the same as the law of $Y_T(y)$ under $\mathscr M_T\mathbb P$. Because moreover $Y_T(y)=X_T(x)$ when $T\ge \tau$, we deduce using also \eqref{eq.Mt} that for all $f\in b\mathcal B(\bar{\mathscr D})$, $x,y\in \bar{\mathscr D}$ and $T>0$, 
 \begin{align}
 \nonumber
 |P_Tf(x)-P_Tf(y)|&= \big|\mathbb E[ f(X_T(x))-\mathscr M_Tf(Y_T(y))]\big|\\
 \nonumber
 &\le |\mathbb E[ (1-\mathscr M_T)f(X_T(x))\mathbf 1_{T\ge \tau}]|+ |\mathbb E[ ( f(X_T(x)) -\mathscr M_Tf(Y_T(y))) \mathbf 1_{  \tau>T}]|\\
 \label{eq.Co=}
 &\le C_T \Vert f\Vert_{\infty}\big (|x-y|^{\alpha} + \frac{1}{T} \mathbb E[(2T)\wedge \tau]\big ).
   \end{align}
The goal is thus to prove that for all $t\ge 0$, $\mathbb E[t\wedge \tau]\to 0$ when $x\to y$.
For $x,y\in \bar{\mathscr D}$, we introduce the function 
 $$e_t(x,y)=e^{- C_{\partial \mathscr D}( \Theta(X_t(x))+\Theta(Y_t(y)))} \  \text{ and }\  S_t(x,y)=e_t(x,y) \, |X_t(x)-Y_t(y)| , $$
 where   $C_{\partial \mathscr D}>0$ and the function $\Theta$ are as in \eqref{eq.a} and \eqref{eq.b}. Since $\Theta\in C^2_b(\bar{\mathscr D})$, one has $\|\Theta\|_\infty< +\infty$, so that
 \begin{equation}\label{eq.e-bound}
 0<c_0:=e^{-2C_{\partial \mathscr D}\|\Theta\|_\infty}\ \le\ e_t(x,y)\ \le\ e^{2C_{\partial \mathscr D}\|\Theta\|_\infty}=:C_0< \infty\qquad\text{for all }t\ge0,\ x,y\in\bar{\mathscr D}.
 \end{equation}
We will use both bounds repeatedly below. In the following we omit to write the dependency of the  solutions in the initial condition $x$ and $y$. Define, for $z\in \mathbb R^d$,  $\bar z= z/ |z|$. 
 Define also $a(Z_t)= |x-y|^\alpha \bar Z_t$, where $Z_t= X_t-Y_t$. Inspired by the computations led in~\cite{lionsRSDE}, performing It\=o formula on $e_t \, (|X_t -Y_t |^2+ \ve)^{1/2}$ and letting $\ve\to 0^+$, one deduces that 
   \begin{align}
 \nonumber
 S_{t\wedge \tau}-S_0  
   &= \int_0^{t\wedge \tau} e_s \,   \bar Z_s\cdot  (\sigma(X_s)- \sigma(Y_s))dB_s\\
  \nonumber
 &\quad +  \int_0^{t\wedge \tau} \frac{e_s}{2 |Z_s|}  \Big (2 Z_s \cdot( b(X_s)-b(Y_s)) + \Vert \sigma (X_s)-\sigma (Y_s)\Vert_{{\rm HS}}^2\, \Big )ds\\
  \nonumber
 &\quad -  \int_0^{t\wedge \tau} e_s \bar Z_s\cdot a(Z_s)ds - \int_0^{t\wedge \tau}   \frac{e_s}{2 |Z_s|}  \vert (\sigma(X_s)-\sigma(Y_s)) \bar Z_s \vert^2ds 
\\
  \nonumber
 &\quad +  \int_0^{t\wedge \tau} e_s   \bar Z_s \cdot \mathbf n_i(X_s)  d|L|_s-\int_0^{t\wedge \tau} e_s   \bar Z_s \cdot  \mathbf n_i(Y_s) d|K|_s\\
  \nonumber
 &\quad - C_{\partial \mathscr D} \int_0^{t\wedge \tau} S_s \,  \big ( \nabla \Theta(X_s )\cdot  \sigma(X_s) +  \nabla \Theta(Y_s )\cdot  \sigma(Y_s)  \big) dB_s\\
  \nonumber
 &\quad - C_{\partial \mathscr D} \int_0^{t\wedge \tau}  S_s \, \Big (  \nabla \Theta(X_s )\cdot  b(X_s) +  \nabla \Theta(Y_s )\cdot  \big(b(Y_s)+a(Z_s)\big)  \Big) ds\\
    \nonumber
 &\quad -  C_{\partial \mathscr D} \int_0^{t\wedge \tau}  S_s \,   \nabla \Theta(X_s )\cdot   \mathbf n_i(X_s) d|L|_s -  C_{\partial \mathscr D}  \int_0^{t\wedge \tau}  S_s   \nabla \Theta(Y_s )\cdot \mathbf n_i(Y_s)  d|K|_s\\
    \nonumber
     &\quad -  \frac{C_{\partial \mathscr D}}{2}\int_0^{t\wedge \tau}  S_s\, \text{Tr}(\sigma\sigma^* {\rm Hess}\, \Theta  \, (X_s) + \sigma\sigma^*{\rm Hess}\, \Theta  \,  (Y_s)  \big )ds  \\
         \nonumber
 &\quad +\frac{C_{\partial \mathscr D}^2}{2 }\int_0^{t\wedge \tau} S_s  \big(  \nabla \Theta(X_s) \sigma(X_s) + \nabla \Theta(Y_s) \sigma(Y_s)\big) \cdot    \big(  \sigma^*(X_s) \nabla \Theta(X_s)  + \sigma^*(Y_s)\nabla \Theta(Y_s)  \big) ds
 \\
 \label{eq.ITO0-}
 &\quad - C_{\partial \mathscr D} \int_0^{t\wedge \tau}  e_s \bar{Z_s}(\sigma(X_s)-\sigma(Y_s))\cdot (\sigma^*(X_s)\nabla \Theta(X_s)+\sigma^*(Y_s)\nabla \Theta(Y_s))\, ds.
 \end{align} 
 Considering the boundary term above involving $d|L|_s$, we have using in \eqref{eq.a} and \eqref{eq.b}, for $s\in [0,t\wedge \tau)$ and when $X_s\in \partial \mathscr D$ (recall \eqref{eq.L=}):
 \begin{align*}
  &  e_s   \bar Z_s \cdot \mathbf n_i(X_s)   -  C_{\partial \mathscr D}    S_s \,   \nabla \Theta(X_s )\cdot   \mathbf n_i(X_s)   \\
  &= - e_s|X_s-Y_s|^{-1}\Big( (Y_s-X_s) \cdot \mathbf n_i(X_s)  +  C_{\partial \mathscr D}  |X_s-Y_s|^2  \Big)\le 0,
 \end{align*} 
 and the same holds for the boundary terms involving $d|K|_s$; these four boundary terms can thus simply be discarded when passing to an upper bound on $S_{t\wedge\tau}$. Note also that, since $\bar Z_s\cdot a(Z_s)=|x-y|^\alpha$,
 $$ \int_0^{t\wedge \tau} e_s \bar Z_s\cdot a(Z_s)ds = |x-y|^\alpha\int_0^{t\wedge \tau} e_s\,ds\ \ge\ c_0\,|x-y|^\alpha (t\wedge \tau),$$
 using $e_s\ge c_0$, see \eqref{eq.e-bound}; hence $-\int_0^{t\wedge\tau}e_s\bar Z_s\cdot a(Z_s)ds\le -c_0|x-y|^\alpha(t\wedge\tau)$.
Recall that $b$ and $\sigma$ are Lipschitz over the bounded domain $\mathscr D$. 
In the following, $C>0$ is an unimportant constant, independent of $x$, $y$, $t\ge 0$ and $\tau$ (it may depend on $c_0,C_0$ and thus on $C_{\partial\mathscr D},\|\Theta\|_\infty$),  which can change from line to line. 
  From \eqref{eq.ITO0-}, discarding the nonpositive boundary terms and using \eqref{eq.e-bound} together with the Lipschitz bounds on $b,\sigma$ and the boundedness of $\Theta,\nabla\Theta,$ and ${\rm Hess}\,\Theta$ on $\bar{\mathscr D}$, we deduce that  there exists a function $F:\bar{\mathscr D}\times \bar{\mathscr D}\to \mathbb R^d$, bounded, such that for all $t\ge 0$:
  \begin{align*} 
 S_{t\wedge \tau} & \le - c_0|x-y|^\alpha (t\wedge \tau) + C\Big( \int_0^{t\wedge \tau} |X_s-Y_s| ds+  |x-y| + \int_0^{t\wedge \tau} F(X_s,Y_s) dB_s\Big)\\
  &\le - c_0|x-y|^\alpha (t\wedge \tau) + C\Big( \int_0^{t} |X_{s\wedge \tau}-Y_{s\wedge \tau}| ds+  |x-y| +  \int_0^{t } \mathbf 1_{s<\tau}F(X_s,Y_s) dB_s\Big),
   \end{align*} 
 where the $|x-y|$ term comes from $S_0=e_0|x-y|\le C_0|x-y|$ (see \eqref{eq.e-bound}), and where $\int_0^t\mathbf 1_{s<\tau}F(X_s,Y_s)\,dB_s$ is a true (mean-zero) martingale because $F$ is bounded and $X_s,Y_s$ live in the compact set $\bar{\mathscr D}$. Using $e_{t\wedge\tau}\ge c_0$ (see \eqref{eq.e-bound}) to bound $|X_{t\wedge\tau}-Y_{t\wedge\tau}|\le \frac{1}{c_0}S_{t\wedge\tau}$, and dropping the nonpositive term $-c_0|x-y|^\alpha(t\wedge\tau)$, and taking expectations, one has for all $t\ge 0$:
\begin{align*}
\mathbb E[|X_{t\wedge \tau}-Y_{t\wedge \tau}|]\le  C\Big( |x-y| +   \int_0^{t} \mathbb E\big[|X_{s\wedge \tau}-Y_{s\wedge \tau}|\big] ds\Big).
\end{align*}
By Grönwall's inequality, we have for all $t\ge 0$: 
$\mathbb E[  |X_{t\wedge \tau}-Y_{t\wedge \tau}|] \le C e^{Ct}\,  |x-y|$. Therefore, because we also have 
   $$c_0\,\mathbb E[|x-y|^\alpha (t\wedge \tau)]\le C\Big(|x-y|+ \mathbb E[|X_{t\wedge \tau}-Y_{t\wedge \tau}|]+  C \mathbb E\big [ \int_0^{t} |X_{s\wedge \tau}-Y_{s\wedge \tau}| ds\big ] \Big),$$
   we deduce that   for all $t\ge 0$, there exists $c=c(t)>0$, such that for all $x,y\in \bar{\mathscr D}$, $\mathbb E[t\wedge \tau]\le c |x-y|^{1-\alpha}$. Coming back to \eqref{eq.Co=}, this shows that for all $T> 0$, there exists $C_T>0$, such that for all $x,y\in \bar{\mathscr D}$ and $\alpha\in (0,1)$,
   $$|P_Tf(x)-P_Tf(y)|\le C_T \Vert f\Vert_{\infty}(|x-y|^{\alpha} +|x-y|^{1-\alpha}).$$
   This concludes the proof of the strong Feller property \ass{C1}.

\rev{\underline{Step 3}. Absolute continuity for every $t$. Let $E\subset\bar{\mathscr D}$ be a Borel set of zero Lebesgue measure and set $g_r:=P_r\mathbf 1_E$, $r>0$. By Step~1, for every $z\in\bar{\mathscr D}$, $g_r(z)=0$ for a.e. $r>0$. The map $(r,z)\mapsto g_r(z)$ being jointly measurable,
Fubini's theorem gives a Lebesgue-null set $N\subset(0,+\infty)$ such that, for every $r\notin N$, $g_r=0$ a.e. on $\mathscr D$. By the strong Feller property (Step~2), $g_r$ is continuous on $\bar{\mathscr D}$, so that, for every $r\notin N$, $g_r=0$ everywhere on $\bar{\mathscr D}$. Let now $t>0$ and $x\in\bar{\mathscr D}$, and choose $r\in(0,t)\setminus N$. By the Markov property, $\mathbb P_x[X_t\in E]=P_{t-r}g_r(x)=0$. This concludes the proof of Proposition~\ref{th.SF}.}

\end{proof}

Note that we have a bit more regularity that asked, $P_t$ actually sends bounded functions to H\"older $1/2$ functions.

 

 \begin{lemma}\label{le.CT}
 Assume  \ass{A1}. Let  $\mathscr D$ be a $C^3$ bounded domain of $\mathbb R^d$.
 For all $T\ge 0$ and $p\ge 1$,  
 there exists $C>0$ such that for all $x,z\in \bar{\mathscr D}$:
 \begin{equation}\label{eq.Ex}
  \mathbb E[\sup_{t\in [0,T]}|X_t(x)-X_t(z)|^{2p}]\le C|x-z|^{2p}.
  \end{equation}
In particular, \ass{C2} is satisfied. 
\end{lemma}
 
\begin{proof}
 Fix $x,z\in\bar{\mathscr D}$ and $p\ge1$, and set $Z_t:=X_t(x)-X_t(z)$. Set
$$
\gamma^\star:=2p\,C_{\partial\mathscr D},\qquad
 e_t^\star(x,z):=e^{-\gamma^\star(\Theta(X_t(x))+\Theta(X_t(z)))},\qquad
S^\star_t:=e_t^\star\,|Z_t|^{2p},
$$
where $\Theta$ and $C_{\partial\mathscr D}$ are as in \eqref{eq.a}-\eqref{eq.b}. 
Set
$$
M_t^\star:=\int_0^t\Xi_s\,dB_s,\qquad \Xi_s:=2p|Z_s|^{2p-1}e_s^\star\,A_s^*\bar Z_s-\gamma^\star S^\star_s\,V_s,
$$
where $A_s:=\sigma(X_s(x))-\sigma(X_s(z))$ and $V_s:=\sigma^*(X_s(x))\nabla\Theta(X_s(x))+\sigma^*(X_s(z))\nabla\Theta(X_s(z))$. 
We have
\begin{equation}\label{eq.e-bound-p}
 0<c^\star_0:=e^{-4pC_{\partial\mathscr D}\|\Theta\|_\infty}\ \le\ e^\star_t(x,z)\ \le\ e^{4pC_{\partial\mathscr D}\|\Theta\|_\infty}=:C^\star_0< \infty,\qquad \forall t\ge0,\ x,z\in\bar{\mathscr D}.
\end{equation}
 Performing It\=o's formula on $e_t^\star \,(|Z_t|^2+\ve)^p$ and letting $\ve\to0^+$, one obtains with very similar computations as those made in the proof of Proposition \ref{th.SF}, 
\begin{equation}\label{eq.Psi-bound}
S^\star_t\ \le\ S^\star_0+C\int_0^tS^\star_s\,ds+M^\star_t, \ \ \forall t\ge 0,
\end{equation}
for some $C>0$ independent of $t\ge 0$, $(x,z)\in \mathscr D^{2}$. 
We have  $|\Xi_s|\le CS^\star_s$, and thus 
$$
\langle M^\star\rangle_t=\int_0^t|\Xi_s|^2\,ds\le C^2\int_0^t(S^\star_s)^2\,ds\le C^2\Big(\sup_{s\le t}S^\star_s\Big)\int_0^tS^\star_s\,ds.
$$
By the Burkholder-Davis-Gundy inequality  and Young's inequality 
$$ab\le \delta a^2+\frac{1}{4\delta}b^2,\ \ \text{ with }  a=(\sup_{s\le t}S^\star_s)^{1/2},\   b=\big(\int_0^tS^\star_s\,ds\big)^{1/2},$$
we deduce that
$$
\mathbb E\Big[\sup_{s\le t}|M^\star_s|\Big]\le C\,\mathbb E\big[\langle M^\star\rangle_t^{1/2}\big]
\le C\,\mathbb E\Big[\Big(\sup_{s\le t}S^\star_s\Big)^{1/2}\Big(\int_0^tS^\star_s\,ds\Big)^{1/2}\Big]
\le \tfrac12\,\mathbb E\Big[\sup_{s\le t}S^\star_s\Big]+C\,\mathbb E\Big[\int_0^tS^\star_s\,ds\Big].
$$
Note that a  standard localization by the stopping times $\tau_n:=\inf\{s\ge0:|Z_s|\ge n\}\wedge T$, followed by taking the limit $n\to+\infty$ and Fatou's lemma, justifies applying this to \eqref{eq.Psi-bound} even before knowing $h(t):=\mathbb E[\sup_{s\le t}S^\star_s]<+\infty$ a priori. Taking $\sup_{s\le t}$ in \eqref{eq.Psi-bound}, then expectations, and absorbing the $\tfrac12 h(t)$ term into the left-hand side, we obtain that:
$$
h(t)\le 2S^\star_0+C\int_0^t h(s)\,ds,\qquad t\in[0,T].
$$
 Grönwall's inequality yields $h(T)=\mathbb E\big[\sup_{t\in[0,T]}S^\star_t\big]\le 2S^\star_0\,e^{CT}$. Since $S^\star_0=e^\star_0|x-z|^{2p}\le C^\star_0|x-z|^{2p}$ by \eqref{eq.e-bound-p}, and since again by \eqref{eq.e-bound-p}, $|Z_t|^{2p}\le S^\star_t/c^\star_0$ for every $t$, we conclude
$$
\mathbb E\Big[\sup_{t\in[0,T]}|X_t(x)-X_t(z)|^{2p}\Big]\ \le\ \frac{1}{c^\star_0}\,\mathbb E\Big[\sup_{t\in[0,T]}\Psi_t\Big]\ \le\ \frac{2C^\star_0}{c^\star_0}\,e^{CT}\,|x-z|^{2p}=C_T\,|x-z|^{2p},
$$
which is exactly \eqref{eq.Ex}, with $C_T$ independent of $x,z\in \bar{\mathscr D}$.

\end{proof}



\begin{proposition}\label{th.SF2}
Assume \ass{B} and let   $\mathscr D$ be a $C^4$ bounded domain of $\mathbb R^d$. Then, for all  $t>0$ and  $x\in \bar{\mathscr D}$, $X_t(x)$ has  a density   $p_t(x,y)$  with respect to the Lebesgue measure $dy$. In addition,  $P_t$ is strongly Feller for $t>0$.
\end{proposition}

\rev{As in Proposition~\ref{th.SF}, the density $p_t(x,\cdot)$ of the law of $X_t(x)$ is obtained for each fixed $(t,x)$ (by the Radon-Nikodym theorem), and no joint measurability of $p_t(x,y)$ with respect to $(x,y)$ or $(t,x,y)$ is claimed nor used. The only transition density for which we use a joint property is the transition density $p^\star_t(x,y)$ of the reflected Brownian motion introduced in Step~1 of the proof below, which is continuous, hence jointly measurable, in $(t,x,y)\in(0,+\infty)\times\bar{\mathscr D}\times\bar{\mathscr D}$ by~\cite[Theorem~3.1 and Lemma~4.3]{bass1991some}.}

 \begin{proof}
The proof is divided into two steps and is based on a Girsanov formula, which will also be used later to prove \ass{C2} under \ass{B} (see the proof of Lemma~\ref{le.CT2}).

\underline{Step 1}. The case $b\equiv 0$: the reflected Brownian motion.

Let $x\in \bar{\mathscr D}$ and let $(Z_t(x), J_t(x))_{t\ge 0}$ be the solution of $dZ_t=\sqrt 2\, dB_t + dJ_t$, i.e.:
\begin{enumerate}
\item[$\mathfrak 1$.] $Z_0=x$, $Z_t\in \bar{\mathscr D}$ for all $t\ge 0$.
\item[$\mathfrak 2$.] A.s. $(Z_t)_{t\ge 0}$ and $(J_t)_{t\ge 0}$ are continuous, and $(J_t)_{t\ge 0}$ has locally bounded variations.
\item[$\mathfrak 3$.] $J_0=0$ and for all $t\ge 0$, $|J|_t= \int_0^t \mathbf 1_{Z_s\in \partial \mathscr D}\, d|J|_s$, $J_t= \int_0^t \mathbf n_i(Z_s)\, d|J|_s$.
\end{enumerate}
Let $(W_t)_{t\ge0}$ be the reflecting Brownian motion on $\bar{\mathscr D}$, i.e. the diffusion associated with the regular Dirichlet form $\mathcal E(u,v)=\int_{\mathscr D}\nabla u\cdot\nabla v\,dy$ on $H^1(\mathscr D)$, and denote by $\mathbb P^W_x$ its law when started at $x$. By \cite[Theorem~3.1 and Lemma~4.3]{bass1991some}, $W$ has a transition density $p^\star_t(x,y)$ with respect to $dy$ which is continuous on $(0,\infty)\times\bar{\mathscr D}\times\bar{\mathscr D}$ (see also~\cite[Section 2]{Rama}). Moreover, by \cite[Theorem~1]{Hsu2}, $W$ is a semimartingale with Skorokhod decomposition
$$
W_t=W_0+\sqrt2\,\beta_t+\int_0^t \mathbf n_i(W_s)\,d\ell_s ,
$$
where $(\beta_t)_{t\geq0}$ is a standard Brownian motion and $(\ell_t)_{t\geq0}$ is a continuous nondecreasing process, with $\ell_0=0$, which increases only on $\{s\ge 0, \ W_s\in\partial\mathscr D\}$. Hence, for a.e. $x\in\bar{\mathscr D}$, under $\mathbb P^W_x$, the pair $(W_t,\int_0^t\mathbf n_i(W_s)\,d\ell_s)_{t\ge0}$ satisfies $\mathfrak 1$-$\mathfrak 3$ (with $|J|=\ell$ and $B=\beta$). By uniqueness in law for \eqref{eq.LR} (see the discussion in Section \ref{sec.weak}), the law of $(Z_t(x))_{t\ge0}$ is the law of $(W_t)_{t\ge 0}$ under $\mathbb P^W_x$. We thus get, for a.e. $x\in\bar{\mathscr D}$, every $t>0$ and every $f\in C(\bar{\mathscr D})$:
\begin{equation}\label{eq.densZ}
\mathbb E[f(Z_t(x))]=\int_{\mathscr D}p^\star_t(x,y)\,f(y)\,dy .
\end{equation}
Both sides of \eqref{eq.densZ} are continuous in $x\in\bar{\mathscr D}$: the left-hand side by \eqref{eq.Ex} (applied with $b\equiv0$) and the uniform continuity of $f$, and the right-hand side by the uniform continuity of $p^\star_t$ on the compact set $\bar{\mathscr D}\times\bar{\mathscr D}$. Hence \eqref{eq.densZ} holds for every $x\in\bar{\mathscr D}$ and every $f\in C(\bar{\mathscr D})$. Since two finite Borel measures on $\bar{\mathscr D}$ which coincide on $C(\bar{\mathscr D})$ are equal, \eqref{eq.densZ} then holds for every $f\in b\mathcal B(\bar{\mathscr D})$. In conclusion, for all $t>0$ and $x\in\bar{\mathscr D}$, $Z_t(x)$ has the density $p^\star_t(x,\cdot)$ with respect to $dy$, and, again by continuity of $p^\star_t$ on the compact set $\bar{\mathscr D}\times\bar{\mathscr D}$,
\begin{equation}\label{eq.bbo}
C(t):=\sup_{x,y\in\bar{\mathscr D}}p^\star_t(x,y)<+\infty, \ \text{ for all } t>0 .
\end{equation}

\underline{Step 2}. Girsanov formula and conclusion.

Fix $x\in\bar{\mathscr D}$ and $T>0$, and set, for $t\in[0,T]$,
$$
B^*_t:= B_t- \frac1{\sqrt 2} \int_0^{t}b(Z_s(x))\,ds,
\quad
m_t(x):=\exp\Big( \frac1{\sqrt 2} \int_0^{t} b(Z_s(x))\cdot dB_s- \frac 14\int_0^{t}|b(Z_s(x))|^2\,ds \Big).
$$
Since $b$ is continuous on the compact set $\bar{\mathscr D}$, it is bounded. Hence Novikov's criterion applies: $(m_t(x))_{t\in[0,T]}$ is a martingale with $\mathbb E[m_T(x)]=1$, and, by Girsanov's theorem, $(B^*_t)_{t\in[0,T]}$ is a Brownian motion under $m_T(x)\,\mathbb P$. Since
$$
Z_t(x)=x+\int_0^t b(Z_s(x))\,ds+\sqrt2\, B^*_t+J_t(x),\quad t\in[0,T],
$$
and since $\mathfrak 1$-$\mathfrak 3$ are pathwise properties (hence preserved under the equivalent probability $m_T(x)\,\mathbb P$), the process $(Z(x),J(x))_{t\ge 0}$ is, under $m_T(x)\,\mathbb P$, a weak solution of \eqref{eq.LR} on $[0,T]$ started at $x$. By weak uniqueness for \eqref{eq.LR} under \ass{B}, the law of $(Z_t(x))_{t\in[0,T]}$ under $m_T(x)\,\mathbb P$ is the law of $(X_t(x))_{t\in[0,T]}$ under $\mathbb P$. Therefore, for every bounded measurable $F:C([0,T],\bar{\mathscr D})\to\mathbb R$,
\begin{equation}\label{eq.Gi}
\mathbb E\big[F(X_{[0,T]}(x))\big]=\mathbb E\big[F(Z_{[0,T]}(x))\,m_T(x)\big].
\end{equation}
In particular,  $X_t(x)$ has a density with respect to $dy$, for all $t>0$ and $x\in\bar{\mathscr D}$ \rev{(indeed, if $A\subset\bar{\mathscr D}$ is a Borel set of zero Lebesgue measure, then $\mathbb P[Z_t(x)\in A]=\int_A p^\star_t(x,y)\,dy=0$ by Step~1, hence $\mathbb P[X_t(x)\in A]=\mathbb E[\mathbf 1_A(Z_t(x))\,m_t(x)]=0$, and the claim follows from the Radon-Nikodym theorem). In particular, $\mathbb P[X_t(x)\in\partial\mathscr D]=0$ for all $t>0$ and $x\in\bar{\mathscr D}$, $\partial\mathscr D$ having zero Lebesgue measure}.

We now prove that $P_t$ is strongly Feller. Fix $t\in(0,T]$, $f\in b\mathcal B(\bar{\mathscr D})$, and let $x_n\to x$ in $\bar{\mathscr D}$.

\underline{Step 2a}  By \eqref{eq.Ex} (with $b\equiv0$), $Z_t(x_n)\to Z_t(x)$ in probability. By Step 1, the law of $Z_t(x_n)$ is $p^\star_t(x_n,y)\,dy$, and by \eqref{eq.bbo} the densities $(p^\star_t(x_n,\cdot))_n$ are bounded by $C(t)$, hence uniformly integrable on $(\mathscr D,dy)$. Then, $f(Z_t(x_n))\to f(Z_t(x))$ in probability follows from \cite[Lemma~3.2]{liming1999uniqueness}.

\underline{Step 2b} $m_t(x_n)\to m_t(x)$ in $L^1$. Write 
$$m_t(x)=\exp\big(\frac1{\sqrt2}I_t(x)-\frac14 \<I\>_t(x)\big),$$ with $I_t(x):=\int_0^t b(Z_s(x))\cdot dB_s$ and $\<I\>_t(x):=\int_0^t|b(Z_s(x))|^2\,ds$. By It\=o's isometry,
$$
\mathbb E\big[|I_t(x_n)-I_t(x)|^2\big]=\int_0^t \mathbb E\big[|b(Z_s(x_n))-b(Z_s(x))|^2\big]\,ds .
$$
For each $s\in[0,t]$, $Z_s(x_n)\to Z_s(x)$ in probability by \eqref{eq.Ex}, so that $b(Z_s(x_n))-b(Z_s(x))\to0$ in probability by continuity of $b$. Since this quantity is bounded by $2\|b\|_\infty$, two applications of the dominated convergence theorem (in $\omega$, then in $s$) give $I_t(x_n)\to I_t(x)$ in $L^2$. Similarly, $$\mathbb E|\<I\>_t(x_n)-\<I\>_t(x)|\le 2\|b\|_\infty\int_0^t\mathbb E|b(Z_s(x_n))-b(Z_s(x))|\,ds\to0.$$ By the continuous mapping theorem, $m_t(x_n)\to m_t(x)$ in probability. Finally, since $m_t(x_n)\ge 0$, we have $0\le (m_t(x)-m_t(x_n))^+\le m_t(x)\in L^1$, and $(m_t(x)-m_t(x_n))^+\to0$ in probability. Hence,   $\mathbb E[(m_t(x)-m_t(x_n))^+]\to0$. Since $|a|=2a^+-a$ for $a\in\mathbb R$ and $\mathbb E[m_t(x_n)]=\mathbb E[m_t(x)]=1$, we deduce
\begin{equation}\label{eq.mL1}
\mathbb E\big[|m_t(x_n)-m_t(x)|\big]=2\,\mathbb E\big[(m_t(x)-m_t(x_n))^+\big]\xrightarrow[n\to\infty]{}0 .
\end{equation}

\underline{Step 2c} Conclusion. By \eqref{eq.Gi} with $T=t$ and $F(\omega)=f(\omega_t)$,
$$
|P_tf(x_n)-P_tf(x)|\le \|f\|_\infty\,\mathbb E\big[|m_t(x_n)-m_t(x)|\big]+\mathbb E\big[\,|f(Z_t(x_n))-f(Z_t(x))|\,m_t(x)\big].
$$
The first term tends to $0$ by \eqref{eq.mL1}. In the second term, the integrand tends to $0$ in probability by Step 2a and is dominated by $2\|f\|_\infty\, m_t(x)\in L^1$, so it tends to $0$. Hence $P_tf\in C(\bar{\mathscr D})$ for every $f\in b\mathcal B(\bar{\mathscr D})$, i.e. $P_t$ is strongly Feller for every $t>0$.

\end{proof}

 \begin{lemma}\label{le.CT2}
 Assume \ass{B} and let $\mathscr D$ be  a $C^4$ bounded domain of $\mathbb R^d$, $d\ge 2$. 
  Then,  \ass{C2} is satisfied. 
\end{lemma}

\begin{proof}
Recall that by Step 2 in the proof of Proposition \ref{th.SF2}, we have  for every bounded measurable $F:C([0,T],\bar{\mathscr D})\to\mathbb R$,
\begin{equation}\label{eq.Gi-path}
\mathbb E\big[F(X_{[0,T]}(x))\big]=\mathbb E\big[F(Z_{[0,T]}(x))\,m_T(x)\big].
\end{equation}
By  Lemma~\ref{le.CT}, applied with $p=1$ to the process $(Z_t)_{t\ge 0}$, gives $C>0$ such that for all $x,x'\in\bar{\mathscr D}$:
$$
\mathbb E\Big[\sup_{t\in[0,T]}|Z_t(x)-Z_t(x')|^2\Big]\le C|x-x'|^2.
$$
In particular, if $x_n\to x$ in $\bar{\mathscr D}$,   $Z_{[0,T]}(x_n)\to Z_{[0,T]}(x)$ in probability in $C([0,T],\bar{\mathscr D})$, and hence, for every $F\in C_b(C([0,T],\bar{\mathscr D}))$,
\begin{equation}\label{eq.step1}
F(Z_{[0,T]}(x_n))\xrightarrow[n\to\infty]{\mathbb P}F(Z_{[0,T]}(x)).
\end{equation}
 On the other hand, recall that by Step 2 in the proof of Proposition  \ref{th.SF2}, 
\begin{equation}\label{eq.step3}
m_T(x_n)\xrightarrow[n\to\infty]{L^1} m_T(x).
\end{equation}
We now conclude the proof of \ass{C2} when \ass{B} holds. 
Fix $F\in C_b(C([0,T],\bar{\mathscr D}))$. By \eqref{eq.Gi-path},
\begin{align*}
\mathbb E_{x_n}\big[F(X_{[0,T]})\big]-\mathbb E_x\big[F(X_{[0,T]})\big]
&=\underbrace{\mathbb E\big[F(Z_{[0,T]}(x_n))\big(m_T(x_n)-m_T(x)\big)\big]}_{\alpha_n}\\
&\quad +\underbrace{\mathbb E\big[\big(F(Z_{[0,T]}(x_n))-F(Z_{[0,T]}(x))\big)m_T(x)\big]}_{\beta_n}.
\end{align*}
We have  $|\alpha_n|\le \|F\|_\infty\,\mathbb E\big[|m_T(x_n)-m_T(x)|\big]\to 0$ by \eqref{eq.step3}. Let us now deal with $\beta_n$. Set $Y_n:=F(Z_{[0,T]}(x_n))-F(Z_{[0,T]}(x))$, so that $|Y_n|\le 2\|F\|_\infty$ and $Y_n\to 0$ in probability by \eqref{eq.step1}. Being bounded and converging in probability, $Y_n\to0$ in $L^1$ as well. Let $\varepsilon>0$. Since $m_T(x)\in L^1(\mathbb P)$, there exists $K_\varepsilon>0$ such that $\mathbb E [m_T(x)\mathbf 1_{m_T(x)>K_\varepsilon} ]<\varepsilon$. Then, we have for all $n\ge 0$:
$$
|\mathbb E[Y_n\,m_T(x)]|\le K_\varepsilon\,\mathbb E[|Y_n|]+2\|F\|_\infty\,\mathbb E\big[m_T(x)\mathbf 1_{m_T(x)>K_\varepsilon}\big]\le K_\varepsilon\,\mathbb E[|Y_n|]+2\|F\|_\infty\,\varepsilon.
$$
Now let $ N_\varepsilon\ge 0$ such that  for all $n\ge  N_\varepsilon$, $\mathbb E[|Y_n|]\le \varepsilon /(1+K_\varepsilon)$. Then, for all $n\ge  N_\varepsilon$, $|\beta_n|\le (1+2\|F\|_\infty)\varepsilon$. Hence,  $\beta_n\to 0$ as $n\to +\infty$.

In conclusion, we have proved that   $\mathbb P_{x_n}[X_{[0,T]}\in\cdot\,]\to \mathbb P_x[X_{[0,T]}\in\cdot\,]$ weakly in $\mathscr P(C([0,T],\bar{\mathscr D}))$. 
This proves \ass{C2}.

\end{proof}

%
%


\subsubsection{On Assumption \ass{C4}} 
\begin{proposition}\label{pr.C4}
Assume  \ass{A1}-\ass{A2} and  let  $\mathscr D$ be a $C^3$ bounded domain of $\mathbb R^d$ (or assume \ass{B} and let $\mathscr D$ be  a $C^4$ bounded domain of $\mathbb R^d$).
Then, for all $t\ge 0$, $P_t^{\mathscr O}$ is (strongly) Feller (i.e. \ass{C4} is satisfied). 
\end{proposition}

Note that the following proof does not in fact require that $\mathscr D$ is bounded. 

\begin{proof} We place ourselves in one of the two settings of Proposition \ref{pr.C4}. 
We claim that \ass{C4} is satisfied if for all $x\in \mathscr O$, $0<\delta_x\le {\rm dist}(x, \Sigma)/10$, and every sequence $(x_n)_{n\in\mathbb N}\subset   \mathscr B(x,\delta_x)\cap \mathscr O$ such that $x_n\to x \in \mathscr O$, we have:
\begin{equation}\label{eq.Ea2autre}
 \lim_{t\to 0^+} \sup_{n\ge 0} \mathbb P_{x_n}[\tau_{\mathscr B(x_n,3\delta_x)}\le t]  =0.
 \end{equation}
 \underline{Step 1}. Let us prove that \eqref{eq.Ea2autre} implies \ass{C4}.
 
 The proof is inspired by the one made to prove~\cite[Theorem 2.2]{chung2001brownian}. Let $x_n\to x\in \mathscr O$, $x_n\in \mathscr O$.  Let $0<\delta_x\le {\rm dist}(x, \Sigma)/10$. We can assume that $(x_n)_n\subset   \mathscr B(x,\delta_x)\cap \mathscr O$.  
Using the Markov property of the process $(X_t)_{t\ge 0}$, one has for $0\le s\le t$,   $f\in b\mathcal B(\mathscr O)$, 
$$P_t^{\mathscr O}f( x_n)= \mathbb E_{x_n}[\mathbf 1_{s<\tau_{\mathscr O}} g_s(X_s)],$$
 where $g_s(z)=\mathbb E_{z}[\mathbf 1_{t-s<\tau_{\mathscr O}} f(X_{t-s})]$. 
  In addition, we have using \eqref{eq.Ea2autre} that 
\begin{align*}
\sup_{n\ge 0}| P_t^{\mathscr O}f(x_n)-P_sg_s(x_n)|\le \Vert f\Vert_{\infty}\sup_{n\ge 0} \mathbb P_{x_n}[\tau_{\mathscr O}\le s]\le  \Vert f\Vert_{\infty}\sup_{n\ge 0} \mathbb P_{x_n}[\tau_{\mathscr B(x_n,3\delta_x)}\le s]\to 0,
\end{align*}
as $s\to 0^+$. 
Note that for the same reason,  we also have $| P_t^{\mathscr O}f(x)-P_sg_s(x)|\to 0$ as $s\to 0^+$. 
By \ass{C1} proved before, the function $P_sg_s$ is continuous over $\bar{\mathscr D}$. Writing 
$$| P_t^{\mathscr O}f(x_n)-P_t^{\mathscr O}f(x)|\le | P_t^{\mathscr O}f(x_n)-P_sg_s(x_n)|+ |P_sg_s(x_n)-P_sg_s(x)|+ | P_t^{\mathscr O}f(x)-P_sg_s(x)|,$$
we immediately deduce that $| P_t^{\mathscr O}f(x_n)-P_t^{\mathscr O}f(x)|\to 0$ as $n\to +\infty$. In conclusion,  \eqref{eq.Ea2autre} implies \ass{C4}.  

\underline{Step 2}. We now prove \eqref{eq.Ea2autre}.

We consider the case when $x\in \partial \mathscr D$ (the case when $x\in \mathscr D$ is treated with the same arguments).
Since $\mathscr D$ is $C^3$, using a tangential-normal system of coordinates $(x',x_d)$ in a neighborhood of $\partial \mathscr D$ in $\bar{\mathscr D}$ which flatten out $\partial \mathscr D$, we can construct a $C^2$ function $\Psi:\bar{\mathscr D}\to [0,1]$ such that $\Psi=0$ on $\mathscr B(x,\delta_x)\cap \mathscr O$, $\Psi=1$ on $\mathscr B (x,2\delta_x)^c\cap \mathscr O$, and
$$\nabla \Psi(z)\cdot \mathbf n_i(z)=0 \text{ for all } z\in \partial \mathscr D.$$
Note that for all $n\ge 0$, $\Psi(x_n)=0$ since $x_n\in \mathscr B(x,\delta_x)\cap \mathscr O$.
Hence, using It\=o's formula and the optional stopping theorem, one has:
\begin{align*}
\mathbb E_{x_n}[\Psi (X_{t\wedge\tau_{\mathscr B(x_n,3\delta_x)}})]& \le \mathbb E_{x_n}\Big[\int_0^{t\wedge \tau_{\mathscr B(x_n,3\delta_x)}} |\mathscr L \Psi (X_{s})| ds\Big] \le t \sup_{z\in \bar{\mathscr D}}|\mathscr L\Psi(z)|.
\end{align*}
In addition, $|X_{ \tau_{\mathscr B(x_n,3\delta_x)}}(x_n)-x_n|= 3\delta_x$. Since $|x_n-x|\le \delta_x$, one has that $|X_{ \tau_{\mathscr B(x_n,3\delta_x)}}(x_n)-x|\ge 2\delta_x$ and thus $\Psi(X_{ \tau_{\mathscr B(x_n,3\delta_x)}}(x_n))=1$ for all $n\ge 0$. We deduce that:
$$ \mathbb P_{x_n}[\tau_{\mathscr B(x_n,3\delta_x)}\le t]= \mathbb E_{x_n}[\mathbf 1_{\tau_{\mathscr B(x_n,3\delta_x)}\le t} \Psi(X_{ \tau_{\mathscr B(x_n,3\delta_x)}})] \le \mathbb E_{x_n}[\Psi (X_{t\wedge\tau_{\mathscr B(x_n,3\delta_x)}})]\le t \sup_{z\in \bar{\mathscr D}}|\mathscr L\Psi(z)|. $$
This ends the proof of \eqref{eq.Ea2autre}.

\end{proof}

\subsubsection{On Assumption \ass{C5}} 

\begin{proposition}\label{pr.C5}
 Assume  \ass{A1}-\ass{A2} and  let  $\mathscr D$ be a $C^3$ bounded domain of $\mathbb R^d$ (or assume \ass{B} and let $\mathscr D$ be  a $C^4$ bounded domain of $\mathbb R^d$).
Then, it holds:
 \begin{enumerate}
 \item[$\mathfrak 1$.]  For all $t>0$, $x\in \mathscr O$, and $U\subset\mathscr O$ nonempty open,
 $\mathbb P_x[X_t\in U, t<\tau_{\mathscr O}]>0$.  
  \item[$\mathfrak 2$.] If  \ass{C} holds, then   there exists $x\in \mathscr O$ such that $\mathbb P_x[\tau_{\mathscr O}< +\infty]>0$. 
   \end{enumerate}
   \end{proposition}
    Proposition \ref{pr.C5}   implies \ass{C5}. 
Let us mention that the proof of Proposition \ref{pr.C5}  does not use the fact that $\mathscr D$ is bounded.    

   \begin{proof}  
  The proof of item $\mathfrak 2$ is immediate. Let us prove item $\mathfrak 1$.    
  Let  $t>0$ and   $U$ be an open subset of $\mathscr O$ where we recall that $\mathscr O=\bar{\mathscr D} \setminus \bar \Sigma$. Then,  there exists $z\in U\cap \mathscr D$ and we can thus consider $r>0$ such that  $\mathscr B(z,r)\subset U\cap\mathscr D$. The proof of item $\mathfrak 1$ is divided into two steps, depending whether $x$ lies inside or on the boundary of $\mathscr D$.
 \medskip

 \underline{Step 1}. Let us first consider the case when  $x\in \mathscr D$. 
 
 Set $\tau_{\mathscr D} = \inf\{t\ge 0, X_t\notin \mathscr D\}$. 
 Let  $(Y_t)_{t\ge 0}$ be  the solution  over $ \mathbb R^d$ of the equation $
 dY_t= \sigma(Y_t)dB_t + b(Y_t)dt$,    
 where we  have extended  the coefficients  outside $\bar{\mathscr D}$ such that:
 \begin{enumerate}
 \item[$i)$] Under \ass{A1}-\ass{A2}, $\sup_{x\in \mathbb R^d}(|b(x)|+ \Vert \sigma (x)\Vert_{\rm{HS}}  )<+\infty $ for all $x\in \mathbb R^d$,  $b,\sigma$ are Lipschitz over $\mathbb R^d$, and $\sigma$ is uniformly elliptic over $\mathbb R^d$. In this case, there is a unique strong solution to the martingale problem
 
  \item[$ii)$] Under \ass{B}, $\sup_{x\in \mathbb R^d}|b(x)|  <+\infty $  and $b$ is continuous over $\mathbb R^d$. In that case, it is well-known that there is a unique weak solution to  $ dY_t= \sigma(Y_t)dB_t + b(Y_t)dt$ (as there is a unique solution to the associated martingale problem~\cite[Chapter 18]{kallenberg1997foundations}). 
 \end{enumerate}
Observe that $L_t(x)=0$ for all $t<\tau_{\mathscr D}$, hence $\tau_\mathscr D = \inf\{t\ge 0, Y_t\notin \mathscr D\}$. Moreover, by uniqueness of the stopped martingale problem~\cite[Chapter 4, Section 5]{EK},  we have for all $x\in \mathscr D$: 
$$(X_{t\wedge \tau_{\mathscr D}}(x))_{t\ge 0} =(Y_{t\wedge \tau_{\mathscr D}}(x))_{t\ge 0}\text{ in law}.$$   

 In both cases $i)$ and $ii)$, it is straightforward to prove that the  Girsanov formula holds  for the law of the trajectories $(Y_{[0,T]}(x))_{x\in \mathbb R^d}$ and we then easily deduce, using the knowledge of the law of the trajectories of a Brownian motion on $\mathbb R^d$   that   for all $y_0,y_1\in \mathscr D$,     $t>0$, and every $r_1>0$:
 $$ \mathbb P_{y_0}[ Y_t \in  \mathscr B(y_1,r_1), t<\tau_{\mathscr D}]>0.$$
 Hence, one deduces that
  $$ \mathbb P_x[ X_t \in U, t<\tau_{\mathscr D}]\ge   \mathbb P_x[ X_t \in \mathscr B(z,r), t<\tau_{\mathscr D}]=\mathbb P_x[ Y_t \in \mathscr B(z,r), t<\tau_{\mathscr D}]>0.$$
We have thus proved the desired result:  for all  $t>0$ and  every open  subset  $U$  of $\mathscr O=\bar{\mathscr D} \setminus \bar \Sigma$,
\begin{equation}\label{eq.st1}
\mathbb P_x[X_t\in U, t<\tau_{\mathscr O}]>0, \ \forall x\in  {\mathscr D}  .
\end{equation}

 \underline{Step 2}. 
 Let us now  consider the case when  $x\in \partial \mathscr D\setminus \bar{\Sigma}$. 
 
 We will first show that the process moves away from the boundary  without touching  $\bar \Sigma$. More precisely, let us prove that   for all $T_F>0$, there exists  $T=T_x\in [0,T_F]$ such that: 
\begin{equation}\label{eq.O-}
\mathbb P_x[X_T\in \mathscr D, T<\tau_{\mathscr O}]>0.
\end{equation}  
Assume that \eqref{eq.O-} does not hold. Then, there  exists $T_F>0$ such that  for all $t\in [0,T_F]$, $\mathbb P_x[X_t\in \partial \mathscr D \text{ or }  \tau_{\mathscr O}\le t ]=1$. Hence, one has:
$$  \mathbb P_x\Big [    \bigcap_{t \in [0,T_F]\cap \mathbb Q}\big \{X_t\in \partial \mathscr D \text{ or }  \tau_{\mathscr O}\le t \}\Big ]=1.$$
 This rewrites
 $$  \mathbb P_x\Big [    \forall t \in [0,T_F]\cap \mathbb Q,  \, X_t\in \partial \mathscr D \text{ or }  \tau_{\mathscr O}\le t  \Big ]=1,$$
 i.e. there exists $\mathbf \Omega_x\subset \mathbf  \Omega$ with $\mathbb P(\mathbf \Omega_x)=1$ such that for all $\omega\in \mathbf \Omega_x$ and all $t \in [0,T_F]\cap \mathbb Q$, either $X_t(\omega)\in \partial \mathscr D$ or $\{X_u(\omega), u\in [0,t]\}\cap \bar \Sigma \neq \emptyset$.
 In what follows, $\mathbf \Omega_x$  denotes a set of probability $1$ which can change from one occurrence to another.  
When $X_0=x$, we have that  a.s. $\sup_{s\in [0,u]} |X_s-x|\to 0$ as 
$u\to 0^+$. Therefore, since $x\in \partial \mathscr D\setminus \bar \Sigma$, we deduce that   when $X_0=x$,   for all $\omega\in \mathbf \Omega_x $,  there exists  $\varepsilon(\omega)\in (0,T_F)$, such that $\{X_u(\omega), u\in [0,\varepsilon(w)]\}\subset \bar{ \mathscr D}\setminus \bar\Sigma$. Hence, for all $  \omega\in \mathbf \Omega_x $  and all $t\in [0,\varepsilon(\omega)]\cap \mathbb Q$, 
$X_t(\omega)\in \partial \mathscr D$. By continuity of the trajectories of the process $(X_t)_{t\ge 0}$ and because $\partial \mathscr D$ is a closed set, we deduce that for all $  \omega\in \mathbf \Omega_x $, it holds $$\{X_u(\omega), u\in [0,\varepsilon(w)]\} \subset \partial \mathscr D,$$
that is
\begin{equation}\label{eq.intB}
    [0,\varepsilon(w)] \subset B(x) :=\{s\ge 0, X_s(x)\in \partial \mathscr D\}.
\end{equation}
However, it is also well-known that, for all $x\in \bar{\mathscr D}$,  the set $B(x)$  is  a.s. a closed set of zero Lebesgue measure without isolated point (i.e.  a zero Lebesgue measure perfect set). \rev{This is standard: it is a consequence of Propositions~\ref{th.SF} and \ref{th.SF2}, which give $\mathbb P_x[X_u\in\partial\mathscr D]=0$ for all $u>0$, together with Fubini's theorem. It can also be proved using the occupation times formula~\cite[Chapter VI, Corollary 1.6]{RevuzYor}.} Therefore, \eqref{eq.intB} is absurd. This achieves the  proof of \eqref{eq.O-}. 

We now end   the proof of item $\mathfrak 1$. 
Let  $t>0$. Consider  $T\in [0,t/2]$   as in \eqref{eq.O-} (for $T_F=t/2$) and set $t=T+u$, $u>0$. 
By the  Markov property of the process $(X_t(x))_{t\ge 0}$, we have 
\begin{align*}
 \mathbb P_x\big [X_t\in \mathscr B(z,r), t<\tau_{\mathscr O}\big ] &=\mathbb E_x\Big[\mathbf 1_{T<\tau_{\mathscr O} }\mathbb P_{X_T }\big [X_{u}\in \mathscr B(z,r), u<\tau_{\mathscr O}\big ] \Big]\\
&\ge \mathbb E_x\Big[ \mathbf 1_{T<\tau_{\mathscr O}, X_T\in \mathscr D }  \,  \mathbb P_{X_T }\big [X_{u}\in \mathscr B(z,r), u<\tau_{\mathscr O}\big ]  \Big].
\end{align*}
Assume that  the last term in the previous inequality vanishes. Then,  a.s. it holds:
 $$ \mathbf 1_{T<\tau_{\mathscr O}, X_T\in \mathscr D }  \,  \mathbb P_{X_T }\big [X_{u}\in \mathscr B(z,r), u<\tau_{\mathscr O}\big ]  =0.$$
On the other hand, by \eqref{eq.O-}, there exists $\mathbf \Omega_{x,T}\subset \mathbf  \Omega$, with $\mathbb P_x[\mathbf \Omega_{x,T}]>0$ such that  for all $  \omega \in \mathbf \Omega_{x,T}$, $T<\tau_{\mathscr O}(\omega)$ and $X_T(\omega)\in \mathscr D$, and therefore, it holds:
 $$  \mathbb P_{X_T(\omega) }\big [X_{u}\in \mathscr B(z,r), u<\tau_{\mathscr O}\big ]  =0, \ \forall   \omega \in \mathbf \Omega_{x,T}.$$
Since $X_T(\omega)\in \mathscr D$, this contradicts \eqref{eq.st1}. Hence,  $\mathbb P_x\big [X_t\in \mathscr B(z,r), t<\tau_{\mathscr O}\big ]>0$, which is the desired result.  The proof of item $\mathfrak 1$ is complete. 

\end{proof}
      
\begin{remark}\label{re.sigma}
Under the assumptions of Theorem \ref{th.1}, when $\Sigma$  has  non empty interior (i.e. for some $x_0\in \Sigma$ and $r>0$, $\mathscr B(x_0,r)\cap \partial \mathscr D\subset \Sigma$), Assumption \ass{C} holds. Indeed, let  $x\in \mathscr D$. Recall that in this case, $(X_{t\wedge \tau_{\mathscr D}}(x))_{t\ge 0} =(Y_{t\wedge \tau_{\mathscr D}}(x))_{t\ge 0}$ in law (thus, $ (\tau_{\mathscr D},\,X_{\tau_{\mathscr D}} )=(\tau_{\mathscr D},\,Y_{\tau_{\mathscr D}})$ in law\rev{, see the end of this remark}, and  in particular $\mathbb P_x[X_{\tau_{\mathscr D}}\in\Sigma]=\mathbb P_x[Y_{\tau_{\mathscr D}}\in\Sigma]$). Moreover, because $\Sigma$  has  non empty interior, it follows that the process $(Y_t(x),t\ge 0)$ has non zero probability to hit $\partial \mathscr D$ at $\Sigma$ (this can be easily  obtained using the  knowledge of the  support of the trajectories of the Brownian motion), i.e. $\mathbb P_x[  Y_{\tau_{\mathscr D}} \in \Sigma]>0$.  This proves  that \ass{C} is satisfied.
\rev{Let us justify the identity in law $(\tau_{\mathscr D},X_{\tau_{\mathscr D}})=(\tau_{\mathscr D},Y_{\tau_{\mathscr D}})$ used above. For a  continuous path $\omega:[0,+\infty)\to\bar{\mathscr D}$, define 
$$\tau_{\mathscr D}(\omega):=\inf\{t\ge0:\omega(t)\in\partial\mathscr D\},\qquad \Phi(\omega):=(\tau_{\mathscr D}(\omega),\,\omega(\tau_{\mathscr D}(\omega)))\ \text{ on } \{\tau_{\mathscr D}(\omega)<+\infty\}.$$
Since $\partial\mathscr D$ is closed, $\tau_{\mathscr D}$ is measurable on $C([0,+\infty),\bar{\mathscr D})$, and so is $\Phi$ on $\{\tau_{\mathscr D}<+\infty\}$. Since both processes start at $x\in\mathscr D$, stopping them at $\tau_{\mathscr D}$ does not change their first hitting time of $\partial\mathscr D$, so that $(\tau_{\mathscr D},X_{\tau_{\mathscr D}})=\Phi(X_{\cdot\wedge\tau_{\mathscr D}})$ and $(\tau_{\mathscr D},Y_{\tau_{\mathscr D}})=\Phi(Y_{\cdot\wedge\tau_{\mathscr D}})$. The two stopped processes having the same law, so do their images under $\Phi$.}      
       \end{remark}
      
 
 \noindent
The proof of Theorem \ref{th.1} is complete.

\section{Analytical properties of the \rev{operator $\mathscr L$ associated with} \eqref{eq.LR}: proof of Theorem \ref{th.2}}


\subsection{Realization of {$-\mathscr L$} with mixed boundary conditions} 
\label{sec.H1}


\subsubsection{The form $\mathbf L$ associated with {$-\mathscr L$} and  mixed boundary conditions}
We assume throughout this section that  $\mathscr D \subset \mathbb{R}^d$ is a $C^\infty$ bounded domain of $\mathbb R^d$, $d\ge 2$, and that $\Sigma$ is a nonempty $C^3$  open subset of $\partial \mathscr D$. 
In all this section:
\begin{enumerate}
\item[$\mathfrak 1$.]   $\SigD := \bar{\Sigma}$ (playing the role of the absorbing part $\bar\Sigma$ of the boundary considered in Section~\ref{sec.NE}). Call $\SigD$ the Dirichlet part.
 \item[$\mathfrak 2$.]  $\SigN := \partial\mathscr D \setminus \SigD$
(the Neumann part). Note that  $\SigN$ is open.  
 \item[$\mathfrak 3$.] $\sigma \equiv \sqrt 2\, \text{Id}_{\mathbb R^d}$ and $b:\mathbb R^d\to \mathbb R^d$ is $C^\infty$. 
 \item [$\mathfrak 4$.]  $\mathbf n_e=-\mathbf n_i$ is the outward-pointing unit normal to $\mathscr D$.
\end{enumerate}

Define
$$
  C^\infty_{\SigD}(\mathscr D)
  := \bigl\{ w|_\mathscr D, \  w \in C^\infty_c(\mathbb{R}^d) \text{ and }  \supp w \cap \SigD = \emptyset \bigr\}.
$$
We denote by $H^1_{0,\SigD}(\mathscr D)$ the closure of $C^\infty_{\SigD}(\mathscr D)$
in     $H^1(\mathscr D)$.  
We recall that when $\sigma \equiv \sqrt 2\, \text{Id}_{\mathbb R^d}$, the operator defined in \eqref{eq.Lo} {(\rev{the differential operator associated with} \eqref{eq.LR})} writes:
$$
  {\mathscr L \phi = \Delta \phi + b \cdot \nabla \phi.}
$$
{Since the principal part of $-\mathscr L$ is the nonnegative operator $-\Delta$, we work in this section with}
$$
  {-\mathscr L \phi = -\Delta \phi - b \cdot \nabla \phi .}
$$
{Its formal $L^2(\mathscr D)$-adjoint is $-\mathscr L^*$, where
$\mathscr L^* \phi := \Delta \phi - \operatorname{div}(b\,\phi)= \Delta\phi - b\cdot\nabla\phi-(\operatorname{div}b)\,\phi$.}

Define the sesquilinear form $\mathbf{L} : \HsD \times \HsD \to \mathbb{C}$ (associated with {$-\mathscr L$}) by
\begin{equation}\label{eq:form}
  \mathbf{L}(u,v)
  := \int_\mathscr D (\nabla u \cdot \overline{\nabla v}
  {-} (b \cdot \nabla u)\,\bar v).
\end{equation}
Moreover, for any $\kappa \geq   \|b\|_\infty^2/2$, the shifted  form
$\mathbf{L}_{\kappa}(u,v) := \mathbf{L}(u,v) + \kappa(u,v)_{L^2(\mathscr D)}$
is accretive on $\HsD$ (see \cite[Definition 1.4]{Ouhabaz}):
\begin{equation}\label{eq:coercive}
  \Re\, \mathbf{L}_{\kappa}(u,u) \geq \tfrac{1}{2}\|\nabla u\|^2_{L^2}
  \qquad \forall\, u \in \HsD.
\end{equation}
It is also continuous, closed and sectorial on $L^2(\mathscr D)$.
Let $A_2$ denote the operator on $L^2(\mathscr D)$ associated with $\mathbf{L}$, i.e. the unique closed operator of maximal domain satisfying \cite[Section 1.2.3]{Ouhabaz}:
\begin{equation}\label{eq.LA2}
      \mathbf{L}(u,v) = \langle A_2u, v\rangle _{L^2(\mathscr D)}
      \quad \forall\, v \in \HsD, \qquad u \in \dom(A_2) \subset \HsD.
\end{equation} 

\subsubsection{The operator $A_2$ and its adjoint}
\label{sec.A12}



Note that since $b\in W^{1,\infty}(\mathscr D)^d$, for all $u\in H^1(\mathscr D)$, we have the equivalences
\begin{equation}\label{eq.equivalence}
    \Lr u \in L^2(\mathscr D) \iff \Lr^* u \in L^2(\mathscr D) \iff \Delta u \in L^2(\mathscr D).
\end{equation}

 \begin{proposition}\label{pr.domA2}
 The operator $A_2$  is   quasi-m-accretive and quasi-m-sectorial,  and $-A_2$     is the generator of a holomorphic $C_0$-semigroup  $S_2=(e^{-tA_2})_{t\ge 0}$ on $L^2(\mathscr D)$.  In addition, 
 $\|e^{-tA_2}\|_{\mathcal L(L^2)}\le  e^{\frac12\|b\|_\infty^2t}$.
Moreover, 
$$
  \dom(A_2)
  = \bigl\{\, u \in \HsD, \  \mathscr L u \in L^2(\mathscr D) \text{ and }  
    \partial_{\mathbf n_e} u |_{\SigN} = 0 \,\bigr\},
$$
where $\mathscr L u \in L^2(\mathscr D)$ is understood in the sense of distributions and
$\partial_{\mathbf n_e} u|_{\SigN} = 0$ in $H^{-1/2}(\SigN)$.
Moreover, $A_2u = {-}\mathscr L u$ for every $u \in \dom(A_2)$. 
\end{proposition}

\begin{proof}
Applying \cite[Proposition 1.22]{Ouhabaz} to $\mathbf L_{\kappa_0}$, $\kappa_0:=\|b\|_\infty^2/2$, whose associated operator is $A_2+\kappa_0$, shows that $A_2+\kappa_0$ is m-accretive, hence $\|e^{-tA_2}\|_{\mathcal L(L^2)}\le e^{\kappa_0 t}$ \cite[Proposition 1.51]{Ouhabaz}. For $\kappa>\kappa_0$, $A_2+\kappa$ is sectorial by \cite[Proposition 1.27]{Ouhabaz}, and the holomorphy of $(e^{-tA_2})_{t\ge0}$ follows from \cite[Theorem 1.52]{Ouhabaz}.

For $u \in \HsD$ with $\mathscr L u \in L^2(\mathscr D)$, the trace $\partial_{\mathbf n_e} u \in H^{-1/2}(\partial\mathscr D)$ is well-defined via the Green formula
\begin{equation}\label{eq:green}
  \int_\mathscr D \nabla u \cdot \overline{\nabla v}
  = -\int_\mathscr D (\Delta u)\,\bar v
  + \langle \partial_{\mathbf n_e} u,\, \bar v \rangle_{\partial\mathscr D},
  \qquad \forall\, v \in H^1(\mathscr D),
\end{equation}
where the last bracket denotes the duality pairing between $H^{-1/2}(\partial\mathscr D)$ and $H^{1/2}(\partial\mathscr D)$.

\medskip
\noindent
\textbf{($\subset$).}
Let $u \in \dom(A_2)$ {and set $f:=A_2u$}. By  definition (see  \eqref{eq.LA2}): 
\begin{equation}\label{eq:formid}
  \mathbf{L}(u,v) = \langle f,v  \rangle_{L^2(\mathscr D)} \qquad \forall\, v \in \HsD.
\end{equation}
Testing with $v \in \mathcal{D}(\mathscr D) :=C^\infty_c(\mathscr D)$ and integrating by parts gives
\[
  \langle -\Delta u {-} b \cdot \nabla u,\, v
  \rangle_{\mathcal{D}',\mathcal{D}} = \langle f,v  \rangle_{L^2(\mathscr D)},
\]
so ${-}\mathscr L u = f \in L^2(\mathscr D)$ in the sense of distributions (and thus in $L^2(\mathscr D)$). Above $\langle \cdot, \cdot \rangle_{\mathcal{D}', \mathcal{D}}$ denotes the duality bracket between distributions and test functions. Recall that  $\partial_{\mathbf n_e} u \in H^{-1/2}(\partial\mathscr D)$ is well-defined.
Applying \eqref{eq:green} to the form \eqref{eq:form} and using $v|_{\SigD} = 0$
for all $v \in \HsD$,
\[
  \mathbf{L}(u,v)
  = -\int_\mathscr D (\Delta u)\,\bar v
    + \langle \partial_{\mathbf n_e} u,\, \bar v\rangle_{\SigN}
    {-} \int_\mathscr D (b \cdot \nabla u)\,\bar v
  =  \langle f,v \rangle _{L^2(\mathscr D)} + \langle \partial_{\mathbf n_e} u,\, \bar v\rangle_{\SigN}.
\]
Comparing with \eqref{eq:formid} gives
$\langle \partial_{\mathbf n_e} u,\, \bar v\rangle_{\SigN} = 0$
for all $v \in \HsD$.
Since the trace map $\HsD \to H^{1/2}(\SigN)$ is surjective,
we conclude
$\partial_{\mathbf n_e} u|_{\SigN} = 0 \text{ in } H^{-1/2}(\SigN)$. 

\medskip
\noindent
\textbf{($\supset$).}
Let $u \in \HsD$ with ${-}\mathscr L u = f \in L^2(\mathscr D)$ and $\partial_{\mathbf n_e} u|_{\SigN} = 0$. 
For every $v \in \HsD$, applying \eqref{eq:green} and using $v|_{\SigD} = 0$,
\begin{align*}
  \mathbf{L}(u,v)
  &= -\int_\mathscr D (\Delta u)\,\bar v
     + \underbrace{\langle \partial_{\mathbf n_e} u,\,\bar v\rangle_{\SigD}}_{=\,0}
     + \langle \partial_{\mathbf n_e} u,\,\bar v\rangle_{\SigN}
     {-} \int_\mathscr D (b \cdot \nabla u)\,\bar v \\
  &= \int_\mathscr D (-\Delta u {-} b \cdot \nabla u)\,\bar v
     + \underbrace{\langle \partial_{\mathbf n_e} u,\,\bar v\rangle_{\SigN}}_{=\,0}
  = \langle f,v  \rangle_{L^2(\mathscr D)}.
\end{align*}
Thus $u \in \dom(A_2)$ and $A_2u = f = {-}\mathscr L u$.

\end{proof}


For $u,v \in \HsD$, define the adjoint form \cite[Definition 1.5]{Ouhabaz}  of $\mathbf L$ by 
\begin{equation}\label{eq.ouh}
  \mathbf L^*(v,u) := \overline{\mathbf L(u,v)}= \int_\mathscr D \nabla v \cdot \overline{\nabla u}   {-} \int_{\mathscr{D}} (v b) \cdot \nabla \bar{u}. 
  \end{equation}
Using the divergence theorem  and the fact that
$u,v \in \HsD$ vanish    on $\SigD$,
$$
  \int_\mathscr D (vb)\cdot \nabla \bar u
  = -\int_\mathscr D (b\cdot\nabla v)\,\bar u
    - \int_\mathscr D v\,(\operatorname{div}b)\,\bar u
    + \int_{\SigN} (b\cdot \mathbf n_e)\, v\,\bar u ,
$$
we deduce that
for all $u,v \in \HsD$,
\begin{equation}\label{eq:Lstar}
  \mathbf L^*(v,u)
  = \int_\mathscr D \big(\nabla v \cdot \overline{\nabla u}
    {+} (b\cdot\nabla v)\,\bar u
    {+} (\operatorname{div}b)\,v\,\bar u\big)
  \;{-}\; \int_{\SigN} (b\cdot \mathbf n_e)\, v\,\bar u  .
\end{equation}
By the representation theorem,  $\mathbf L^*$ is the form
associated with $A^*_2$, the Banach-space adjoint of $A_2$: $\dom(A^*_2)\subset\HsD$
and $\mathbf L^*(v,u) = (A^*_2v,u)_{L^2(\mathscr D)}$ for all $u\in\HsD$, $v\in\dom(A^*_2)$ \cite[Proposition 1.24]{Ouhabaz}.

\begin{proposition}\label{prop:domAstar-L2}
Under the hypotheses of Proposition~\ref{pr.domA2}, the $L^2$-adjoint $A_2^*$
of $A_2$ satisfies
$$
  \dom(A_2^*)
  = \Bigl\{\, v \in \HsD, \ \mathscr L^* v \in L^2(\mathscr D) \text{ and }
    \partial_{\mathbf n_e} v {-} (b\cdot \mathbf n_e)\,v  |_{\SigN} = 0 \ \text{in } H^{-1/2}(\SigN) \,\Bigr\},
$$
where, in the sense of distributions,  $
  \mathscr L^* v {=\Delta v - b\cdot\nabla v - (\operatorname{div}b)\,v}$,
 and $\partial_{\mathbf n_e} v \in H^{-1/2}(\partial\mathscr D)$
is the  trace given by \eqref{eq:green}. Moreover ${A_2^*v = -\mathscr L^* v= -\Delta v+b\cdot\nabla v+(\operatorname{div}b)\,v}$ for every $v$ in this domain. 
\end{proposition}

\begin{proof}
Recall by \ref{eq.equivalence} that if $v\in\HsD$ and $\mathscr L^*v\in L^2(\mathscr D)$ then
$\Delta v \in L^2(\mathscr D)$. Hence $\partial_{\mathbf n_e} v \in
H^{-1/2}(\partial\mathscr D)$ is well-defined via \eqref{eq:green}, exactly as
in the preliminary step of the proof of Proposition~\ref{pr.domA2}.

\smallskip\noindent
\textbf{($\subset$).} Let $v\in\dom(A_2^*)$ and set $g:=A_2^*v$, so that
$\mathbf L^*(v,u) = \langle g,u  \rangle_{L^2(\mathscr D)}$ for all $u\in\HsD$. Testing \eqref{eq:Lstar}
against $u\in C^\infty_c(\mathscr D)$ (the boundary term drops out) gives
${-}\mathscr L^* v = g \in L^2(\mathscr D)$ in the sense of distributions. Applying
\eqref{eq:green} to the first term of \eqref{eq:Lstar} and using $u|_{\SigD}=0$:
\begin{align*}
  \mathbf L^*(v,u)
  &= -\int_\mathscr D \Delta v\,\bar u + \langle \partial_{\mathbf n_e} v,\bar u\rangle_{\SigN}  {+} \int_\mathscr D (b\cdot\nabla v)\bar u {+} \int_\mathscr D (\operatorname{div}b)v\bar u
    {-} \int_{\SigN} (b\cdot \mathbf n_e) v \bar u\\
  &= \langle g,u  \rangle_{L^2(\mathscr D)} + \big\langle \partial_{\mathbf n_e} v {-} (b\cdot \mathbf n_e)v,\ \bar u\big\rangle_{\SigN}.
\end{align*}
Since $\mathbf L^*(v,u)=\langle g,u  \rangle_{L^2(\mathscr D)}$, we get
$\langle \partial_{\mathbf n_e} v {-} (b\cdot \mathbf n_e)v,\bar u\rangle_{\SigN}=0$ for all $u\in\HsD$.
Because the trace map $\HsD \to H^{1/2}(\SigN)$ is surjective,
$\partial_{\mathbf n_e} v {-} (b\cdot \mathbf n_e)v|_{\SigN} = 0$ in $H^{-1/2}(\SigN)$.

\smallskip\noindent
\textbf{($\supset$).} Conversely, if $v\in\HsD$ satisfies ${-}\mathscr L^* v = g \in
L^2(\mathscr D)$ and $\partial_{\mathbf n_e} v {-} (b\cdot \mathbf n_e)v|_{\SigN}=0$, the same
computation run in reverse, for every $u\in\HsD$, gives
$\mathbf L^*(v,u) = \langle g,u  \rangle_{L^2(\mathscr D)}$, so $v\in\dom(A_2^*)$ and $A_2^*v = g = {-}\mathscr L^*v$.

\end{proof}



\subsubsection{Extensions to   $C_0$-semigroups  on $L^p(\mathscr D)$}

\begin{proposition}\label{pr.extensionS}
For $p\in (1,+\infty)$, the  semigroup $S_2=(e^{-tA_2})_{t\ge 0}$ (resp. {$S_2^*=(e^{-tA_2^*})_{t\ge 0}$}) generated by $-A_2$ (resp. $-A_2^*$)  on $L^2(\mathscr D)$ extends consistently to a
 $C_0$-semigroup $S_p$ (resp. $S_p^*$) on $L^p(\mathscr D)$ (resp. on $L^{p'}(\mathscr D)$, where $1/p+ 1/{p'}=1$){, and $S_p^*(t)$ is the Banach-space adjoint of $S_p(t)$}.
Let $-A_p$ denote the generator of $S_p$, i.e. $S_p=(e^{-tA_p})_{t\ge0}$. In addition, 
\begin{itemize}
  \item[$\mathfrak 1$.] If $p \in (2,+\infty)$, then
    $\dom(A_p) = \{ u \in \dom(A_2) \cap L^p(\mathscr D),  A_2 u \in L^p(\mathscr D) \}$
    and $A_p u = A_2u$.
    \smallskip
    
  \item[$\mathfrak 2$.] If $p \in (1,2)$, then $A_p$ is the closure of $A_2$ in $L^p(\mathscr D)$.
\end{itemize}
Finally, the infinitesimal generator of the dual semigroup $S_p^*$ is precisely the operator adjoint $  -A_p^*$ (namely  $S^{*}_p(t) = e^{-tA_p^*}$) and it holds: 
\begin{itemize}
  \item[$\mathfrak 3$.] If $p \in (1,2)$ (equivalently $p'\in(2,+\infty)$), then
$
      \dom(A_{p}^*)
      = \{\,
          v \in \dom(A_2^*) \cap L^{p'}(\mathscr D), 
          A_2^* v \in L^{p'}(\mathscr D)
        \} $
    and $A_{p}^* v = {A_2^*v=-\mathscr L^* v}$. 
    \smallskip
    
  \item[$\mathfrak 4$.] If $p \in (2,+\infty)$ (equivalently $p'\in(1,2)$), then $A_{p}^*$
    is the closure of $A_2^*$ in $L^{p'}(\mathscr D)$.
\end{itemize} 

\end{proposition}

\begin{proof} The proof is divided into  two steps. 
 
\underline{Step 1}. Extensions of $S_2$ and $S_2^*$ to a   $C_0$-semigroups on $L^p(\mathscr D)$

The form $\mathbf L$ is of the type considered in \cite[(2.1)]{OuhabazJOT}. Here $V=\HsD$, the second-order coefficients are $a_{kj}=\delta_{kj}$, the first-order coefficients are $-b_k$ in front of $\partial_k u\,\bar v$ (where $b=(b_1,\dots,b_d)$), there is no first-order term in front of $u\,\partial_k\bar v$, and there is no zero-order term. All coefficients are bounded since $b\in C^\infty(\bar{\mathscr D})$. One easily checks that all the assumptions  of \cite[Theorem 5.4]{OuhabazJOT} are satified. 
By \cite[Theorem 5.4 $(i)$]{OuhabazJOT} (with, say, $\varepsilon=1$), there exist $C,\delta>0$ and $\omega\ge0$ such that the following holds for every $t>0$. There is a measurable function $p(t,\cdot,\cdot)$ on $\mathscr D\times\mathscr D$ such that, for all $f\in L^2(\mathscr D)$, and a.e. $x,y\in\mathscr D$,
\begin{equation}\label{eq.kernelrep}
S_2(t)f(x)=\int_{\mathscr D}p(t,x,y)\,f(y)\,dy
\end{equation}
and
\begin{equation}\label{eq.gaussian}
|p(t,x,y)|\le C\,t^{-d/2}\,e^{\omega t}\,e^{-\frac{|x-y|^2}{4\delta t}}.
\end{equation}
 \smallskip
\noindent\textit{$L^q$ bounds.} Since $\int_{\mathbb R^d}e^{-|z|^2/(4\delta t)}\,dz=(4\pi\delta t)^{d/2}$, integrating \eqref{eq.gaussian} in $y$ (resp.\ in $x$) gives, with $M:=C(4\pi\delta)^{d/2}$,
\begin{equation}\label{eq.schur}
\operatorname*{ess\,sup}_{x\in\mathscr D}\int_{\mathscr D}|p(t,x,y)|\,dy\ \le\ M e^{\omega t}
\quad\text{and}\quad
\operatorname*{ess\,sup}_{y\in\mathscr D}\int_{\mathscr D}|p(t,x,y)|\,dx\ \le\ M e^{\omega t}.
\end{equation}
We claim that  that for all $t\ge0$ and $q\in(1,+\infty]$,
\begin{equation}\label{eq.Lqbound}
\|S_2(t)f\|_{L^q(\mathscr D)}\le Me^{\omega t}\|f\|_{L^q(\mathscr D)},\qquad f\in L^2(\mathscr D)\cap L^q(\mathscr D).
\end{equation}
We prove this claim. 
Enlarging $C$ if necessary, we may assume $M\ge1$. Then \eqref{eq.Lqbound} is trivial for $t=0$, so let $t>0$.

\emph{Case $q=+\infty$.} By \eqref{eq.kernelrep} and \eqref{eq.schur}, for a.e.\ $x\in\mathscr D$, it holds that:
$$
|S_2(t)f(x)|\le\|f\|_{L^\infty(\mathscr D)}\int_{\mathscr D}|p(t,x,y)|\,dy\le Me^{\omega t}\|f\|_{L^\infty(\mathscr D)}.
$$

\emph{Case $q\in(1,+\infty)$.} Let $q'$ be the conjugate exponent of $q$. For a.e.\ $x$, we have:
$$
|S_2(t)f(x)|\le\int_{\mathscr D}|p(t,x,y)|^{\frac1{q'}}\,|p(t,x,y)|^{\frac1q}|f(y)|\,dy\le\Big(\int_{\mathscr D}|p(t,x,y)|\,dy\Big)^{\frac1{q'}}\Big(\int_{\mathscr D}|p(t,x,y)|\,|f(y)|^q\,dy\Big)^{\frac1q}.
$$
Hence, we have: 
$$
\|S_2(t)f\|_{L^q(\mathscr D)}^q\le (Me^{\omega t})^{\frac q{q'}}\int_{\mathscr D}|f(y)|^q\Big(\int_{\mathscr D}|p(t,x,y)|\,dx\Big)dy\le (Me^{\omega t})^{q}\|f\|^q_{L^q(\mathscr D)} .
$$
This proves \eqref{eq.Lqbound}.

\smallskip
\noindent\textit{The adjoint semigroup.} Let $f,g\in L^2(\mathscr D)$. By \eqref{eq.schur} and since $f,g\in L^1(\mathscr D)$, we have that  $\int_{\mathscr D}\int_{\mathscr D}|p(t,x,y)||f(y)||g(x)|\,dy\,dx<+\infty$. Fubini's theorem then gives $(S_2(t)f,g)_{L^2(\mathscr D)}=(f,h)_{L^2(\mathscr D)}$, where
\begin{equation}\label{eq.pbar}
    h(y):=\int_{\mathscr D}\bar p(t,x,y)\,g(x)\,dx .
\end{equation}
Hence $S_2^*(t)g=S_2(t)^*g=h$. In other words, $S_2^*(t)$ has the kernel $p^*(t,y,x):=\bar p(t,x,y)$, which satisfies \eqref{eq.schur} with the roles of $x$ and $y$ exchanged. The proof of \eqref{eq.Lqbound} therefore applies verbatim, and \eqref{eq.Lqbound} also holds for $S_2^*(t)$.

\smallskip
\noindent\textit{Extension to $C_0$-semigroups.} Let $p\in(1,+\infty)$ and $T\in\{S_2,S_2^*\}$. Since $L^2(\mathscr D)\cap L^p(\mathscr D)$ is dense in $L^p(\mathscr D)$, \eqref{eq.Lqbound} shows that each $T(t)$ extends uniquely to a bounded operator $T_p(t)$ on $L^p(\mathscr D)$ with $\|T_p(t)\|_{\mathcal L(L^p(\mathscr D))}\le Me^{\omega t}$. By \eqref{eq.Lqbound}, $T(s)$ maps $L^2(\mathscr D)\cap L^p(\mathscr D)$ into itself. By density, $(T_p(t))_{t\ge0}$ is therefore a semigroup, consistent with $T$ on $L^2(\mathscr D)\cap L^p(\mathscr D)$. It is also strongly continuous. Since it is bounded on $[0,1]$, it suffices to check that $\|T_p(t)f-f\|_{L^p(\mathscr D)}\to0$ as $t\to0^+$ for $f$ in a dense subset of $L^p(\mathscr D)$:
\begin{itemize}
\item[$i)$] if $p\le2$ and $f\in L^2(\mathscr D)$, then $\|T_p(t)f-f\|_{L^p(\mathscr D)}\le|\mathscr D|^{\frac1p-\frac12}\|T(t)f-f\|_{L^2(\mathscr D)}\to0$;
\item[$ii)$] if $p>2$, let $f\in L^\infty(\mathscr D)$ and $t\in(0,1]$. By \eqref{eq.Lqbound} with $q=+\infty$, $g_t:=T(t)f-f$ satisfies $\|g_t\|_{L^\infty(\mathscr D)}\le(Me^{\omega}+1)\|f\|_{L^\infty(\mathscr D)}$. Hence $\|g_t\|_{L^p(\mathscr D)}^p\le\|g_t\|_{L^\infty(\mathscr D)}^{p-2}\|g_t\|_{L^2(\mathscr D)}^2\to0$.
\end{itemize}
We then set $S_p:=(S_2)_p$ and $S_p^*:=(S_2^*)_{p'}$. Use the duality $\langle f,g\rangle:=\int_{\mathscr D}f\,\bar g$ between $L^p(\mathscr D)$ and $L^{p'}(\mathscr D)$. For $f\in L^2(\mathscr D)\cap L^p(\mathscr D)$ and $g\in L^2(\mathscr D)\cap L^{p'}(\mathscr D)$,
$$
\langle S_p(t)f,g\rangle=(S_2(t)f,g)_{L^2(\mathscr D)}=(f,S_2^*(t)g)_{L^2(\mathscr D)}=\langle f,S^*_p(t)g\rangle .
$$
By density, $S_p^*(t)$ is the Banach-space adjoint of $S_p(t)$. Since $L^p(\mathscr D)$ is reflexive for $p\in(1,+\infty)$, the generator of the adjoint semigroup $(S_p(t)^*)_{t\ge0}$ is the adjoint $-A_p^*$ of the generator $-A_p$ of $S_p$ \cite[Corollary 1.10.6]{Pazy}. That is, $S^{*}_p(t) = S_p(t)^{*}$, $t\ge 0$.

\underline{Step 2}.  The operators  $A_p$ and their domains

The domain  $\mathscr D$ bounded gives the continuous embedding $L^q(\mathscr D)\subset L^r(\mathscr D)$ for $q\ge r$, and for all $p\in(1,+\infty)$, $S_p=(e^{-tA_p})_{t\ge0}$ is consistent with $S_2$ on $L^2(\mathscr D)\cap L^p(\mathscr D)$, i.e.
\begin{equation}\label{eq.consist}
S_p(t)f=S_2(t)f,\qquad f\in L^2(\mathscr D)\cap L^p(\mathscr D),\ t\ge0.
\end{equation}
For  $p\in(1,+\infty)$, $f\in L^2(\mathscr D)\cap L^p(\mathscr D)$, and $\kappa$ with $\operatorname{Re}\kappa$ large enough so that $\kappa\in\rho(-A_2)\cap\rho(-A_p)$ (where $\rho(T)$ denotes the resolvent set of $T$),
\begin{equation}\label{eq.reslemma}
(\kappa+A_2)^{-1}f=(\kappa+A_p)^{-1}f.
\end{equation}
Indeed  \eqref{eq.reslemma} is a consequence of the fact that 
\begin{equation*}
(\kappa+A_2)^{-1}f=\int_0^\infty e^{-\kappa t}S_2(t)f\,dt\ \text{ in }L^2(\mathscr D),\qquad
(\kappa+A_p)^{-1}f=\int_0^\infty e^{-\kappa t}S_p(t)f\,dt\ \text{ in }L^p(\mathscr D),
\end{equation*}
together with   \eqref{eq.consist}.

 \textit{Proof of $\mathfrak 1$.} 
 
 Here $L^p(\mathscr D)\subset L^2(\mathscr D)$, so $L^2(\mathscr D)\cap L^p(\mathscr D)=L^p(\mathscr D)$. Define the operator $B_p$ by 
\begin{equation*}
\dom(B_p):=\big\{u\in\dom(A_2)\cap L^p(\mathscr D), \, A_2u\in L^p(\mathscr D)\big\}, \quad B_pu:=A_2u,
\end{equation*}
 Fix $\kappa\in \rho(-A_2)\cap\rho(-A_p)$.

\emph{(a)} Let $f\in L^p(\mathscr D)\subset L^2(\mathscr D)$ and $u:=(\kappa+A_2)^{-1}f\in\dom(A_2)$. By \eqref{eq.reslemma} (with $q=p$), $u=(\kappa+A_p)^{-1}f$, so $u\in\dom(A_p)\subset L^p(\mathscr D)$, and  $A_2u=f-\kappa u\in L^p(\mathscr D)$. Thus $u\in\dom(B_p)$ and
$
(\kappa+B_p)u=f$. Consequently, 
 $(\kappa+B_p)\circ(\kappa+A_2)^{-1} |_{L^p}=\operatorname{id}_{L^p}$. In particular $\kappa+B_p$ is surjective onto $L^p(\mathscr D)$.

\emph{(b)} If $u\in\dom(B_p)$ and $(\kappa+B_p)u=0$, then $u\in\dom(A_2)$ and $(\kappa+A_2)u=0$ in $L^2(\mathscr D)$. Since $\kappa\in\rho(-A_2)$, $u=0$. Hence $\kappa+B_p$ is injective.

By \emph{(a)} and \emph{(b)}, $\kappa+B_p:\dom(B_p)\to L^p(\mathscr D)$ is bijective with inverse $(\kappa+A_2)^{-1}|_{L^p}$, which equals $(\kappa+A_p)^{-1}$ (since for any $f \in  L^p(\mathscr D)$, $(\kappa + A_2)^{-1}f = (\kappa + A_p)^{-1}f$). Therefore, since $(\kappa + B_p)^{-1} = (\kappa + A_p)^{-1}$, 
\begin{equation*}
\dom(B_p)= \left\{(\kappa + B_p)^{-1} f, \  f \in L^p(\mathscr D)\right\} = \left\{(\kappa + A_p)^{-1} f, \  f \in L^p(\mathscr D)\right\} =  \operatorname{Ran}\big((\kappa+A_p)^{-1}\big)=\dom(A_p), 
\end{equation*}
and $B_p=A_p$ 
on this common domain. This proves the desired result. 

\medskip
\noindent\textit{Proof of $\mathfrak 2$ ($p\in(1,2)$).} 

Here $L^2(\mathscr D)\subset L^p(\mathscr D)$, densely, so $A_2$,  with domain $\dom(A_2)\subset L^2(\mathscr D)\subset L^p(\mathscr D)$ and values $A_2u\in L^2(\mathscr D)\subset L^p(\mathscr D)$,  is a densely defined operator on $L^p(\mathscr D)$, $\dom(A_2)$ being dense in $L^2(\mathscr D)$ and hence in $L^p(\mathscr D)$. Fix $\kappa>0$ in $\rho(-A_2)\cap\rho(-A_p)$.

\emph{(a)} $\bar A_2\subset A_p$ (where the closure  is taken  in $L^p(\mathscr D)$). Let $u\in\dom(A_2)$ and $f:=(\kappa+A_2)u\in L^2(\mathscr D)=L^2(\mathscr D)\cap L^p(\mathscr D)$. By \eqref{eq.reslemma}, $u=(\kappa+A_2)^{-1}f=(\kappa+A_p)^{-1}f$, so $u\in\dom(A_p)$ and $A_pu=f-\kappa u=A_2u$. Thus $A_2\subset A_p$. Since $A_p$ is closed in $L^p(\mathscr D)$, we deduce that $A_2$ has a closed extension in $L^p(\mathscr D)$ (namely $A_p$), $A_2$ is therefore closable in $L^p(\mathscr D)$, and $\bar A_2\subset A_p$. 

\emph{(b)} $A_p\subset\bar A_2$. Let $u\in\dom(A_p)$ and $g:=(\kappa+A_p)u\in L^p(\mathscr D)$. Since $L^2(\mathscr D)$ is dense in $L^p(\mathscr D)$, pick $g_n\in L^2(\mathscr D)$ with $g_n\to g$ in $L^p(\mathscr D)$. As $L^2(\mathscr D)\subset L^p(\mathscr D)$, $g_n\in L^2(\mathscr D)\cap L^p(\mathscr D)$, so set $u_n:=(\kappa+A_2)^{-1}g_n\in\dom(A_2)$. By \eqref{eq.reslemma}, $u_n=(\kappa+A_p)^{-1}g_n$ as well. Since $(\kappa+A_p)^{-1}$ is bounded on $L^p(\mathscr D)$, we have in $L^p(\mathscr D)$ and as $n\to +\infty$:
\begin{equation*}
u_n=(\kappa+A_p)^{-1}g_n\to (\kappa+A_p)^{-1}g=u \ \text{ and } \ 
A_2u_n=g_n-\kappa u_n\to  g-\kappa u=A_pu. 
\end{equation*}
Thus $u_n\in\dom(A_2)$, $u_n\to u$ and $A_2u_n\to A_pu$ in $L^p(\mathscr D)$ imply  $u\in\dom(\bar A_2)$ with $\bar A_2u=A_pu$. Hence $A_p\subset\bar A_2$.

Combining \emph{(a)} and \emph{(b)} yields $A_p=\bar A_2$.

The proof of items $\mathfrak 3$ and $\mathfrak 4$ follows from the arguments used in Step 2 applied to $A_p^*$ on $L^{p'}(\mathscr D)$). 

\end{proof}

\subsection{Existence of sufficiently smooth non negative eigenfunctions for $A_p$ and $A_p^*$}

We decompose the proof of Theorem \ref{th.2} into Theorem \ref{th.3} and several Propositions. Theorem \ref{th.3} is a more precise and involved version of item $\mathfrak 1$ of Theorem \ref{th.2}. Then, Lemma \ref{le.comp} and Proposition \ref{pr.dominance} together with Proposition \ref{pr.sectorial} prove item $\mathfrak 2$. Finally, Propositions \ref{pr.simple} and \ref{th.Brown}, respectively deals with items $\mathfrak 3$ and $\mathfrak 4$.

\begin{theorem}\label{th.3}
Let the assumptions of Theorem \ref{th.2} hold. 
Then, there exist a real number $\lambda  \in \mathbb{R}$ and an eigenfunction
$$
\mathbf u \in \Big ( \bigcap_{p \in (1,+\infty)} \dom(A_p)  \Big) \cap  \Big( \bigcap_{\varepsilon>0} C^{0, 1/2-\varepsilon}(\bar{\mathscr D})  \Big) \cap C^\infty(\bar{\mathscr D}\setminus \partial \Sigma),
$$
satisfying $\mathbf u > 0$ in $\mathscr D \cup \SigN$ and $A_p \mathbf u = \lambda  \mathbf u$, for every $p \in (1,+\infty)$. 

Moreover, $\lambda$ is an eigenvalue of the dual operator $A_{p}^*$ for all $p \in (1,+\infty)$, with an associated dual eigenfunction
$$
\mathbf v \in  \Big( \bigcap_{p \in (1,+\infty)} \dom(A_{p}^*)  \Big) \cap  \Big( \bigcap_{\varepsilon>0} C^{0, 1/2-\varepsilon}(\bar{\mathscr D})  \Big)\cap C^\infty(\bar{\mathscr D}\setminus \partial \Sigma),
$$
satisfying $\mathbf v > 0$ in $\mathscr D$ and $A_{p}^* \mathbf v = \lambda \mathbf v $ for every $p \in (1,+\infty)$.
\end{theorem}

We now turn to the proof of Theorem \ref{th.3}. 

\subsubsection{Resolvent operators and a Krein-Rutman theorem}

We recall some well-known results on operator theory in Appendix \ref{sec.OperatorTheory}, in particular the Krein-Rutman theorem \ref{KR}.
The goal is to apply that theorem  to the resolvent of $A_p$, for some well-chosen $p> 1$. 
Define for $\kappa\in \rho(-A_{p})$  the resolvent operator $
  R_\kappa^{p} := (A_{p} + \kappa)^{-1}
  : L^{p}(\mathscr D) \to L^{p}(\mathscr D)$. 
The range of   $R_\kappa^{p}$ is    $\dom(A_{p})$.   
 Let 
\begin{equation}\label{eq.w2}
\omega_2:=\frac12\|b\|_\infty^2,
\end{equation}
  so that $\|e^{-tA_2}f\|_{L^2(\mathscr D)}\le e^{\omega_2 t}\|f\|_{L^2(\mathscr D)}$ for all $t\ge0$ (see Proposition~\ref{pr.domA2}). Let $p\ge 2$. By Riesz-Thorin interpolation with $\|e^{-tA_2}f\|_{L^\infty(\mathscr D)}\le\|f\|_{L^\infty(\mathscr D)}$, one also has $\|S_{p}(t)\|_{\mathcal L(L^{p}(\mathscr D))}\le e^{2\omega_2t/p}\le e^{\omega_2t}$. Hence every $\kappa>\omega_2$ belongs to $\rho(-A_2)\cap\rho(-A_{p})$, and, since $L^{p}(\mathscr D)\subset L^2(\mathscr D)$, \eqref{eq.reslemma} gives
\begin{equation}\label{eq.Rp0R2}
\forall p\ge 2, \forall \kappa>\omega_2, \forall f\in L^{p}(\mathscr D), \quad
R_\kappa^{p}f=R_\kappa^2f.
\end{equation}
 Set 
$$
  p_{\max} :=
  \begin{cases}
    +\infty, & d \leq 3, \\
    \dfrac{2d}{d-3}, & d \geq 4.
  \end{cases}
$$
In both cases $p_{\max} > 2$. Fix once and for all  
$$
  p_0 \in [2,\, p_{\max}).
$$

\begin{lemma}\label{lem:compact}
Let $\kappa\in \rho(-A_{p_0})$. Then $R_\kappa^{p_0} : L^{p_0}(\mathscr D) \to L^{p_0}(\mathscr D)$ is compact.
\end{lemma}

\begin{proof}
Fix $\kappa_0> \omega_2$ (see \eqref{eq.w2}).  

\underline{Step 1}. Reduction to $L^2$.

Since $\mathscr D$ is bounded and $p_0 \ge 2$, we have
$L^{p_0}(\mathscr D) \hookrightarrow L^2(\mathscr D)$ with
$\|f\|_{L^2(\mathscr D)} \leq |\mathscr D|^{1/2 - 1/p_0}\|f\|_{L^{p_0}(\mathscr D)}$.
Since $\kappa_0\in\rho(-A_2)\cap\rho(-A_{p_0})$, \eqref{eq.reslemma} gives, for every $f \in L^{p_0} (\mathscr D)\subset L^2(\mathscr D)$,
$R_{\kappa_0}^{p_0} f = R_{\kappa_0}^2 f$.

\underline{Step 2}. Savar\'e regularity.

For any $g \in L^2(\mathscr D)$, the solution $w = R_{\kappa_0}^2 g \in \HsD$
of the mixed boundary value problem satisfies
(see \cite{Savare1997})
\[
\forall \varepsilon \in (0,\varepsilon_0), \exists C_\ve>0, \forall  w \in H^{3/2-\varepsilon}(\mathscr D), \quad
  \|w\|_{H^{3/2-\varepsilon}(\mathscr D)} \leq C_\varepsilon \|g\|_{L^2(\mathscr D)}.
\]
Combined with Step 1:
\[
  \|R_{\kappa_0}^{p_0} f\|_{H^{3/2-\varepsilon}(\mathscr D)}
  \leq C_\varepsilon |\mathscr D|^{1/2-1/p_0} \|f\|_{L^{p_0}(\mathscr D)},
\]
so $R_{\kappa_0}^{p_0} : L^{p_0}(\mathscr D) \to H^{3/2-\varepsilon}(\mathscr D)$ is bounded.

\underline{Step 3}. Rellich-Kondrachov.

By the choice $p_0 < p_{\max}$, one has $\frac{3}{2} > \frac{d}{2} - \frac{d}{p_0}$.
Choose $\varepsilon > 0$ small enough that $s_\varepsilon := 3/2 - \varepsilon > d/2 - d/p_0$.
Then $H^{s_\varepsilon}(\mathscr D) \hookrightarrow L^{p_0}(\mathscr D)$ compactly by Rellich-Kondrachov.
Hence the composition
  $$R_{\kappa_0}^{p_0} :  L^{p_0}(\mathscr D) \stackrel{\text{bdd}}{\to} H^{s_\varepsilon}(\mathscr D)  \stackrel{\text{cpt}}{\hookrightarrow} L^{p_0}(\mathscr D)$$
is compact. 
Let $\kappa\in\rho(-A_{p_0})$. By the resolvent identity in $L^{p_0}(\mathscr D)$,
$$
  R_\kappa^{p_0}-R_{\kappa_0}^{p_0}=(\kappa_0-\kappa)\,R_{\kappa_0}^{p_0}R_\kappa^{p_0},
  \qquad\text{i.e.}\qquad
  R_\kappa^{p_0}=R_{\kappa_0}^{p_0}\big(I+(\kappa_0-\kappa)R_\kappa^{p_0}\big).
$$
Thus, the operator $R_\kappa^{p_0}$ is compact.

\end{proof}

\begin{lemma}\label{lem:positive} Let  $\kappa\in \rho(-A_{p_0})$.
If $0 \leq f \in L^{p_0}(\mathscr D)$ then $R_\kappa^{p_0} f \geq 0$ a.e.
\end{lemma}

\begin{proof}
Set $u = R_\kappa^{p_0} f$ and recall that $u^\pm \in \HsD$ since $u\in\HsD$.
Since a.e. $\nabla u^+ \cdot \nabla u^- = 0$ and $(b \cdot \nabla u^+) u^- = 0$, testing the equation
$\mathbf{L}_\kappa(u, v) = \langle f, v\rangle_{L^2(\mathscr D)}$ with $v = u^-$ gives
\[
  0 \leq \langle f, u^-\rangle _{L^2(\mathscr D)}
  = \mathbf{L}_\kappa(u, u^-)
  = \mathbf{L}_\kappa(u^+ - u^-, u^-)
  = -\mathbf{L}_\kappa(u^-, u^-)
  \leq -c \|u^-\|_{H^1(\mathscr D)}^2 \leq 0.
\]
Hence $u^- = 0$, i.e.\ $u \geq 0$ a.e.

\end{proof}

\begin{proposition} \label{pr.irred2}
Let  $\kappa>\omega_2$ (see \eqref{eq.w2}). Then
$$
  0\le f\in L^{p_0}(\mathscr D),\ f\not\equiv0 \quad\implies\quad R_{\kappa}^{p_0}f>0
  \text{ a.e. in }\mathscr D.
$$
In particular $R_\kappa^{p_0}$ is an irreducible compact positive operator on the Banach
lattice $L^{p_0}(\mathscr D)$ from Proposition \ref{pr.irr}.
\end{proposition}

\begin{proof}
\underline{Step 1}. Positivity improving property of $(e^{-tA_2})_{t\ge0}$.

By \cite[Corollary~4.3]{Ouhabaz} the semigroup $(e^{-tA_2})_{t\ge0}$ is positive on $L^2(\mathscr D)$. Since $\mathscr D$ is connected, \cite[Theorem~4.5]{Ouhabaz} shows that $(e^{-tA_2})_{t\ge0}$ is irreducible \cite[Definition~2.8]{Ouhabaz}, i.e.  for every $t>0$ and every $0\le f\in L^2(\mathscr D)$, $f\not\equiv0$,
\begin{equation}
  e^{-tA_2}f(x) > 0 \quad\text{for a.e. } x\in\mathscr D. \label{eq:irrL2}
\end{equation}

\underline{Step 2}. Positivity improving property of $R^2_\kappa$.

Since $\kappa>\omega_2$, the Hille-Yosida representation gives
$$
  R_{\kappa}^{2}f = \int_0^\infty e^{-{\kappa} t}\,e^{-tA_{2}}f\,dt, \qquad f\in L^{2}(\mathscr D),
$$
as an absolutely convergent Bochner integral in $L^2(\mathscr D)$.
Fix $0\le f\in L^2(\mathscr D)$, $f\not\equiv0$. Since $t\mapsto e^{-tA_2}f$ is continuous from $\mathbb R_+$ into $L^2(\mathscr D)$, it admits a jointly measurable representative $h$ on $\mathbb R_+\times\mathscr D$, i.e.\ $h(t,\cdot)=e^{-tA_2}f$ a.e.\ in $\mathscr D$ for every $t\ge0$ \rev{(for instance,  one can extend $e^{-tA_2}f$ by $0$ outside $\mathscr D$ and then consider  $h(t,x):=\limsup_{n\to\infty} |\mathscr B(0,1/n)|^{-1}\int_{\mathscr B(x,1/n)}e^{-tA_2}f(y)\,dy$. Each average is jointly continuous in $(t,x)$, hence $h$ is jointly measurable, and by the Lebesgue differentiation theorem $h(t,\cdot)=e^{-tA_2}f$ a.e.\ for every $t\ge0$)}.
By \eqref{eq:irrL2}, for every fixed $t>0$ the set
$$
  N_t := \{x\in\mathscr D,\  h(t,x)\le0\}
$$
is Lebesgue-null. Denote by ${\rm Leb}$ the Lebesgue measure on $\mathbb R^d$ and by ${\rm Leb}_1$ the Lebesgue measure on $\mathbb R$. Let $E:=\{(t,x)\in(0,+\infty)\times\mathscr D,\ h(t,x)\le0\}$, a measurable subset of the product space. By the Fubini-Tonelli theorem (integrating in $x$ first),
$$
  ({\rm Leb}_1\otimes {\rm Leb})(E) = \int_0^\infty {\rm Leb}(N_t)\,dt = 0.
$$
By the Fubini-Tonelli theorem again (integrating in $t$ first), with $E^x:=\{t>0,\ h(t,x)\le0\}$,
$$
  0 = ({\rm Leb}_1\otimes {\rm Leb})(E) = \int_{\mathscr D}{\rm Leb}_1(E^x)\,dx,
$$
so that ${\rm Leb}_1(E^x)=0$ for a.e.\ $x\in\mathscr D$. Set
$$
  N^1 := \{x\in\mathscr D,\ {\rm Leb}_1(E^x)>0\}, \text{ so that } {\rm Leb}(N^1)=0.
$$
On the other hand, by the Cauchy-Schwarz inequality and since ${\rm Leb}(\mathscr D)<\infty$,
$$
\int_0^\infty e^{-\kappa t}\|h(t,\cdot)\|_{L^1(\mathscr D)}\,dt\le {\rm Leb}(\mathscr D)^{1/2}\|f\|_{L^2(\mathscr D)}\int_0^\infty e^{-(\kappa-\omega_2) t}\,dt<\infty .
$$
By the Fubini-Tonelli theorem once more, $\int_0^\infty e^{-\kappa t}|h(t,x)|\,dt<+\infty$ for a.e.\ $x\in\mathscr D$, and for such $x$ the Bochner integral $R_{\kappa}^2f=\int_0^\infty e^{-\kappa t}h(t,\cdot)\,dt$ coincides with the Lebesgue integral $\int_0^\infty e^{-\kappa t}h(t,x)\,dt$. Set
$$
  N^2 := \Big\{x\in\mathscr D,\ \int_0^\infty e^{-\kappa t}|h(t,x)|\,dt=+\infty
  \ \text{or}\ R_{\kappa}^2f(x)\ne\int_0^\infty e^{-\kappa t}h(t,x)\,dt\Big\},
  \text{ so that } {\rm Leb}(N^2)=0.
$$
Let $N:=N^1\cup N^2$, so that ${\rm Leb}(N)=0$, and fix $x\in\mathscr D\setminus N$. Since $x\notin N^1$, $h(t,x)>0$ for a.e.\ $t>0$, hence $\int_0^\infty e^{-\kappa t}h(t,x)\,dt>0$. Since $x\notin N^2$, this integral is finite and equals $R_{\kappa}^2f(x)$. Therefore $R_{\kappa}^2f(x)>0$ for every $x\in\mathscr D\setminus N$, i.e.
\begin{equation}
  0\le f\in L^2(\mathscr D),\ f\not\equiv0 \implies R_{\kappa}^2f>0 \text{ a.e. in }\mathscr D. \label{eq:irrL2res}
\end{equation}

\underline{Step 3}. Conclusion.

Let $0\le f\in L^{p_0}(\mathscr D)$, $f\not\equiv0$. Since $p_0\ge2$, $f\in L^2(\mathscr D)$, so that \eqref{eq.Rp0R2} and \eqref{eq:irrL2res} give
\begin{equation}\label{eq.irr}
  0\le f\in L^{p_0}(\mathscr D),\ f\not\equiv0 \implies R_\kappa^{p_0}f=R_\kappa^{2}f>0 \text{ a.e. in }\mathscr D,
\end{equation}
which is the first assertion of the proposition. In particular $R_\kappa^{p_0}$ is a positive operator on $L^{p_0}(\mathscr D)$. It is compact by Lemma~\ref{lem:compact}, since $\kappa\in\rho(-A_{p_0})$, and it is irreducible by \eqref{eq.irr} and Proposition~\ref{pr.irr} above. This ends the proof of the proposition.

\end{proof}

 The space $L^{p_0}(\mathscr D)$ is a Banach lattice with positive cone
$$K =L^{p_0}_+(\mathscr D) :=  \{ f \in L^{p_0}(\mathscr D)),  f(x) \ge 0 \text{ a.e. on } \mathscr D  \},$$ which is total (its closed linear span is $L^{p_0}(\mathscr D)$), and whose dual cone is $K^*=L^{p_0'}_+(\mathscr D)$.
We fix once and for all 
$$\kappa>\omega_2.$$
 By Proposition~\ref{pr.irred2},
$R_\kappa^{p_0}$ is compact, positive, and irreducible.
Applying Theorem~\ref{KR} and setting $r_{p_0} := \mathsf r_{sp}(R_\kappa^{p_0}) > 0$, there exist
$\mathbf u_{p_0} \in K \setminus \{0\}$ and $\mathbf v_{p_0'}\in K^*\setminus\{0\}$ such that
\begin{equation}\label{eq.KRuv}
  R_\kappa^{p_0} \mathbf u_{p_0} = r_{p_0}\, \mathbf u_{p_0}
  \qquad\text{and}\qquad
  (R_\kappa^{p_0})^* \mathbf v_{p_0'} = r_{p_0}\, \mathbf v_{p_0'} .
\end{equation}
Since $\mathbf u_{p_0}=r_{p_0}^{-1}R_\kappa^{p_0}\mathbf u_{p_0}$ with $0\le \mathbf u_{p_0}\not\equiv0$, \eqref{eq.irr} gives $\mathbf u_{p_0}>0$ a.e.\ in $\mathscr D$. Moreover, for every $0\le f\in L^{p_0}(\mathscr D)$, $f\not\equiv0$, \eqref{eq.irr} gives
$$
r_{p_0}\int_{\mathscr D} f\,\mathbf v_{p_0'}=\int_{\mathscr D} (R_\kappa^{p_0}f)\,\mathbf v_{p_0'}>0,
$$
since $R_\kappa^{p_0}f>0$ a.e.\ and $0\le\mathbf v_{p_0'}\not\equiv0$. Choosing $f=\mathbf 1_{\{\mathbf v_{p_0'}=0\}}$ shows that $\{\mathbf v_{p_0'}=0\}$ is Lebesgue-null, i.e.\ $\mathbf v_{p_0'}>0$ a.e.\ in $\mathscr D$.

Since $\mathbf u_{p_0}=r_{p_0}^{-1}R_\kappa^{p_0}\mathbf u_{p_0}\in\operatorname{ran}(R_\kappa^{p_0})=\dom(A_{p_0})$,
applying $A_{p_0} + \kappa$ to the first identity in \eqref{eq.KRuv} yields
\begin{equation*}
A_{p_0} \mathbf u_{p_0} = \lambda_{p_0} \mathbf u_{p_0}, \qquad \lambda_{p_0} := r_{p_0}^{-1} - \kappa \in \mathbb{R}.
\end{equation*}
Similarly, since $L^{p_0}(\mathscr D)$ is reflexive, the adjoint of $R_\kappa^{p_0}$ is $(R_\kappa^{p_0})^*=(A_{p_0}^* + \kappa)^{-1}$, so that $\mathbf v_{p_0'}\in\dom(A_{p_0}^*)$ and
\begin{equation*}
A_{p_0}^* \mathbf v_{p_0'} = \lambda_{p_0} \mathbf v_{p_0'} .
\end{equation*}
The functions $\mathbf u_{p_0}$, $\mathbf v_{p_0'}$ and the number $\lambda_{p_0}$ are the candidates for $\mathbf u$, $\mathbf v$ and $\lambda$ in Theorem~\ref{th.3}. We shall now study their regularity properties.

\subsubsection{On the functions  $\mathbf u_{p_0} $ and $\mathbf v_{p_0'} $}
\begin{lemma}\label{lem:v1-H1}
The function $\mathbf v_{p_0'}$ belongs to $\dom(A_2^*)\subset\HsD$, and $A_2^*\mathbf v_{p_0'}=\lambda_{p_0}\mathbf v_{p_0'}$.
\end{lemma}

\begin{proof}
\underline{Step 1}. $\mathbf v_{p_0'}  \in L^2(\mathscr D)$.

Since $\kappa>\omega_2=\frac12\|b\|_\infty^2$, the proof of Lemma~\ref{lem:compact} (applied with $\kappa_0=\kappa$) shows that
\[
  R_\kappa^{p_0} = \iota \circ \tilde R_\kappa^2 \circ j,
\]
where $j : L^{p_0}(\mathscr D) \hookrightarrow L^2(\mathscr D)$ is the
bounded embedding, $\tilde R_\kappa^2 : L^2(\mathscr D) \to
H^{3/2-\varepsilon}(\mathscr D)$ denotes $R_\kappa^2$ viewed with values in $H^{3/2-\varepsilon}(\mathscr D)$ (bounded by Savar\'e's estimate),
and $\iota : H^{3/2-\varepsilon}(\mathscr D) \hookrightarrow L^{p_0}(\mathscr D)$ is the Sobolev embedding (for $\varepsilon>0$ small enough).
Taking Banach-space adjoints (with respect to the pairing $\langle f,g\rangle=\int_{\mathscr D}fg$ because every operator considered has real coefficients) reverses the composition:
$$
  (R_\kappa^{p_0})^* = j^* \circ (\tilde R_\kappa^2)^* \circ \iota^*,
$$
where
$$
  \iota^* : L^{p_0'}(\mathscr D) \hookrightarrow \bigl(H^{3/2-\varepsilon}(\mathscr D)\bigr)^*,
  \qquad
  (\tilde R_\kappa^2)^* : \bigl(H^{3/2-\varepsilon}(\mathscr D)\bigr)^* \to L^2(\mathscr D), \qquad
  j^* : L^2(\mathscr D) \hookrightarrow L^{p_0'}(\mathscr D).
$$
By \eqref{eq.KRuv}, $\mathbf v_{p_0'}  = r_{p_0}^{-1} (R_\kappa^{p_0})^* \mathbf v_{p_0'}$. Hence
$$
  \mathbf v_{p_0'}
  = r_{p_0}^{-1}\, j^*\Big((\tilde R_\kappa^2)^* \bigl(\iota^* \mathbf v_{p_0'} \bigr)\Big)
  \ \in\ L^2(\mathscr D).
$$

\smallskip
\noindent
\underline{Step 2}. $\mathbf v_{p_0'}\in\dom(A_2^*)$.

Let $g\in L^2(\mathscr D)$. Since $p_0'\le 2$, $g\in L^{p_0'}(\mathscr D)$. Since $A_2$ has real coefficients and $\kappa$ is real, the adjoint of $R_\kappa^2$ with respect to the pairing $\int_{\mathscr D}fg$ is $(A_2^*+\kappa)^{-1}$. Hence, for every $h\in L^{p_0}(\mathscr D)\subset L^2(\mathscr D)$, using \eqref{eq.Rp0R2},
$$
\int_{\mathscr D}\big((A_{p_0}^*+\kappa)^{-1}g\big)\,h
=\int_{\mathscr D}g\,R_\kappa^{p_0}h
=\int_{\mathscr D}g\,R_\kappa^{2}h
=\int_{\mathscr D}\big((A_2^*+\kappa)^{-1}g\big)\,h .
$$
Since $h\in L^{p_0}(\mathscr D)$ is arbitrary, $(A_{p_0}^*+\kappa)^{-1}g=(A_2^*+\kappa)^{-1}g$. Applying this with $g=\mathbf v_{p_0'}\in L^2(\mathscr D)$ (Step 1), we obtain from \eqref{eq.KRuv}
\begin{equation}\label{eq.domAv}
  \mathbf v_{p_0'}  = r_{p_0}^{-1}\,(A_2^*+\kappa)^{-1}\mathbf v_{p_0'}  \in \dom(A_2^*)\subset\HsD\quad \text{ and }
  \quad A_2^*\mathbf v_{p_0'}=\lambda_{p_0}\mathbf v_{p_0'} .
\end{equation}
The proof of the lemma is complete.

\end{proof}

Note that because $p_0\geq2$, the similar result with $\mathbf u_{p_0}$ is immediate.

By Lemma~\ref{lem:v1-H1}, $\mathbf v_{p_0'} \in \dom(A_2^*)\subset\HsD$ and $A_2^*\mathbf v_{p_0'}=\lambda_{p_0}\mathbf v_{p_0'}$. Recall that $\mathbf L^*(v,\psi) = \langle A_2^*v,\psi\rangle _{L^2(\mathscr D)}$ for every $\psi\in\HsD$ and $v\in\dom(A_2^*)$. Applying this with $v=\mathbf v_{p_0'}$ yields
$$\mathbf L^*(\mathbf v_{p_0'},\psi) = \lambda_{p_0}\,\langle \mathbf v_{p_0'},\psi\rangle _{L^2(\mathscr D)}, \qquad \forall\,\psi\in\HsD.$$
Hence, by \eqref{eq:Lstar}, since $\mathbf v_{p_0'}$ is real-valued, for every real-valued $\psi\in\HsD$,
\begin{equation}\label{eq:v1-weak}
  \int_\mathscr D \nabla \mathbf v_{p_0'} \cdot\nabla\psi
  {+} \int_\mathscr D (b\cdot\nabla \mathbf v_{p_0'} )\,\psi  {+} \int_\mathscr D (\operatorname{div}b)\,\mathbf v_{p_0'} \,\psi
  {-} \int_{\SigN} (b\cdot \mathbf n_e)\,\mathbf v_{p_0'} \,\psi
  = \lambda_{p_0}\int_\mathscr D \mathbf v_{p_0'} \,\psi. 
\end{equation}
In view of \eqref{eq:v1-weak}, two features are new compared with the treatment of $\mathbf u_{p_0}$: the zeroth-order term ${(\operatorname{div}b)\mathbf v_{p_0'}}$ and the boundary flux term on $\SigN$.

\begin{proposition}\label{pr.conti}
For all $\ve>0$, $\mathbf u_{p_0}, \mathbf v_{p_0'} \in C^{0,1/2-\ve}(\bar{\mathscr D})$.
\end{proposition}

\begin{proof}
We use \cite[Theorem 5.3 and Corollary 5.4]{Shamir1968}. They apply to weak solutions $w$ of mixed problems of the form
\begin{equation}\label{eq.mixedShamir}
Aw=f \ \text{ in }\mathscr D,\qquad w=0\ \text{ on }\SigD,\qquad Bw=0\ \text{ on }\SigN,
\end{equation}
where $\partial\Sigma$ is a $C^3$ submanifold of $\partial\mathscr D$ of dimension $d-2$, $A=\sum_{|\beta|\le2}a_\beta D^\beta$ is a second-order elliptic operator with real coefficients $a_\beta\in C^{1+|\beta|}(\bar{\mathscr D})$ (zero-order terms included), and $B=\sum_{|\beta|\le1}b_\beta D^\beta$ is a first-order operator with real coefficients $b_\beta\in C^{1+|\beta|}(\bar{\mathscr D})$ which covers $A$ on $\partial\mathscr D$ \cite[Section 5, I-III]{Shamir1968}. If $f\in L^p(\mathscr D)$ for every $p<+\infty$ and $w\in H^1(\mathscr D)$, then, near each point of $\partial\Sigma$, $w\in H^{s,p}$ for every $p>4$ and $s< 1/2+ 2/p$ \cite[Theorem 5.3]{Shamir1968}. Hence $w$ is H\"older continuous of every order $\alpha< 1/2$ near $\partial\Sigma$ \cite[Corollary 5.4]{Shamir1968}. Away from $\partial\Sigma$, the boundary condition is locally either purely Dirichlet or purely Neumann/Robin, and standard elliptic regularity (the coefficients being smooth) gives $w\in C^\infty$ there. By compactness of $\bar{\mathscr D}$, we deduce $w\in C^{0,1/2-\ve}(\bar{\mathscr D})$ for every $\ve>0$.

By Proposition~\ref{pr.domA2}, $\mathbf u_{p_0}\in\HsD$ is a weak solution of \eqref{eq.mixedShamir} with
$$
A={-\mathscr L-\lambda_{p_0}=-\Delta-b\cdot\nabla-\lambda_{p_0}},\qquad B=\partial_{\mathbf n_e},\qquad f=0 . 
$$
By Lemma~\ref{lem:v1-H1} and Proposition~\ref{prop:domAstar-L2}, $\mathbf v_{p_0'}\in\HsD$ is a weak solution of \eqref{eq.mixedShamir} with
$$
A={-\mathscr L^*-\lambda_{p_0}=-\Delta+b\cdot\nabla+(\operatorname{div}b-\lambda_{p_0})},\qquad B=\partial_{\mathbf n_e}{-}(b\cdot\mathbf n_e),\qquad f=0.
$$
In both cases the coefficients of $A$ and $B$ are smooth on $\bar{\mathscr D}$ (since $b\in C^\infty(\bar{\mathscr D})$ and $\partial\mathscr D$ is smooth, $\mathbf n_e$ being extended inside $\mathscr D$ as the gradient of the signed distance to $\partial\mathscr D$), and $B$ covers $A$ since its principal part is the normal derivative. The claim follows.

\end{proof}

By interior elliptic regularity, $\mathbf u_{p_0},\mathbf v_{p_0'}\in C^\infty(\mathscr D)$. In addition, local elliptic boundary regularity (for Neumann's and Dirichlet's problems) yield $\mathbf u_{p_0},\mathbf v_{p_0'}\in C^\infty(\bar{\mathscr{D}}\setminus\partial\SigD)$. Moreover, $\mathbf u_{p_0}\ge0$, $\mathbf u_{p_0}\not\equiv0$, and
$$
-\Delta\mathbf u_{p_0}{-}b\cdot\nabla\mathbf u_{p_0}+\lambda_{p_0}^-\,\mathbf u_{p_0}=\lambda_{p_0}^+\,\mathbf u_{p_0}\ge0 \quad\text{in }\mathscr D , 
$$
so that, $\mathscr D$ being connected, the strong maximum principle yields $\mathbf u_{p_0}>0$ in $\mathscr D$. Similarly, setting $c:={\operatorname{div}b-\lambda_{p_0}}$, one has $-\Delta\mathbf v_{p_0'}{+}b\cdot\nabla\mathbf v_{p_0'}+c^+\mathbf v_{p_0'}=c^-\mathbf v_{p_0'}\ge0$ in $\mathscr D$, hence $\mathbf v_{p_0'}>0$ in $\mathscr D$. In addition, $\mathbf u_{p_0}>0$ on $\SigN$. \rev{Indeed, let $x_0\in\SigN$. Since $x_0\notin\partial\Sigma$, $\mathbf u_{p_0}$ is $C^1$ up to the boundary near $x_0$ (local Neumann regularity), and $\partial_{\mathbf n_e}\mathbf u_{p_0}(x_0)=0$. If $\mathbf u_{p_0}(x_0)=0$, then $x_0$ is a boundary minimum of $\mathbf u_{p_0}$ (as $\mathbf u_{p_0}>0$ in $\mathscr D$), and, $\partial\mathscr D$ being smooth,  we can use Hopf's lemma applied to the inequality above to deduce that  $\partial_{\mathbf n_e}\mathbf u_{p_0}(x_0)<0$, a contradiction.}

 \begin{proposition}\label{pr:allp}
Let  $p\in(1,+\infty)$
\begin{enumerate}
\item[$\mathfrak 1$.] We have $\mathbf{u}_{p_0} \in \dom(A_p)$ and $A_p \mathbf{u}_{p_0} = \lambda_{p_0} \mathbf{u}_{p_0}$.
\item[$\mathfrak 2$.] We have $\mathbf v_{p_0'} \in \dom(A_{p}^*)$ and $A_{p}^* \mathbf v_{p_0'} = \lambda_{p_0} \mathbf v_{p_0'}$.
\end{enumerate}
\end{proposition}

\begin{proof}
We only treat $\mathbf{u}_{p_0}$, the proof of item $\mathfrak 2$ goes similarly using in addition Lemma \ref{lem:v1-H1}.

Since $p_0\ge2$, Proposition~\ref{pr.extensionS} gives $\mathbf u_{p_0}\in\dom(A_{p_0})\subset\dom(A_2)$ and $A_2\mathbf u_{p_0}=A_{p_0}\mathbf u_{p_0}=\lambda_{p_0}\mathbf u_{p_0}$. By Proposition~\ref{pr.conti}, $\mathbf u_{p_0}\in C(\bar{\mathscr D})\subset L^p(\mathscr D)$ for every $p\in(1,+\infty)$, hence
$$
A_2\mathbf u_{p_0}=\lambda_{p_0}\mathbf u_{p_0}\in L^p(\mathscr D),\qquad \forall\,p\in(1,+\infty).
$$
Fix $p\in(1,+\infty)$. We consider two cases.
\begin{itemize}
\item[$i)$] \emph{Case $p\in[2,+\infty)$.} By Proposition~\ref{pr.extensionS},
$\dom(A_p)=\{w\in\dom(A_2)\cap L^p(\mathscr D),\ A_2w\in L^p(\mathscr D)\}$ and $A_pw=A_2w$. Hence $\mathbf u_{p_0}\in\dom(A_p)$ and $A_p\mathbf u_{p_0}=A_2\mathbf u_{p_0}=\lambda_{p_0}\mathbf u_{p_0}$.
\smallskip 

\item[$ii)$] \emph{Case $p\in(1,2)$.} By Proposition~\ref{pr.extensionS}, $A_p$ is the closure of $A_2$ in $L^p(\mathscr D)$ and  in particular $A_2\subset A_p$. Hence $\mathbf u_{p_0}\in\dom(A_2)\subset\dom(A_p)$ and $A_p\mathbf u_{p_0}=A_2\mathbf u_{p_0}=\lambda_{p_0}\mathbf u_{p_0}$.
\end{itemize}

\end{proof}

This completes the proof of Theorem \ref{th.3}.

\subsubsection{Compactness of the resolvent of $A_p$ for any $p>1$}
  \begin{lemma}\label{le.comp}
Let $p\in(1,+\infty)$. For every $\kappa\in\rho(-A_p)$, $R_\kappa^p:=(A_p+\kappa)^{-1}$ is compact on $L^p(\mathscr D)$. Consequently, $(R_\kappa^p)^*=(A_p^*+\kappa)^{-1}$ is compact on $L^{p'}(\mathscr D)$, i.e.\ $A_p$ and $A_p^*$ have compact resolvent.
\end{lemma}
 
 \begin{proof}
If $p=p_0$, this is Lemma~\ref{lem:compact}. Fix $p\in(1,+\infty)\setminus\{p_0\}$. If $p<p_0$, pick $q\in(1,p)$; if $p>p_0$, pick $q\in(p,+\infty)$. In both cases $p$ lies strictly between $p_0$ and $q$, so there is $\theta\in(0,1)$ with
$$
\frac1p=\frac{1-\theta}{p_0}+\frac{\theta}{q}.
$$
By the Riesz-Thorin theorem, $L^p(\mathscr D)$ is an interpolation space of exponent $\theta$ with respect to the pair $(L^{p_0}(\mathscr D),L^{q}(\mathscr D))$: every linear operator $T$ which is bounded on $L^{p_0}(\mathscr D)$ with norm $M_0$ and on $L^{q}(\mathscr D)$ with norm $M_1$ is bounded on $L^p(\mathscr D)$ with norm at most $M_0^{1-\theta}M_1^{\theta}$.

\underline{Step 1}. Compactness at a common point $\kappa_0$.

 Since $-A_2$, $-A_q$, $-A_{p_0}$ and $-A_p$ generate $C_0$-semigroups (Proposition~\ref{pr.extensionS}), their resolvent sets all contain a common half-line $(\kappa_*,+\infty)$; fix $\kappa_0>\kappa_*$. Since $p_0\ge2$, we have $L^{p_0}(\mathscr D)\cap L^{q}(\mathscr D)=L^{\max(p_0,q)}(\mathscr D)\subset L^2(\mathscr D)$, and by \eqref{eq.reslemma} the resolvents $R^{r}_{\kappa_0}$, $r\in\{q,p_0,p\}$, all coincide with $R^2_{\kappa_0}$ on $L^2(\mathscr D)\cap L^r(\mathscr D)$. They therefore define a single linear operator $T$ on
$$
L^{p_0}(\mathscr D)+L^{q}(\mathscr D)=\big\{f_0+f_1,\ f_0\in L^{p_0}(\mathscr D),\ f_1\in L^{q}(\mathscr D)\big\},
$$
which here equals $L^{\min(p_0,q)}(\mathscr D)$ since $\mathscr D$ is bounded, and whose restriction to $L^r(\mathscr D)$ is $R^r_{\kappa_0}$ for $r\in\{q,p_0,p\}$.

The operator $T$ is compact on $L^{p_0}(\mathscr D)$ by Lemma~\ref{lem:compact}, and bounded on $L^{q}(\mathscr D)$. Since $p_0<+\infty$, the interpolation pair $(L^{p_0}(\mathscr D),L^{q}(\mathscr D))$ satisfies Persson's approximation hypothesis \cite[Proposition, p.~218]{Persson1964}. Hence, by \cite[Theorem, p.~217]{Persson1964}, which extends Krasnosel'skii's theorem for $L^p$ spaces, $T$ is compact on $L^p(\mathscr D)$. That is, $R^p_{\kappa_0}$ is compact on $L^p(\mathscr D)$.

\underline{Step 2}. General $\kappa\in\rho(-A_p)$.

By the resolvent identity in $L^p(\mathscr D)$,
$$
R_\kappa^{p}=R_{\kappa_0}^{p}\big(I+(\kappa_0-\kappa)R_\kappa^{p}\big),
$$
which is compact as the composition of a compact and a bounded operator.

\underline{Step 3}. The adjoint.

By Schauder's theorem, $(R_\kappa^p)^*$ is compact on $L^{p'}(\mathscr D)$, and $(R_\kappa^p)^*=(A_p^*+\kappa)^{-1}$ since $L^p(\mathscr D)$ is reflexive.
\end{proof}

\begin{lemma}\label{le.Sp}
    Let $p\in(1,+\infty)$, let $\mu\in\mathbb C$. Then
    $$
    \ker(A_p-\mu)=\big\{w\in L^p(\mathscr D):\ S_p(t)w=e^{-\mu t}w \ \text{ for all } t\ge0\big\}.
    $$
    Moreover, the semigroup $S_p$ is positive, real, and positivity improving.
\end{lemma}

\begin{proof}~

\underline{Step 1}. Indeed, if $w\in\dom(A_p)$ and $A_pw=\mu w$, then by \cite[Theorem 1.2.4]{Pazy}, $\frac{d}{dt}\big(e^{\mu t}S_p(t)w\big)=e^{\mu t}S_p(t)(\mu-A_p)w=0$, so that $S_p(t)w=e^{-\mu t}w$. Conversely, if $w\in E_\mu$, then $t^{-1}(S_p(t)w-w)=t^{-1}(e^{-\mu t}-1)w\to-\mu w$ as $t\to0^+$, so that $w\in\dom(A_p)$ and $A_pw=\mu w$.

\underline{Step 2}. The semigroup $S_2$ is positive \cite[Corollary~4.3]{Ouhabaz}, and $S_p(t)=S_2(t)$ on the dense subspace $L^2(\mathscr D)\cap L^p(\mathscr D)$ of $L^p(\mathscr D)$, see Proposition~\ref{pr.extensionS}. Hence $S_p(t)$ is positive (approximate $0\le f\in L^p(\mathscr D)$ by $\min(f,n)\in L^2(\mathscr D)\cap L^p(\mathscr D)$ in $L^p(\mathscr D)$). It is therefore also real, i.e.\ it maps real-valued functions to real-valued functions (write $f=f^+-f^-$). Let now $t>0$ and $0\le f\in L^p(\mathscr D)$ with $f\not\equiv0$. For $n\in\mathbb N$, set $f_n:=\min(f,n)\in L^\infty(\mathscr D)\subset L^2(\mathscr D)\cap L^p(\mathscr D)$ (recall that $\mathscr D$ is bounded). For $n$ large enough, $f_n\not\equiv0$. Since $0\le f_n\le f$, by positivity of $S_p(t)$ and by \eqref{eq:irrL2},
$$
S_p(t)f\ \ge\ S_p(t)f_n=S_2(t)f_n>0\quad\text{a.e.\ in }\mathscr D.
$$

\end{proof}

We now  prove that $\lambda$ dominates the spectrum of $A_{p}$ for every $p\in(1,+\infty)$, i.e.\ $\mathrm{Re}(\mu)\ge\lambda$ for every $\mu\in\sigma(A_{p})$.
For  $p\in(1,+\infty)$, we  work with the complexification of $L^p(\mathscr D)$ (write $L^p(\mathscr D;,\mathbb C)$), to which $A_p$, $A_p^*$, $S^p(t):=e^{-tA_p}$ and $R_\kappa^p$ extend in the usual way ($T(f_1+if_2):=Tf_1+iTf_2$ for real $f_1,f_2$). Recall from Theorem \ref{th.3}: $\lambda\in\mathbb R$, $A_p\mathbf u=\lambda\mathbf u$, $A_p^*\mathbf v=\lambda\mathbf v$ for every $p\in(1,+\infty)$, with $\mathbf v>0$ a.e. in $\mathscr D$.

\begin{proposition}\label{pr.dominance}
Let the assumptions of Theorem \ref{th.2} hold, and let $\lambda$, $\mathbf u$, $\mathbf v$ be
given by Theorem \ref{th.3}.  
Then, for every $p\in(1,+\infty)$ and every eigenvalue $\mu\in\sigma(A_p)\cup \sigma(A_p^*) $,
$\mathrm{Re}(\mu)\ge\lambda$.  
\end{proposition}

\begin{proof}
Fix $p\in(1,+\infty)$.
%
%
By Lemma   \ref{le.comp},  $A_{p}$ has compact resolvent, hence its spectrum  consists entirely of isolated eigenvalues of finite multiplicity, which can only accumulate at infinity. Let $\mu\in\mathbb C$ be an eigenvalue of $A_p$ with eigenfunction $0\not\equiv
w\in\dom(A_p)$, $A_pw=\mu w$. Then from lemma \ref{le.Sp},
\begin{equation}\label{eq.semiw}
S_p(t)w=e^{-\mu t}w,\qquad\forall t\ge0.
\end{equation} 
  The same argument, applied to the eigenvalue equation
$A_p^*\mathbf v=\lambda\mathbf v$ from Theorem \ref{th.3}, gives
\begin{equation}\label{eq.semiv}
S_p^*(t)\mathbf v=e^{-\lambda t}\mathbf v,\qquad\forall t\ge0.
\end{equation} 
 Fix $t>0$. By \eqref{eq.semiw},
$|S_p(t)w|=e^{-t\,\mathrm{Re}(\mu)}|w|$ a.e.\ in $\mathscr D$. Combined with Lemma
\ref{lem:domination} applied to $T=S_p(t)$ (positive by Lemma \ref{le.Sp}) and $g=w$,
\begin{equation}\label{eq.comp}
e^{-t\,\mathrm{Re}(\mu)}\,|w|\ \le\ S_p(t)|w|\qquad\text{a.e.\ in }\mathscr D.
\end{equation}
Both sides of \eqref{eq.comp} belong to $L^p_+(\mathscr D)$. Multiplying by
$\mathbf v\in L^{p'}_+(\mathscr D)$ (Theorem \ref{th.3}) and integrating over $\mathscr D$
(all products lie in $L^1(\mathscr D)$ by H\"older's inequality) gives
\begin{equation}\label{eq.pair}
e^{-t\,\mathrm{Re}(\mu)}\int_{\mathscr D}|w|\,\mathbf v\,dx
\ \le\ \int_{\mathscr D}\big(S_p(t)|w|\big)\,\mathbf v\,dx
=\langle |w|,S_p^*(t)\mathbf v\rangle
\overset{\eqref{eq.semiv}}{=}e^{-\lambda t}\int_{\mathscr D}|w|\,\mathbf v\,dx.
\end{equation}
Since $|w|\not\equiv0$
and $\mathbf v>0$ a.e. by Theorem \ref{th.3}, we get for every
$t>0$:
$$
-t\,\mathrm{Re}(\mu)\ \le\ -\lambda t\quad\implies \quad \mathrm{Re}(\mu)\ \ge\ \lambda.
$$
This proves the result for $\sigma(A_p)$ and then for $\sigma(A^*_p)$ since $\sigma(A^*_p)= \overline{\sigma(A_p)}$.

\end{proof}

 \subsubsection{$A_p$ and $A_p^*$ are quasi-sectorial}
 
For  $\psi \in (0, \frac{\pi}{2}]$, define   the open sector
$$
\bm  \Sigma(\psi) := \{z \in \mathbb{C},\ z \neq 0 \text{ and } |\arg z| < \psi\}.
$$
We recall~\cite[Definition 1.44]{Ouhabaz}  that a strongly continuous semigroup $(T(t))_{t \ge 0}$ acting on a Banach space  $X$ is called a bounded holomorphic semigroup on the sector $\bm\Sigma(\psi)$ if $(T(t))_{t \ge 0}$ admits a holomorphic extension $(T(z))_{z \in \bm\Sigma(\psi)}$ such that for each $\theta \in (0, \psi)$, $(T(z))_{z \in \bm\Sigma(\theta)}$ is uniformly bounded and strongly continuous at $0$.  


  \begin{proposition}\label{pr.sectorial}
There exist $\kappa\ge0$ and $\psi\in(0,\frac\pi2]$ such that, for every $p\in(1,+\infty)$, the semigroup $(e^{-t(A_p+\kappa)})_{t\ge0}$ extends to a bounded holomorphic semigroup on the sector $\bm\Sigma(\psi)$. Equivalently, its generator $-(A_p+\kappa)$ is sectorial in $L^p(\mathscr D)$ in the sense of \cite[Definition II.4.1]{EnNa01} (i.e.\ $A_p$ is quasi-sectorial). In particular, there exists $\theta_0\in[0,\frac\pi2)$ such that $\sigma(A_p+\kappa)\subset\bar{\bm\Sigma}(\theta_0)$. The same holds for $A_p^*$ on $L^{p'}(\mathscr D)$, $p\in(1,+\infty)$.
\end{proposition}

\begin{proof}
Let $\omega$ be the constant in \eqref{eq.gaussian} and fix $\kappa>\max(\omega,\|b\|_\infty^2/2)$. Set $B_\kappa:=A_2+\kappa$, which is the operator associated with the form $\mathbf L_{\kappa}$.

\emph{Holomorphy on $L^2$.} By \cite[Definition 1.7]{Ouhabaz} and the estimate preceding Proposition~\ref{pr.domA2}, $\mathbf L_{\kappa}$ is sectorial with constant $C_{\kappa}$. Hence $B_\kappa$ is a sectorial operator \cite[Proposition 1.27]{Ouhabaz}. Moreover, $-1\in\rho(B_\kappa)$ \cite[Proposition 1.22]{Ouhabaz}, and $\mathrm{dist}(-1,\bm \Sigma(\arctan C_{\kappa}))>0$. By \cite[Theorem 1.54]{Ouhabaz}, $(e^{-zB_\kappa})$ is a contraction semigroup (in $L^2(\mathscr D)$), holomorphic on $\bm \Sigma(\psi)$ with $\psi:=\frac\pi2-\arctan C_{\kappa}$.

\emph{Gaussian bound.} The kernel of $e^{-tB_\kappa}$ is $e^{-\kappa t}p(t,x,y)$. Since $\kappa\ge\omega$, \eqref{eq.gaussian} gives $|e^{-\kappa t}p(t,x,y)|\le C\,t^{-d/2}e^{-|x-y|^2/(4\delta t)}$ for all $t>0$ and a.e.\ $(x,y)$.

By \cite[Theorem 6.16]{Ouhabaz}, $(e^{-tB_\kappa})_{t\ge0}$ extends to a bounded holomorphic semigroup on $\bm \Sigma(\psi)$ in $L^p(\mathscr D)$ for every $p\in(1,+\infty)$. By consistency, its restriction to real times is $(e^{-\kappa t}S_p(t))_{t\ge0}$, whose generator is $-(A_p+\kappa)$. 

The equivalence between bounded holomorphy on a sector and sectoriality is classical, see \cite[Definition II.4.1 and Theorem II.4.6]{EnNa01} applied to the generator $-(A_p+\kappa)$ (or \cite[Theorem 2.5.2]{Pazy}, which applies since, by \eqref{eq.Lqbound} and $\kappa>\omega$, the semigroup $(e^{-t(A_p+\kappa)})_{t\ge0}$ is uniformly bounded and $0\in\rho(A_p+\kappa)$). 

For $A_p^*$, apply the same argument to $B_\kappa^*=A_2^*+\kappa$. It is associated with the sectorial form $\mathbf L^*_{\kappa}$, which has the same constant since 
$$\mathbf L^*_{\kappa}(u,u)=\overline{\mathbf L_{\kappa}(u,u)}.$$
 Its kernel $\bar p(t,y,x) \,e^{-\kappa t}$ (recall \eqref{eq.pbar}) satisfies the same Gaussian bound. This gives the result  on $L^{p'}(\mathscr D)$, $p\in(1,+\infty)$.
 
\end{proof}

\subsubsection{The eigenvalue $\lambda$ is algebraically simple for $A_p$ and $A_p^*$}

\begin{proposition}\label{pr.simple}
Let the assumptions of Theorem~\ref{th.2} hold, and let $\lambda\in\mathbb R$, $\mathbf u$ and $\mathbf v$ be as in Theorem~\ref{th.3}. Then, for all $p\in(1,+\infty),$
$$
\ker(A_p-\lambda)=\mathrm{span}\{\mathbf u\}
\, \,  \text{ and } \, \, 
\ker(A_p^*-\lambda)=\mathrm{span}\{\mathbf v\}.
$$ Moreover, $\lambda$ is algebraically simple, both for $A_p$ and for its adjoint $A_p^*$: for all $p\in(1,+\infty),$
\begin{enumerate}
\item[$\mathfrak 1$.] For every $k\ge1$, $\ker ((A_p-\lambda)^k )=\mathrm{span}\{\mathbf u\}$ and $\ker\big((A_p^*-\lambda)^k\big)=\mathrm{span}\{\mathbf v\}$;
\item[$\mathfrak 2$.] The eigenvalue $\lambda$ is an isolated point of $\sigma(A_p)=\sigma(A_p^*)$, and the Riesz projections $P_p$ of $A_p$ and $P_p^*$ of $A_p^*$ at $\lambda$ have rank one. More precisely,
$$
P_pf=\frac{\int_{\mathscr D}f\,\mathbf v}{\int_{\mathscr D}\mathbf u\,\mathbf v}\ \mathbf u,\quad f\in L^p(\mathscr D),
\qquad
P_p^*g=\frac{\int_{\mathscr D}g\,\mathbf u}{\int_{\mathscr D}\mathbf u\,\mathbf v}\ \mathbf v,\quad g\in L^{p'}(\mathscr D).
$$
\end{enumerate}
\end{proposition}

\begin{proof}
Fix $p\in(1,+\infty)$. Recall that $\mathbf u>0$ a.e.\ in $\mathscr D$, that $\mathbf v>0$ a.e.\ in $\mathscr D$, and that $\mathbf v\in C(\bar{\mathscr D})\subset L^\infty(\mathscr D)$. In particular $\int_{\mathscr D}\mathbf u\,\mathbf v\in\mathbb R_+^*$.

\underline{Step 1}. { Duality with $\mathbf v$.} For every $f\in L^p(\mathscr D)$ and $t\ge0$,
\begin{equation}\label{eq.dualv}
\int_{\mathscr D}\big(S_p(t)f\big)\,\mathbf v=e^{-\lambda t}\int_{\mathscr D}f\,\mathbf v .
\end{equation}
For $p\in(1,+\infty)$: since $\mathbf v\in\dom(A_p^*)$ and $A_p^*\mathbf v=\lambda\mathbf v$, Lemma \ref{le.Sp} applied to $S_p^*$ gives $S_p^*(t)\mathbf v=e^{-\lambda t}\mathbf v$. Since $S_p^*(t)$ is the Banach-space adjoint of $S_p(t)$ (Proposition~\ref{pr.extensionS}) and $\mathbf v$, $\lambda$ are real, \eqref{eq.dualv} follows.

\underline{Step 2}. Reduction to real-valued eigenvectors.

Let $w\in\ker(A_p-\lambda)=E_\lambda$. Since $S_p(t)$ is real (Lemma \ref{le.Sp}) and $\lambda\in\mathbb R$, taking real and imaginary parts in $S_p(t)w=e^{-\lambda t}w$ shows that $\mathrm{Re}\,w$ and $\mathrm{Im}\,w$ belong to $E_\lambda$. It is thus enough to prove that every real-valued $w\in E_\lambda$ is a real multiple of $\mathbf u$.

\underline{Step 3}. Every real-valued $w\in E_\lambda$ has a constant sign.

Let $w\in E_\lambda$ be real-valued. Since $-|w|\le w\le|w|$ and $S_p(t)$ is positive, $|S_p(t)w|\le S_p(t)|w|$. Hence, for every $t\ge0$,
$$
g_t:=S_p(t)|w|-e^{-\lambda t}|w|\ \ge\ |S_p(t)w|-e^{-\lambda t}|w|=0 .
$$
By \eqref{eq.dualv} with $f=|w|$, $\int_{\mathscr D}g_t\,\mathbf v=0$. Since $g_t\ge0$ and $\mathbf v>0$ a.e. in $\mathscr D$, $g_t=0$ a.e. in $\mathscr D$. Thus $|w|\in E_\lambda$, and therefore $w^\pm=\frac12(|w|\pm w)\in E_\lambda$. Now, if $0\le h\in E_\lambda$ and $h\not\equiv0$, then $h=e^{\lambda}S_p(1)h>0$ a.e. in $\mathscr D$ by Lemma \ref{le.Sp}. Since $w^+w^-=0$ a.e., $w^+$ and $w^-$ cannot be both $\not\equiv0$. Hence $w\ge0$ a.e.\ or $w\le0$ a.e. in $\mathscr D$. 

\underline{Step 4}. Conclusion for $A_p$. 

Let $w\in E_\lambda$ be real-valued, and set $c:=\int_{\mathscr D}w\,\mathbf v\big/\int_{\mathscr D}\mathbf u\,\mathbf v\in\mathbb R$. By Theorem~\ref{th.3}, $\mathbf u\in E_\lambda$, so $z:=w-c\,\mathbf u$ is a real-valued element of $E_\lambda$ such that $\int_{\mathscr D}z\,\mathbf v=0$. By Step~3, $z$ has a constant sign, and since $\mathbf v>0$ a.e., we get $z=0$, i.e.\ $w=c\,\mathbf u$. By Step~2, $\ker(A_p-\lambda)=\mathrm{span}\{\mathbf u\}$, and $\dim\ker(A_p-\lambda)=1$ since $\mathbf u\not\equiv0$.

\underline{Step 5}. The dual operator.

Let $p\in(1,+\infty)$. The proof is the same, with the roles of $(S_p,\mathbf u)$ and $(S_p^*,\mathbf v)$ exchanged. We only check that Lemma \ref{le.Sp} still holds. The first result holds for any $C_0$-semigroup. The operator $S_p^*(t)$ is positive and real, as the adjoint of a positive operator. It is positivity improving: let $t>0$ and $0\le g\in L^{p'}(\mathscr D)$ with $g\not\equiv0$, and assume that $B:=\{S_p^*(t)g=0\}$ has positive measure. Then, by Lemma \ref{le.Sp},
$$
0=\int_{\mathscr D}\mathbf 1_B\,S_p^*(t)g=\int_{\mathscr D}\big(S_p(t)\mathbf 1_B\big)\,g>0,
$$
a contradiction. Finally, since $\mathbf u\in\dom(A_p)$ and $A_p\mathbf u=\lambda\mathbf u$, we have $S_p(t)\mathbf u=e^{-\lambda t}\mathbf u$, hence $\int_{\mathscr D}(S_p^*(t)g)\,\mathbf u=e^{-\lambda t}\int_{\mathscr D}g\,\mathbf u$ for every $g\in L^{p'}(\mathscr D)$, and $\mathbf u>0$ a.e. in $\mathscr D$. Steps~2, 3, and 4 then give $\ker(A_p^*-\lambda)=\mathrm{span}\{\mathbf v\}$.

\underline{Step 6}. Proof of item $\mathfrak 1$.

Let $p\in(1,+\infty)$. Since $\ker(A_p-\lambda)\subset\ker ((A_p-\lambda)^k )$ for all $k\ge1$, and by induction on $k$, it is enough to show that $\ker ((A_p-\lambda)^2 )\subset\ker(A_p-\lambda)$. Let $w\in\dom(A_p)$ be such that $(A_p-\lambda)w\in\dom(A_p)$ and $(A_p-\lambda)^2w=0$. Then $(A_p-\lambda)w\in\ker(A_p-\lambda)=\mathrm{span}\{\mathbf u\}$, i.e.\ $(A_p-\lambda)w=c\,\mathbf u$ for some $c\in\mathbb C$. By \cite[Theorem 1.2.4]{Pazy} and since $S_p(t)\mathbf u=e^{-\lambda t}\mathbf u$, we have for all $t\ge0$
$$
\frac{d}{dt}\big(e^{\lambda t}S_p(t)w\big)=e^{\lambda t}S_p(t)(\lambda-A_p)w=-c\,e^{\lambda t}S_p(t)\mathbf u=-c\,\mathbf u,
\quad\text{so that}\quad
e^{\lambda t}S_p(t)w=w-c\,t\,\mathbf u .
$$
Integrating against $\mathbf v$ and using \eqref{eq.dualv} (which holds for complex-valued $f$ as well, by linearity), we get
$\int_{\mathscr D}w\,\mathbf v=\int_{\mathscr D}w\,\mathbf v-c\,t\int_{\mathscr D}\mathbf u\,\mathbf v$ for all $t\ge0$. Since $\int_{\mathscr D}\mathbf u\,\mathbf v>0$, we get $c=0$, i.e.\ $w\in\ker(A_p-\lambda)$. The same argument applies to $A_p^*$, $p\in(1,+\infty)$, with the roles of $(S_p,\mathbf u)$ and $(S_p^*,\mathbf v)$ exchanged.

\underline{Step 7}. Proof of item $\mathfrak 2$.

Let $p\in(1,+\infty)$. Since $A_p$ has compact resolvent (Lemma \ref{le.comp}), $\lambda$ is an isolated point of $\sigma(A_p)$ and the Riesz projection
$$
P_p:=\frac1{2i\pi}\oint_{|z-\lambda|=r}(z-A_p)^{-1}\,dz
$$
(with $r>0$ small enough) has finite rank \cite[Theorem III.6.29]{Ka66}. Moreover, $\mathrm{Ran}\,P_p\subset\dom(A_p)$ is invariant under $A_p$, and the restriction $A_p|_{\mathrm{Ran}\,P_p}$ has spectrum $\{\lambda\}$ \cite[Theorem III.6.17]{Ka66}. Since $\mathrm{Ran}\,P_p$ has finite dimension $m$, the Cayley-Hamilton theorem gives $(A_p-\lambda)^m=0$ on $\mathrm{Ran}\,P_p$, so that, by $i)$, $\mathrm{Ran}\,P_p\subset\ker ((A_p-\lambda)^m )=\mathrm{span}\{\mathbf u\}$. Conversely, $\mathbf u\in\mathrm{Ran}\,P_p$ since $(z-A_p)^{-1}\mathbf u=(z-\lambda)^{-1}\mathbf u$ for $z\ne\lambda$ close to $\lambda$. Hence $\mathrm{Ran}\,P_p=\mathrm{span}\{\mathbf u\}$, so $P_pf=\ell(f)\,\mathbf u$ for some $\ell\in (L^p(\mathscr D))^*$. Finally, since $ ((z-A_p)^{-1} )^*=(\bar z-A_p^*)^{-1}$ and $\lambda\in\mathbb R$, $P_p^*$ is the Riesz projection of $A_p^*$ at $\lambda$, and $A_p^*$ has compact resolvent (Theorem~\ref{th.3}). By the same argument, $\mathrm{Ran}\,P_p^*=\mathrm{span}\{\mathbf v\}$. Since $P_p^*g=\ell\,\int_{\mathscr D}g\,\mathbf u$ for $g\in L^{p'}(\mathscr D)$ (where we identify $\ell$ with an element of $L^{p'}(\mathscr D)$), we get $\ell=a\,\mathbf v$ for some $a\in\mathbb C$, and $P_p\mathbf u=\mathbf u$ gives $a=1/\int_{\mathscr D}\mathbf u\,\mathbf v$.

\end{proof}

\subsubsection{Integrability of the normal derivative of $\mathbf{v}$}
\begin{proposition}\label{th.Brown}
Let the assumptions of Theorem \ref{th.2} hold and consider $\mathbf v$ as in Theorem \ref{th.3}.
Then, there exists $q>1$  such that  
$$
\partial_{\mathbf n_e}  \mathbf v \in  L^q(\partial\mathscr D). 
$$
\end{proposition}

\begin{proof}
We aim at using the results of \cite{OtBr13}. However, to apply the approach of \cite{OtBr13}, we must slightly adapt their arguments to account for the fact that $\mathbf{v}$ satisfies the inhomogeneous equation ${-}\mathscr{L}^*\mathbf{v} = \lambda \mathbf{v}$ instead of the harmonic equation $\Delta \mathbf{v} = 0$. 
Set $
  F := \Delta \mathbf v = {-\lambda \mathbf v + b\cdot\nabla \mathbf v +
  (\operatorname{div}b)\,\mathbf v} \ \in L^2(\mathscr D)$.  
Since   ${F}\in L^2(\mathscr D)$, there is  a unique solution $w$ to
$$
  w\in H^2(\mathscr D)\cap H^1_0(\mathscr D), \qquad
  \Delta w = F \ \text{ in }\mathscr D, \qquad w=0 \ \text{ on }\partial\mathscr D,
$$
with $\|w\|_{H^2(\mathscr D)}\le C\|F\|_{L^2(\mathscr D)}$ and thus $\|\nabla w\|_{L^2(\partial\mathscr D)}\le C\|F\|_{L^2(\mathscr D)}$.  Note again that since $F$ is smooth over $\bar{\mathscr D}\setminus \partial \Sigma$, $w \in C^\infty(\bar{\mathscr D}\setminus \partial \Sigma)$ (again by standard interior and local Dirichlet boundary elliptic regularity results).
 
For a function $h : \mathscr D \to \mathbb R$, the non-tangential maximal function $h^\star$ is defined by
\bes
h^\star(x)= \sup_{y\in \Gamma_\alpha(x)} | h(y)|, \ x\in \partial\mathscr D,
\ees
where for $\alpha>0$, $
\Gamma_\alpha(x)= \{y\in\mathscr D, \ |x-y|\le (1+\alpha) \operatorname{dist}(y,\partial\mathscr D)\}$. 
For different values of $\alpha$, the corresponding non-tangential maximal functions have equivalent $L^p$ norms,  hence the dependence on $\alpha$ is typically omitted.  
Set $\zeta:=\mathbf v-w$. Then,
$$\Delta \zeta =\Delta\mathbf v-\Delta w=0 \text{ in } \mathscr D \ \  \text{ and }  \ \ 
  \zeta |_{\SigD}=\mathbf v|_{\SigD}-w|_{\SigD}=0.
$$
On the other hand, using the Robin condition
$\partial_{\mathbf n_e}\mathbf v={(b\cdot\mathbf n_e)\mathbf v}$ from
Proposition~\ref{prop:domAstar-L2}, we have that: 
$$
  \partial_{\mathbf n_e}\zeta |_{\SigN}
  = \partial_{\mathbf n_e}\mathbf v|_{\SigN}-\partial_{\mathbf n_e}w|_{\SigN}
  = {(b\cdot\mathbf n_e)\,\mathbf v|_{\SigN}}-\partial_{\mathbf n_e}w|_{\SigN}.
$$
Since $\mathbf v\in C^{0,1/2-\varepsilon}(\bar{\mathscr D})$ (Proposition~\ref{pr.conti}) and $b\in L^\infty(\mathscr D)$, the term
$(b\cdot\mathbf n_e)\mathbf v|_{\SigN}$ is bounded, hence lies in every
$L^p(\SigN)$. Therefore,  because $\partial_{\mathbf n_e}w|_{\SigN}\in
L^2(\SigN)$, we have $\partial_{\mathbf n_e}\zeta |_{\SigN}\in
L^2(\SigN)$.    
 As $\zeta$ solves
$\Delta \zeta =0$ in $\mathscr D$ with $\zeta |_{\SigD}=0$ and $\partial_{\mathbf
n_e}\zeta |_{\SigN} \in L^2(\SigN)$, \cite[Theorem~1.2(c)]{OtBr13} implies,
for some  $p>1$:
$$
  \|(\nabla \zeta )^\star \|_{L^p(\partial\mathscr D)} < +\infty.
$$
Note that $\partial_{\mathbf n_e}\zeta |_{\SigN}$ is smooth and consequently,  
 $\zeta \in C^\infty(\bar{\mathscr D}\setminus \partial \Sigma)$.
In addition, we clearly have for all $x\in \partial \mathscr D\setminus \partial \Sigma$: 
$$|\nabla  \zeta |(x)\le  |(\nabla   \zeta )^\star|(x).$$
Hence, since $\partial \Sigma$ has zero Lebesgue measure in $\partial \mathscr D$ (recall $\partial \Sigma$ is $C^3$ by assumption), 
$$|\nabla  \zeta |\le  |(\nabla  \zeta )^\star| \text{ a.e. in } \partial \mathscr D.$$
Since, for every $x\in \partial \mathscr D\setminus \partial \Sigma$ (and thus a.e. in $ \partial  \mathscr D$),  $\nabla\mathbf v(x)=\nabla \zeta (x)+\nabla w(x)$, we have
$|\nabla\mathbf v(x)| \le |\nabla \zeta (x)| +|\nabla w(x)| \le  |(\nabla \zeta )^\star(x)| +|\nabla w(x)| $. In conclusion, with $q = \min(2,p)$, we have that:
$$
  \| \nabla  \mathbf v  \|_{L^q(\partial\mathscr D)}
  \le \|(\nabla \zeta )^\star \|_{L^q(\partial\mathscr D)} + \| \nabla w  \|_{L^q(\partial\mathscr D)}
  < +\infty.
$$

\end{proof}


\section{The probabilistic-analytic correspondence: proof of Theorem \ref{th.CorFeyn}}

As for the previous section, we assume throughout that $\mathscr D\subset\mathbb R^d$ is a $C^\infty$ bounded domain of $\mathbb R^d,$ $d\geq2$, and that $\Sigma$ is a nonempty $C^3$ open subset of $\partial\mathscr D$.

We note that in this  setting, Assumptions  \ass{A1}, \ass{A2} and \ass{C} are satisfied (see Remark~\ref{re.sigma}). As a result, there exists a unique quasi-stationary distribution  $\bm{\rho}$ for    the process $(X_t)_{t\ge 0}$ with absorbing boundary part $\bar \Sigma$  in  $\mathscr P(\mathscr O)$ (see Theorem \ref{th.1} and Definition \ref{QSD}). 

\subsection{Feynman-Kac formula and quasi-stationary distribution}
%

\begin{proposition}\label{lem:sgreg}
 Let the assumptions of Theorem \ref{th.2} hold.   Fix $p>\max(2(d-2),4)$ and  
let  $f \in \dom(A_p)$.   Set $
  w(t) := e^{-tA_p}f$, $ t \geq 0$. 
Then:
\begin{enumerate} 
  \item[$\mathfrak 1$.] $w \in C^1(\mathbb R_+,\,L^p(\mathscr D)) \cap C(\mathbb R_+,\,\dom(A_p))$,
        with $\partial_t w(t) = -A_p w(t)$ in $L^p(\mathscr D)$ for all $t \geq 0$.
  \item[$\mathfrak 2$.] $\dom(A_p) \hookrightarrow C(\bar{\mathscr D})$ continuously,
        and consequently
$
          w \in C (\mathbb R_+,\,C(\bar{\mathscr D}) ).
$
  \item[$\mathfrak 3$.] $(t,x) \mapsto w(t,x)$ is jointly continuous on
        $\mathbb R_+ \times \bar{\mathscr D}$.
        \end{enumerate} 
\end{proposition}

\begin{proof}
The first item is a consequence of the fact that  $f \in \dom(A_p)$. In addition $\dom(A_p) \hookrightarrow C(\bar{\mathscr D})$ when $p>\max(2(d-2),4)$ is  a consequence of   \cite[Theorem 5.3]{Shamir1968} (together with the Sobolev embedding $W^{s,p}(\mathscr D)\hookrightarrow C(\bar{\mathscr D})$, $s >  {d}/{p}$). Therefore,  $w \in C  (\mathbb R_+,\,C(\bar{\mathscr D}) )$.  
The joint continuity  of $(t,x)\mapsto w(t,x)$ on $\mathbb R_+\times\bar{\mathscr D}$
follows immediately from the fact that $w \in C  (\mathbb R_+,C(\bar{\mathscr D}) )$.

\end{proof}

Recall the definition of the process $(X_t)_{t\ge 0}$ in \eqref{eq.LR} and recall also \eqref{eq.Oo}.  Theorem \ref{th.3} together with the following proposition imply all the assertions of  Theorem \ref{th.2}, except the last one on the exit law starting from the quasi-stationary distribution, which is treated in the next section.


\begin{proposition}\label{pr.fin}
 Let the assumptions of Theorem \ref{th.2} hold.   Fix $p>\max(2(d-2),4)$ and  
let  $f \in \dom(A_p)$. Set $
  w(t) := e^{-tA_p}f$, $ t \geq 0$. 
  For all $(t,x)\in \mathbb R_+ \times \mathscr D$, the following Feynman-Kac formula holds true:
$$
w(t,x) = P_t^{\mathscr O} f(x) = \mathbb{E}_x [ f(X_t) \mathbf{1}_{t < \tau_{\mathscr O}}  ].
$$
Moreover, 
$\nu (dx):=(\mathbf v (x)/ \int_{\mathscr D} \mathbf v) dx$ is a quasi-stationary distribution for the process $(X_t)_{t\ge 0}$ killed when exiting $\bar{\mathscr D}$ through $\SigD=\bar {\Sigma}$ (see Definition \ref{QSD}), where $\mathbf v$ is given by Theorem \ref{th.3}. Using Theorem \ref{th.1}, it is unique and $\nu = \bm{\rho}$. In addition,   $\lambda$ (given by Theorem \ref{th.3}) is positive and $\lambda =\eta$, where $\eta$ is given by Theorem \ref{th.1}. 
\end{proposition}

This proposition proves items $\mathfrak 1$ and $\mathfrak 2$ of Theorem \ref{th.CorFeyn}. Furthermore, since $\mathbf u$ is continuous up to the boundary (see Theorem \ref{th.3}) and $C^\infty$  on $\bar{\mathscr D} \setminus \partial \Sigma$, it is straightforward  with the It\=o formula and the same arguments we use in the proof of Proposition \ref{pr.fin} below,  to show that 
 $P_t^{\mathscr O} \mathbf u(x) = e^{-\lambda t}  \mathbf u(x)$ for all $t\ge 0$ and $x\in \bar{\mathscr D}$. 
Thus, by  Theorem \ref{th.1},  we have by uniqueness:
\begin{equation}\label{eq.Uu}
\frac{\mathbf{u}}{\bm{\rho}(\mathbf u)}  = \bm \varphi.
\end{equation}
In particular,  $\bm \varphi$ admits an Hölder  continuous extension up to the boundary. This proves item $\mathfrak 3$ of Theorem \ref{th.CorFeyn}.

\begin{proof}
We start with the Feynman-Kac formula whose key ingredient is the fact that $w$ is jointly continuous. 
Let
$
V_r = \{ x \in \bar{\mathscr D},  \operatorname{dist}(x, \SigD) \le  r \},
 $
and define
 $
\tau_r = \inf\{ t \geq 0,  X_t \in V_r \}.
 $
Then,
 $
\tau_r \uparrow \tau_{\mathscr O} $ a.s. as $ r \downarrow 0$. 
Now fix $r > 0$. Since $\bar{\mathscr D}\setminus V_{r/2}$ is at a positive distance away from the interface $\partial \SigD$, the mixed boundary condition is locally either purely Dirichlet or purely Neumann, hence by   standard interior and local Dirichlet/Neumann boundary elliptic regularity for  parabolic problems,
$$
w \in C^{\infty}\bigl({(0,T]} \times (\bar{\mathscr D} \setminus V_{{r/2}})\bigr)\
\text{ and } \ {\partial_t w(s,y) = -A_pw(s,y)=\mathscr Lw(s,y), \ \forall s> 0,\ y\in \bar{\mathscr D} \setminus V_{r/2}},$$
{where we recall that $A_p$ is the realization of $-\mathscr L$, $\mathscr L=\Delta+b\cdot\nabla$ being the \rev{differential operator associated with} \eqref{eq.LR}.} {Fix $t>0$ and $\varepsilon\in(0,t)$.} Therefore It\=o's formula {applied to $s\mapsto w(t-s,X_s)$} is perfectly legitimate up to $\theta:=(t-\varepsilon)\wedge\tau_r$:
\begin{align*}
{w(t -\theta,\, X_{\theta})}
&= {w(t,x)}
+ \int_0^{{\theta}} {(-\partial_t w + \mathscr Lw)(t-s, X_s)}\, ds
+ {\sqrt2}\int_0^{{\theta}} \nabla w({t-s},X_s) \cdot \, dB_s\\
&\quad
{-} \int_0^{{\theta}} \partial_{{\mathbf n_e}}  w({t-s},X_s) \, d{|L|_s}.
\end{align*}
Now, we notice that:
\begin{itemize}
    \item[$\mathfrak 1$.] The drift integral vanishes because ${\partial_t w(s,y) = \mathscr Lw(s,y)}$ for all $s{>} 0$ and  ${y}\in \bar{\mathscr D} \setminus V_r$.
    \item[$\mathfrak 2$.] The local-time integral vanishes because before $\tau_r$ the process can only hit the reflecting part $\partial\mathscr  D \setminus V_r$, where $\partial_{\mathbf n_e} w = 0$ pointwise.
    \item[{$\mathfrak 3$.}] {The stochastic integral has zero mean since $\nabla w$ is bounded on $[\varepsilon,t]\times(\bar{\mathscr D}\setminus V_r)$.}
\end{itemize}
{Taking expectations, we get
$w(t,x)=\mathbb E_x\big[w(t-\tau_r,X_{\tau_r})\mathbf 1_{\tau_r\le t-\varepsilon}+w(\varepsilon,X_{t-\varepsilon})\mathbf 1_{t-\varepsilon<\tau_r}\big]$. Letting $\varepsilon\to0^+$, using that $w$ is bounded and jointly continuous on $[0,t]\times\bar{\mathscr D}$ (Proposition~\ref{lem:sgreg}), $w(0,\cdot)=f$, and the dominated convergence theorem,}
 we have that:
$$
w(t, x) = \mathbb{E}_x\!\left[
  w(t - \tau_r,\, X_{\tau_r})\,\mathbf{1}_{\tau_r \le t}
  + f(X_t)\,\mathbf{1}_{t < \tau_r}
\right]
$$
{(on $\{\tau_r=t\}$, one has $w(t-\tau_r,X_{\tau_r})=w(0,X_t)=f(X_t)$).} 
Recalling that $\tau_r\to\tau_\mathscr O$, the whole issue reduces to proving that as $r\to 0$:
$
\mathbb{E}_x\!\left[
  w(t - \tau_r,\, X_{\tau_r})\,\mathbf{1}_{\tau_r \le  t}
\right] \to 0.
$
 Note that $\mathbb{E}_x\!\left[
  w(t - \tau_r,\, X_{\tau_r})\,\mathbf{1}_{\tau_r \le  t}
\right]=\mathbb{E}_x\!\left[
  w(t - \tau_r,\, X_{\tau_r})\,\mathbf{1}_{\tau_r <  t}
\right]$, since $\mathbb P_x[\tau_r =  t]\le \mathbb P_x[X_t\in \partial V_r]=0$ (recall $X_t$ has a density).  Since $(t,x) \mapsto w(t,x)$ is jointly continuous on
        $\mathbb R_+ \times \bar{\mathscr D}$  (see Proposition \ref{lem:sgreg}), we can apply the dominated convergence theorem to deduce  that 
$$\mathbb{E}_x [
  w(t - \tau_r,\, X_{\tau_r})\,\mathbf{1}_{\tau_r \le  t}
 ] \to \mathbb{E}_x [
  w(t -\tau_{\mathscr O},\, X_{\tau_{\mathscr O}})\,\mathbf{1}_{\tau_{\mathscr O} \le  t}
 ]  =0$$
  since $X_{\tau_{\mathscr O}}\in \SigD$ a.s. and $w(\cdot, x)=0$ for $x\in \SigD$. This shows that for all $f \in \dom(A_p)$ and every $(t,x)\in \mathbb R_+ \times \mathscr D$,  
 $$
w(t,x) = P_t^{\mathscr O} f(x).$$
Let us now show that $\nu$ is a quasi-stationary distribution   for  the process 
$(X_t)_{t\ge 0}$ with absorbing boundary part $\SigD$. 
Let $f \in \dom(A_p)$ and fix $p > \max(2(d-2),4)$. We have thanks to the Feynman-Kac formula proved before:
\begin{align*}
\mathbb{E}_{\nu} [ f(X_t) \mathbf{1}_{t < \tau_{\mathscr O}}  ]= \int_{\mathscr D} \mathbb{E}_x [ f(X_t) \mathbf{1}_{t < \tau_{\mathscr O}}  ] \nu(dx)
&=\frac1{\int_{\mathscr D} \mathbf v } \int_{\mathscr D}  e^{-tA_p}f(x) \, \mathbf v (x)\, dx \\
&=\frac1{\int_{\mathscr D} \mathbf v } \int_{\mathscr D}   f(x) \, e^{-tA_{p}^*} \mathbf v (x)\, dx\\
&= e^{-\lambda t}\frac1{\int_{\mathscr D} \mathbf v} \int_{\mathscr D}   f(x)   \mathbf v (x)\, dx,
\end{align*}
where we have used that, because $\mathbf v \in \dom (A_{p}^*)$ (see Proposition \ref{pr:allp}) and $A_p^*\mathbf v=\lambda \mathbf v$ (see Theorem \ref{th.3}), it holds that:
$$e^{-tA_{p}^*} \mathbf v (x)= e^{-\lambda t}\mathbf v (x)$$ 
as shown after \eqref{eq.dualv}. Hence $\nu P_t^{\mathscr O} f= \mathbb{E}_{\nu } [ f(X_t) \mathbf{1}_{\{t < \tau_{\mathscr O}\}}  ]= e^{-\lambda t} \nu (f)$ for all $t\ge 0$ and $f \in \dom(A_p)$. 

Let us complete the proof by identifying $e^{-\lambda t}$ with $\mathbb P_\nu[t<\tau_{\mathscr O}]$. Let $f_n \in C^\infty_c(\mathscr D)\subset \dom(A_p)$ be such that $0\le f_n\le 1$ and $f_n\uparrow \mathbf 1$ pointwise on $\mathscr D$. Applying the identity above to $f_n$ gives:
\[
\mathbb E_\nu\bigl[f_n(X_t)\mathbf 1_{t<\tau_{\mathscr O}}\bigr] = e^{-\lambda t}\,\nu(f_n), \ \forall n.
\]
On the right-hand side, since $\nu$ is a probability measure and $f_n\uparrow \mathbf 1$, the monotone convergence theorem yields $\nu(f_n)\to\nu(\mathbf 1)=1$, so that the right-hand side converges to $e^{-\lambda t}$. 
On the left-hand side, on the event $\{t<\tau_{\mathscr O}\}$ we have $X_t(x)\in\mathscr O=\bar{\mathscr D}\setminus \bar{\Sigma}$. Since $X_t$ has a density with respect to the Lebesgue measure for every $t>0$ and $x\in \bar{\mathscr D}$ (see \rev{Proposition~\ref{th.SF2}, which applies since the assumptions of Theorem~\ref{th.2} imply \ass{B}}), we have that $\mathbb P_\nu[X_t\in\partial\mathscr D]=0$ for every $t>0$. Thus,   $X_t\in\mathscr D$ $\mathbb P_\nu$-a.s.\ on $\{t<\tau_{\mathscr O}\}$, and consequently $f_n(X_t)\to 1$ $\mathbb P_\nu$-a.s.\ on this event. Since $0\le f_n\le1$, the dominated convergence theorem  gives
$$
\mathbb E_\nu\bigl[f_n(X_t)\mathbf 1_{t<\tau_{\mathscr O}}\bigr] = \mathbb E_\nu\bigl[f_n(X_t)\mathbf 1_{t<\tau_{\mathscr O}}\mathbf 1_{\{X_t\in\mathscr D\}}\bigr]  \to  \mathbb E_\nu\bigl[\mathbf 1_{t<\tau_{\mathscr O}}\bigr] = \mathbb P_\nu[t<\tau_{\mathscr O}].
$$
Letting $n\to+\infty$ in the identity above therefore yields
$$
\mathbb P_\nu[t<\tau_{\mathscr O}] = e^{-\lambda t}, \qquad \forall\, t\ge0.
$$
Therefore, $\nu$ is a  
 quasi-stationary distribution   for the reflected diffusion
$(X_t)_{t\ge 0}$ with absorbing boundary part $\bar \Sigma$  (see Definition \ref{QSD}). By uniqueness (Theorem \ref{th.1}), it is the unique one: $\nu=\bm \rho$. 
In particular  $\lambda=\eta$   and consequently $\lambda$ is positive (where $\lambda$ is given by   Theorem \ref{th.3} and $\eta>0$ by Theorem \ref{th.1}). 
This proves the desired claim. 

 \end{proof}
 

 

\subsection{On the  exit position distribution}
 
\begin{lemma}\label{le.regmix}
Let $u \in H^1(\mathscr D)$ be such that $\mathscr L u \in L^p(\mathscr D)$ for every
$p\in(1,+\infty)$, $u|_{\SigD} \in C^\infty(\SigD)$,  and $\partial_{\mathbf n_e} u|_{\SigN}=0$ in $H^{-1/2}(\SigN)$. Then
$u \in C(\bar{\mathscr D})$ and $u \in C^\infty(\bar{\mathscr D} \setminus \partial \Sigma)$.
\end{lemma}

\begin{proof}
Applying \cite[Theorem 5.3]{Shamir1968} together with  the Sobolev embedding  yields $u \in C(\bar{\mathscr{D}})$. The higher regularity $u \in C^\infty(\bar{\mathscr D} \setminus \partial \Sigma)$ then follows from standard interior and local Dirichlet/Neumann boundary elliptic regularity results (recall $b:\bar{\mathscr D}\to \mathbb R^d$ is smooth as well as $\partial \mathscr D$).

\end{proof}

\begin{proposition} \label{pr.Phih}
Let $h \in C^\infty(\bar{\mathscr D},\mathbb R)$ and set, for $x\in \bar{\mathscr D}$,
$$
  \Phi_h(x) := \mathbb E_x\bigl[h(X_{\tau_{\mathscr O}})\bigr].
$$
Then, $\Phi_h$ is well defined and bounded on $\bar{\mathscr D}$, and:
\begin{enumerate}
\item[$\mathfrak 1$.] $\Phi_h \in H^1(\mathscr D) \cap C(\bar{\mathscr D})$, $\mathscr L \Phi_h = 0$
in $\mathscr D$ (in the sense of distributions).
\item[$\mathfrak 2$.]  $\Phi_h = h$ on $\SigD$ (trace sense), $\partial_{\mathbf n_e} \Phi_h|_{\SigN} = 0$
in $H^{-1/2}(\SigN)$.
\item[$\mathfrak 3$.]  $\Phi_h \in C^1(\bar{\mathscr D} \setminus \partial \Sigma)$. \end{enumerate}
\end{proposition}

\begin{proof}
Our objective is to construct a solution $\Psi_h$ to the following non-homogeneous mixed boundary value problem:$$\mathscr L \Psi_h=0 \quad \text{in }\mathscr D, \qquad \partial_{\mathbf n_e}\Psi_h\big\vert _{\SigN} = 0, \qquad \Psi_h\big\vert _{\SigD}=h,$$with sufficient regularity \rev{(recall  that there is an optimal regularity threshold that cannot be exceeded, see the discussion in Section~\ref{sec.St2})} to rigorously apply It\=o's formula as we did before, to deduce $\Psi_h=  \Phi_h$.

\underline{Step 1}. Injectivity of the operator $A_2$.

 We first check that $A_2$ has trivial
kernel. Let $\psi\in\dom(A_2)$ satisfy $A_2\psi=0$, i.e.\ ${-}\mathscr L\psi=0\in 
L^q(\mathscr D)$ for every $q\in(1,+\infty)$ and $\partial_{\mathbf n_e}\psi|_{\SigN}
=0$ in $H^{-1/2}(\SigN)$ (see Proposition \ref{pr.domA2}). By Lemma~\ref{le.regmix},  
$$\psi\in C(\bar{\mathscr D}) \cap  C^\infty(\bar{\mathscr D} \setminus \partial \Sigma).$$ 
In particular  $\partial_{\mathbf n_e}\psi 
=0$ pointwise on $\SigN$.  
Fix $x\in\mathscr D$ and $r>0$. By It\=o formula   on $\psi(X_{t})$, we get using the boundary condition  
$$
  \psi(x) = \mathbb E_x\bigl[\psi(X_{\tau_r})\bigr], \quad
  \forall\, r>0,
$$
where we recall that $\tau_r = \inf\{ t \geq 0,  X_t \in V_r \}$ and $\tau_r\to\tau_\mathscr O$ as $r\to0$. 
 By dominated  convergence theorem and because $\psi\in C(\bar{\mathscr D})$, we then deduce  that:
$$
  \psi(x) = \mathbb E_x\bigl[\psi(X_{\tau_{\mathscr O}})\bigr].
$$
Finally, because a.s. $X_{\tau_{\mathscr O}}\in\SigD$ 
and $\psi|_{\SigD}=0$ pointwise (since $\psi\in\dom(A_2)\subset\HsD$ and $\psi\in C(\bar{\mathscr D})$), we get that  
$\psi(x)=\mathbb E_x[0]=0$. As $x\in\mathscr D$ was arbitrary, $\psi\equiv0$, i.e.\
$\ker A_2 = \{0\}$, the desired result.

\smallskip\noindent
\underline{Step 2}. Variational problem.

For any $u\in H^1(\mathscr D)$
with $\mathscr Lu\in L^2(\mathscr D)$, the computation made in the proof of Proposition~\ref{pr.domA2}
(which only uses that the test function $\varphi$ satisfies $\varphi|_{\SigD}=0$, not any condition on $u|_{\SigD}$) gives
\begin{equation}\label{eq:star}
  \mathbf L(u,\varphi) = \langle {-}\mathscr Lu,\varphi\rangle  _{L^2(\mathscr D)}
  + \langle \partial_{\mathbf n_e} u,\varphi\rangle_{\SigN}, \qquad \forall\,\varphi\in\HsD. 
\end{equation}
Hence $\mathscr Lu=0$ and $\partial_{\mathbf n_e}u|_{\SigN}=0$ are together
equivalent, via \eqref{eq:star}, to
$$
  \mathbf L(u,\varphi) = 0, \qquad \forall\,\varphi\in\HsD.
$$
To prove the proposition, we perform a reasoning by analysis and synthesis. 
Let  
$$\Psi_h=h+w, \quad w:=\Psi_h-h.$$
 Since we require
$\Psi_h\in H^1(\mathscr D)$ and $\Psi_h|_{\SigD}=h|_{\SigD}$, we must have $w|_{\SigD}=(\Psi_h-h)|_{\SigD}=0$,
i.e.\ $w\in\HsD$. By the equivalence established above (applied to
$u=\Psi_h=h+w\in H^1(\mathscr D)$),
\begin{equation}\label{eq.equiv3}
  \mathscr L\Psi_h=0 \ \text{ and } \ \partial_{\mathbf n_e}\Psi_h|_{\SigN}=0
  \quad\Longleftrightarrow\quad
  \mathbf L(\Psi_h,\varphi)=0, \qquad \forall\,\varphi\in\HsD.
\end{equation}
By bilinearity of $\mathbf L$, the right-hand condition reads
$\mathbf L(w,\varphi)=-\mathbf L(h,\varphi)$ for all $\varphi\in\HsD$. Therefore, the problem to solve is the following: 
\begin{equation}\label{eq:varPhih}
  \text{Find } w\in\HsD \text{ s.t. } \ \
  \mathbf L(w,\varphi)=-\mathbf L(h,\varphi), \qquad \forall\,\varphi\in\HsD.
\end{equation}

 \noindent
\underline{Step 3}. We   solve \eqref{eq:varPhih}.

The right-hand side $\varphi\mapsto -\mathbf L(h,\varphi)$ is a bounded linear
functional on $\HsD$ (namely it belongs to $\HsD^*$), by boundedness of $\mathbf L$ on $H^1(\mathscr D)\times
H^1(\mathscr D)$ (recalled after \eqref{eq:form}). 
Let us now look at the function in the left-hand side of \eqref{eq:varPhih}. 
By \eqref{eq:coercive}, $\mathbf L_\kappa$ is bounded and coercive on $\HsD$, so
by Lax-Milgram the operator
$$
  K_\kappa:\HsD\to\HsD^*, \qquad K_\kappa v:=\mathbf L_\kappa(v,\cdot),
$$
is an isomorphism. Since $\mathbf L(\cdot,\cdot)=\mathbf L_\kappa(\cdot,\cdot)
-\kappa\langle\cdot,\cdot\rangle_{L^2}$, the operator associated with
$\mathbf L$,
$$
  K_0:\HsD\to\HsD^*, \qquad K_0v:=\mathbf L(v,\cdot),
$$
satisfies $K_0=K_\kappa-\kappa J$, where $J:=\iota^*\iota:\HsD\to\HsD^*$,
$\iota:\HsD\hookrightarrow L^2(\mathscr D)$ being the embedding and
$\iota^*:L^2(\mathscr D)\to\HsD^*$ its adjoint. Since $\iota$ is compact
(Rellich-Kondrachov), so is $J=\iota^*\iota$, and hence $K_0=K_\kappa-\kappa J$ is a
compact perturbation of the isomorphism $K_\kappa$. By the Fredholm alternative,
$K_0$ is bijective as soon as it is injective.

Let us now check the  injectivity of $K_0$. Suppose $K_0 \Psi =0$, i.e.\ $\mathbf L(\Psi,\varphi)=0$ for all $\varphi\in\HsD$.
Applying the equivalence  \eqref{eq.equiv3} with $\Psi$ instead of $\Psi_h$  gives $\mathscr L\Psi=0$ and
$\partial_{\mathbf n_e}\Psi|_{\SigN}=0$, i.e.\ $\Psi\in\dom(A_2)$ and $A_2\Psi=0$. Since
$A_2$ is injective, we conclude $\Psi=0$. Hence $K_0$ is
injective, and therefore, by the Fredholm alternative, an isomorphism $\HsD\to\HsD^*$.
We have thus solved \eqref{eq:varPhih}.

\smallskip\noindent
\underline{Step 4}. Conclusion

Recall $\Psi_h:= w+h$. By construction $ \Psi_h\in  H^1(\mathscr D)$,  $\Psi_h|_{\SigD}=h|_{\SigD}$, and  $\mathbf L(\Psi_h,\varphi)=0$ for all $\varphi\in\HsD$. In particular,  by \eqref{eq.equiv3}, $\partial_{\mathbf n_e}\Psi_h|_{\SigN}=0$ and 
$$\mathscr L\Psi_h=0\in L^q(\mathscr D), \ \ \forall q\ge 1.$$ Thus, Lemma~\ref{le.regmix} applies directly to $\Psi_h$ and gives
$\Psi_h\in C(\bar{\mathscr D})\cap 
C^\infty(\bar{\mathscr D} \setminus \partial \Sigma)$.
With this regularity in hand, we can use It\=o formula on $\Psi_h(X_t)$ as we did before (use the stopping time $\tau_r$)   to deduce that 
$\Psi_h(x)=\mathbb E_x[h(X_{\tau_{\mathscr O}})]$ for every $x\in\mathscr D$. The proof of the proposition is complete. 

\end{proof}

\begin{lemma} \label{le.green2}
Let $u\in H^1(\mathscr D)$ with $\mathscr L u=0$ in $\mathscr D$ and
$\partial_{\mathbf n_e} u|_{\SigN}=0$, and let $v \in \dom(A_2^*)$ with
$A_2^*v = \lambda v$. Then:
\begin{equation}\label{eq:key}
  \lambda \int_{\mathscr D} u\,v \; dx
 = 
 - \langle u, \partial_{\mathbf n_e} v  \rangle_{\SigD}.
\end{equation}
\end{lemma}

\begin{proof}
By Proposition~\ref{prop:domAstar-L2}, $v\in\HsD$ (so $v|_{\SigD}=0$),
${-}\mathscr L^*v = \lambda v \in L^2(\mathscr D)$, and
$$\partial_{\mathbf n_e}v {-} (b\cdot\mathbf n_e)v|_{\SigN} = 0\text{  in }
 H^{-1/2}(\SigN).$$
  By hypothesis, $u\in H^1(\mathscr D)$ with $\Delta u =
{-}b\cdot\nabla u\in L^2(\mathscr D)$ (since $\mathscr Lu={\Delta u+b\cdot\nabla u=}0$) and $\partial_{\mathbf
n_e}u|_{\SigN}=0$.
Applying the Green formula \eqref{eq:green} twice (with $(w,\phi)=(u,v)$,
then $(w,\phi)=(v,u)$) and subtracting (both left-hand sides equal
$\int_\mathscr D \nabla u\cdot\nabla v$) gives Green's second identity
\begin{equation}\label{eq:green2}
  \int_{\mathscr D} (u\,\Delta v - v\,\Delta u )\,dx
  =
  {\langle u,\partial_{\mathbf n_e} v \rangle_{\partial\mathscr D}}
  -  \langle v,\partial_{\mathbf n_e} u \rangle_{\partial\mathscr D}.
\end{equation}
Since $\Delta u = {-}b\cdot\nabla u$ and, from ${-\mathscr L^*v=-\Delta v
+b\cdot\nabla v+(\operatorname{div}b)v}=\lambda v$,
 $
  \Delta v = {-\lambda v + b\cdot\nabla v +
  (\operatorname{div}b)v}$,
we get
$$
  u\,\Delta v - v\,\Delta u
  = -\lambda\,uv {+} \bigl(u(b\cdot\nabla v)+v(b\cdot\nabla u)
  + uv\,\operatorname{div}b\bigr)
  = -\lambda\,uv {+} \operatorname{div}(uv\,b).
$$
Integrating over $\mathscr D$ and applying the divergence theorem to the last
term,
$$
  \int_{\mathscr D}\bigl(u\,\Delta v - v\,\Delta u\bigr)\,dx
  = -\lambda\int_{\mathscr D} u\,v\,dx
  {+} \int_{\partial\mathscr D} u\,v\,(b\cdot\mathbf n_e) .
$$
Combining with \eqref{eq:green2}, we deduce that:
\begin{equation}\label{eq:almost}
  \lambda\int_{\mathscr D} u\,v\,dx
  =  \langle v,\partial_{\mathbf n_e} u \rangle_{\partial\mathscr D}
  -  \langle u,\partial_{\mathbf n_e}v \rangle_{\partial\mathscr D}
  {+} \int_{\partial\mathscr D} u\,v\,(b\cdot\mathbf n_e)\, .
\end{equation}
Split $\partial\mathscr D=\SigD\sqcup\SigN$. On $\SigD$, $v=0$, so the
contribution of the left-hand side reduces to $-\langle u,\partial_{\mathbf n_e}v\rangle_{\SigD}$.
On $\SigN$, $\partial_{\mathbf n_e}u=0$, so the contribution reduces to
$$
  -\langle u,\partial_{\mathbf n_e}v\rangle_{\SigN}
  {+} \int_{\SigN} u\,v\,(b\cdot\mathbf n_e)\,
  = \langle u,\, -\partial_{\mathbf n_e}v {+} (b\cdot\mathbf n_e)v
 \rangle_{\SigN} = 0,
$$
using the adjoint Robin condition $\partial_{\mathbf n_e}v =
{(b\cdot\mathbf n_e)v}$ on $\SigN$. Substituting both contributions into
\eqref{eq:almost} yields \eqref{eq:key}. The proof of the lemma is complete. 

\end{proof}

The following proposition implies the last statement of Theorem \ref{th.CorFeyn}. 
\begin{proposition} 
\label{pr.exitlaw}
Let the assumptions of Theorem \ref{th.2} hold. 
Let $\bm{\rho}(dx) = \mathbf v(x)dx/\int_{\mathscr D}\mathbf v $ be the quasi-stationary
distribution of Theorem~\ref{th.1} (see also Proposition \ref{pr.fin}), and let $\mu_{\bm{\rho}}$ denote the law of
$X_{\tau_{\mathscr O}}$ under $\mathbb P_{\bm{\rho}}$ (see Theorem \ref{th.1}). Set 
$\theta_{\mathbf v}:=-\partial_{\mathbf n_e}\mathbf v/(\lambda\int_{\mathscr D}\mathbf v) \in L^q(\partial \mathscr D)
$, for some $q>1$ (Proposition \ref{th.Brown}).
Then,  for every $f\in C^\infty(\SigD)$,
\begin{equation}\label{eq:muformula}
  \int_{\SigD} f\,d\mu_{\bm{\rho}} = \frac{-\int_{\SigD} f \,  \partial_{\mathbf n_e} \mathbf v}{\lambda\int_{\mathscr D}\mathbf v}.
\end{equation}
In particular, the 
law $\mu_{\bm\rho} $ of the exit point  
$X_{\tau_{\mathscr O}}$ under $\mathbb P_{\bm{\rho}}$ has density 
 $ \theta_{\mathbf v} $ over $\SigD$. 
\end{proposition}

\begin{proof} 
 Let $h\in C^\infty(\bar{\mathscr D})$ and $\Phi_h$ as in
Proposition~\ref{pr.Phih}. By definition of $\mu,\bm{\rho}$ and since $\Phi_h(x)=\mathbb E_x[h(X_{\tau_{\mathscr O}})]$,
$$
  \int_{\SigD} h\,d\mu_{\bm{\rho}} = \mathbb E_{\bm{\rho}}\bigl[h(X_{\tau_{\mathscr O}})\bigr]
  = \int_{\mathscr D}\mathbb E_x\bigl[h(X_{\tau_{\mathscr O}})\bigr]\,\bm{\rho}(dx)
  = \frac{1}{\int_{\mathscr D}\mathbf v}\int_{\mathscr D}\Phi_h(x)\,\mathbf v(x)\,dx.
$$
By Theorem \ref{th.3}, $\mathbf v\in \dom(A_2^*)$. 
By Proposition~\ref{pr.Phih}, $u =\Phi_h$ satisfies the hypotheses of
Lemma~\ref{le.green2}, whose identity \eqref{eq:key} gives
$$
  \lambda\int_{\mathscr D}\Phi_h\,\mathbf v\,dx
  =  
  - \langle \Phi_h,\partial_{\mathbf n_e}\mathbf v\rangle_{\SigD}
  = -\langle h,\partial_{\mathbf n_e}\mathbf v\rangle_{\SigD},
$$
using $\Phi_h|_{\SigD}=h$. Combining the two displays and the fact that  $\theta_{\mathbf v} \in L^1(\partial \mathscr D)$ 
gives \eqref{eq:muformula}.
The proof is complete. 

\end{proof}
%
 \appendix
\section{Unbounded domains: some examples}\label{app.unbounded}
 
In this appendix, we discuss some examples of unbounded domains $\mathscr D$ to which the approach of Section~\ref{sec.Gu} still applies for existence and uniqueness of a quasi-stationary distribution in some weighted spaces of probability measures. For simplicity we take $\sigma\equiv\mathrm{Id}_{\mathbb R^d}$, so that (see \eqref{eq.Lo})
$$
\mathscr L=\tfrac12\Delta+b\cdot\nabla ,
$$
and we assume that $b:\mathbb R^d\to\mathbb R^d$ is globally Lipschitz: there is $K_b\ge0$ such that $|b(x)-b(y)|\le K_b|x-y|$ for all $x,y\in\mathbb R^d$. When $\mathscr D$ is unbounded, two new difficulties appear. First, the state space $\bar{\mathscr D}$ is no longer compact, so Assumption \ass{C3} is not automatically satisfied: one needs a Lyapunov function $W$, and the drift $b$ must be confining. Second, \rev{Proposition}~\ref{th.SF} and Lemma~\ref{le.CT} are stated for bounded domains. When $\mathscr D$ is convex, their proofs simplify and apply globally, otherwise, we use a truncation argument together with local pathwise uniqueness.

\subsection{The sign of the normal derivative of the Lyapunov function}
Let $W\in C^2(\bar{\mathscr D})$. By It\=o's formula and \eqref{eq.L=}, since $\nabla W\cdot \mathbf n_i=-\partial_{\mathbf n_e}W$ on $\partial\mathscr D$,
\begin{equation}\label{eq.ItoW}
W(X_t)=W(x)+\int_0^t\mathscr LW(X_s)\,ds+\int_0^t\nabla W(X_s)\cdot dB_s-\int_0^t\partial_{\mathbf n_e}W(X_s)\,d|L|_s .
\end{equation}
The last term is the only difference with the non-reflected case. If $\partial_{\mathbf n_e}W=0$ on $\partial\mathscr D$, it vanishes, so that $W(X_t)-W(x)-\int_0^t\mathscr LW(X_s)\,ds$ is a local martingale, and thus $W$ belongs to the extended domain of the generator, with extended generator $\mathscr LW$, as required in \ass{C3}. If instead $\partial_{\mathbf n_e}W<0$ on a part of $\partial\mathscr D$, the reflection pushes the process towards larger values of $W$, and the local time contributes a positive term to \eqref{eq.ItoW}, which may destroy the Lyapunov inequality. If $\partial_{\mathbf n_e}W\ge0$ on $\partial\mathscr D$, the local-time term in \eqref{eq.ItoW} is nonpositive, so that $W(X_t)-W(x)-\int_0^t\mathscr LW(X_s)\,ds$ is a local supermartingale. This weaker property suffices as soon as \textbf{[C3]} is only used through the supermartingale property of $W(X_t)\exp\big(\int_0^t(-\mathscr LW/W)(X_s)\,ds\big)$ in \cite{guillinqsd}.

For a radial function $W(x)=\varphi(|x|)$ with $\varphi'\ge0$, one has
$$
\partial_{\mathbf n_e}W(x)=\varphi'(|x|)\,\frac{x\cdot\mathbf n_e(x)}{|x|},\qquad x\in\partial\mathscr D,
$$
so that the sign of $\partial_{\mathbf n_e}W$ is that of $x\cdot\mathbf n_e(x)$, which depends on the geometry of $\mathscr D$. In the examples below, $W$ is adapted to the geometry so that $\partial_{\mathbf n_e}W=0$ on $\partial\mathscr D$.

\subsection{Examples}
In each example, $a>0$ is a parameter to be chosen, $W\in C^2(\bar{\mathscr D})$, $W\ge1$, $\partial_{\mathbf n_e}W=0$ on $\partial\mathscr D$, and $W(x)\to+\infty$, $\mathscr LW(x)/W(x)\to-\infty$ as $|x|\to+\infty$, $x\in\bar{\mathscr D}$. Consequently, \ass{C3} holds.

\medskip
\noindent\textbf{Example 1 (half-space).} Let $\mathscr D=\{x\in\mathbb R^d,\ x_d>0\}$ and assume that there exist $c>0$ and $C\ge0$ such that
\begin{equation}\label{eq.confin}
b(x)\cdot x\le -c|x|^2+C,\qquad x\in\mathbb R^d,
\end{equation}
e.g.\ $b(x)=-x+\beta(x)$ with $\beta$ bounded and globally Lipschitz. Let $W(x):=e^{a|x|^2}$ with $0<a<c$. On $\partial\mathscr D$, $\mathbf n_e=-e_d$, hence $\partial_{\mathbf n_e}W=-2ax_dW=0$. Moreover, $\nabla W=2axW$ and $\Delta W=(2ad+4a^2|x|^2)W$, so that, by \eqref{eq.confin},
$$
\frac{\mathscr LW}{W}(x)=ad+2a^2|x|^2+2a\,b(x)\cdot x\le a(d+2C)-2a(c-a)|x|^2\xrightarrow[|x|\to\infty]{}-\infty .
$$

\noindent\textbf{Example 2 (infinite tube).} Let $\mathscr D=\omega\times\mathbb R$, where $\omega\subset\mathbb R^{d-1}$ is a bounded domain with $C^3$ boundary, and write $x=(x',x_d)\in\mathbb R^{d-1}\times\mathbb R$ and $b=(b',b_d)$. Assume that the drift confines the process along the axis of the tube:
\begin{equation}\label{eq.confin2}
b_d(x)\,x_d\le -c\,x_d^2+C,\qquad x\in\mathbb R^d,
\end{equation}
e.g.\ $b(x)=(\beta'(x),-x_d+\beta_d(x))$ with $\beta=(\beta',\beta_d)$ globally Lipschitz and $\beta_d$ bounded. Let $W(x):=e^{a x_d^2}$ with $0<a<c$. On $\partial\mathscr D=\partial\omega\times\mathbb R$, the outward normal is $\mathbf n_e(x)=(\mathbf n_e^{\omega}(x'),0)$, which is orthogonal to $\nabla W=2ax_dW\,e_d$, hence $\partial_{\mathbf n_e}W=0$. Moreover, by \eqref{eq.confin2},
$$
\frac{\mathscr LW}{W}(x)=a+2a^2x_d^2+2a\,b_d(x)\,x_d\le a(1+2C)-2a(c-a)x_d^2\xrightarrow[|x_d|\to\infty]{}-\infty ,
$$
and $|x|\to+\infty$ in $\bar{\mathscr D}$ if and only if $|x_d|\to+\infty$, since $\omega$ is bounded. 

\medskip
\noindent\textbf{Example 3 (narrow entry).} Let $\mathscr D=\mathbb R^d\setminus\bar U$, where $U$ is a nonempty bounded open set with $C^3$ boundary such that $\mathbb R^d\setminus\bar U$ is connected, and assume \eqref{eq.confin}. Here $\partial\mathscr D=\partial U$ and, taking $\Sigma\subset\partial U$, the exit time $\tau_{\mathscr O}$ is the first time the process, reflected on $\partial U\setminus\Sigma$, enters the obstacle $U$ through the window $\Sigma$. When $\Sigma$ is small, this is a narrow entry problem. The domain $\mathscr D$ is never convex. The radial function $e^{a|x|^2}$ has the wrong sign here: if, e.g., $U$ is convex and $0\in U$, then $\mathbf n_e$ points into $U$ and $x\cdot\mathbf n_e(x)<0$ on $\partial\mathscr D$. Since $\partial\mathscr D$ is compact, we instead flatten $W$ near the obstacle. Let $R_0>0$ be such that $\bar U\subset\mathscr B(0,R_0)$, and let $h\in C^\infty([0,+\infty))$ be nondecreasing, with $h=0$ on $[0,R_0]$ and $h(r)=r^2$ for $r\ge R_0+1$. Set $W(x):=e^{a h(|x|)}$ with $0<a<c$. Then $W\equiv1$ on $\mathscr B(0,R_0)$, which is a neighborhood of $\partial\mathscr D$, so that $\partial_{\mathbf n_e}W=0$. In addition, $\mathscr LW$ is bounded on $\bar{\mathscr D}\cap\bar{\mathscr B}(0,R_0+1)$, and for $|x|\ge R_0+1$ the computation of Example~1 applies. 

\subsection{Existence and uniqueness of a quasi-stationary distribution}

\begin{theorem}\label{th.unbounded}
Let $d\ge2$, $\sigma\equiv\mathrm{Id}_{\mathbb R^d}$, and let $b:\mathbb R^d\to\mathbb R^d$ be globally Lipschitz. Assume that one of the following holds:
\begin{enumerate}
\item[$\mathfrak 1$.] \emph{(Convex case)} $\mathscr D$ and $b$ are as in Example~1, or as in Example~2 with $\omega$ convex.
\item[$\mathfrak 2$.] \emph{(Non-convex case)} $\mathscr D$ and $b$ are as in Example~2 with $\omega$ non-convex, or as in Example~3.
\end{enumerate}
Let $W$ be the Lyapunov function of the corresponding example. Then:
\begin{enumerate}
\item[(i)] For every $x\in\bar{\mathscr D}$, \eqref{eq.LR} has a pathwise unique strong solution $(X_t(x))_{t\ge0}$ defined for all times, and $(X_t)_{t\ge0}$ is a strong Markov process with continuous paths.
\item[(ii)] Let $\Sigma\subset\partial\mathscr D$ satisfy \ass{C} (e.g.\ $\Sigma$ with nonempty interior in $\partial\mathscr D$), and set $\mathscr O=\bar{\mathscr D}\setminus\bar\Sigma$. Then \ass{C1}$\to$\ass{C5} hold on $\mathscr E=\bar{\mathscr D}$ with the Lyapunov function $W$. Consequently, all the conclusions of Theorem~\ref{thm:main} hold for $(X_t)_{t\ge0}$ with absorbing boundary part $\bar\Sigma$. In particular, for every $p>1$, there exists a unique quasi-stationary distribution in $\mathscr P_{W^{1/p}}(\mathscr O)$.
\end{enumerate}
\end{theorem}

The techniques used in Theorem \ref{th.unbounded} are adaptable and apply beyond the present examples.  Recall also that the unique quasi-stationary distribution in $\mathscr P_{W^{1/p}}(\mathscr O)$ is independent of $p>1$. 

\begin{proof}
In all cases, since $\mathscr LW/W$ is continuous on $\bar{\mathscr D}$ and tends to $-\infty$ at infinity, there is $c_W\ge0$ such that $\mathscr LW\le c_WW$ on $\bar{\mathscr D}$.

\underline{Step 1}. Well-posedness, \ass{C1} and \ass{C2} in the convex case.

Since $\mathscr D$ is convex and $b$ is globally Lipschitz (hence has linear growth), \eqref{eq.LR} has, for every $x\in\bar{\mathscr D}$, a pathwise unique strong solution defined for all times, by~\cite[Theorem~4.1]{tanaka}. By pathwise uniqueness, $(X_t)_{t\ge0}$ is a strong Markov process. Recall that, $\mathscr D$ being convex,
\begin{equation}\label{eq.convex}
(y-z)\cdot\mathbf n_i(z)\ge0,\quad\text{for all } z\in\partial\mathscr D,\ y\in\bar{\mathscr D},
\end{equation}
i.e.\ \eqref{eq.a} holds with $C_{\partial\mathscr D}=0$.

\noindent\textit{Proof of \ass{C2}.} Let $x,z\in\bar{\mathscr D}$, and let $X=X(x)$ and $X'=X(z)$ be driven by the same Brownian motion, with local times $L$ and $L'$. Since $\sigma=\mathrm{Id}_{\mathbb R^d}$, the process $X_t-X'_t=x-z+\int_0^t(b(X_s)-b(X'_s))\,ds+L_t-L'_t$ has finite variation, so that
\begin{align*}
|X_t-X'_t|^2&=|x-z|^2+2\int_0^t(X_s-X'_s)\cdot(b(X_s)-b(X'_s))\,ds\\
&\quad +2\int_0^t(X_s-X'_s)\cdot\mathbf n_i(X_s)\,d|L|_s-2\int_0^t(X_s-X'_s)\cdot\mathbf n_i(X'_s)\,d|L'|_s .
\end{align*}
Since $|L|$ (resp.\ $|L'|$) only increases when $X\in\partial\mathscr D$ (resp.\ $X'\in\partial\mathscr D$), the last two terms are nonpositive by \eqref{eq.convex}. By Gronwall's lemma, we get a.s.
\begin{equation}\label{eq.Lipflow}
\sup_{t\in[0,T]}|X_t(x)-X_t(z)|\le e^{K_bT}\,|x-z|,\qquad T\ge0,\ x,z\in\bar{\mathscr D}.
\end{equation}
In particular, if $x_n\to x$, then $X_{[0,T]}(x_n)\to X_{[0,T]}(x)$ a.s.\ in $C([0,T],\bar{\mathscr D})$, which implies \ass{C2}.

\noindent\textit{Proof of \ass{C1}.} We use the coupling of the proof of \rev{Proposition}~\ref{th.SF}. Let $T>0$, $\alpha\in(0,1)$, $x\neq y\in\bar{\mathscr D}$, and $X=X(x)$. Let $(Y_t,K_t)_{t\ge0}$ be the solution, starting at $y$, of
$$
dY_t=b(Y_t)\,dt+|x-y|^\alpha\,\bar Z_t\,\mathbf 1_{t<\tau}\,dt+dB_t+dK_t,\qquad \bar Z_t:=\frac{X_t-Y_t}{|X_t-Y_t|},
$$
where $K$ is the reflection term of $Y$ on $\partial\mathscr D$ (satisfying $\mathfrak 1$-$\mathfrak 3$ in the proof of \rev{Proposition}~\ref{th.SF}), $\tau:=\inf\{t\ge0,\ X_t=Y_t\}$, and $Y_t=X_t$ for $t\ge\tau$. It is constructed as in the proof of \rev{Proposition}~\ref{th.SF}, by solving the equation up to the stopping times $\inf\{t\ge0,\ |X_t-Y_t|\le\varepsilon\}$, on which the coupling drift is Lipschitz in $Y$, and letting $\varepsilon\to0$. On $[0,\tau)$, $X-Y$ has finite variation and does not vanish, so that, using the Lipschitz property of $b$ and \eqref{eq.convex} as above,
\begin{align*}
d|X_t-Y_t|&=\bar Z_t\cdot(b(X_t)-b(Y_t))\,dt-|x-y|^\alpha\,dt+\bar Z_t\cdot\mathbf n_i(X_t)\,d|L|_t-\bar Z_t\cdot\mathbf n_i(Y_t)\,d|K|_t\\
&\le\big(K_b|X_t-Y_t|-|x-y|^\alpha\big)\,dt .
\end{align*}
Hence $|X_t-Y_t|\le|x-y|e^{K_bt}$ for $t<\tau$ and then, plugging this bound into the previous inequality,
$$
0<|X_t-Y_t|\le|x-y|\,e^{K_bt}-|x-y|^\alpha t,\qquad t<\tau .
$$
Assume that 
$$|x-y|^{1-\alpha}e^{K_bT}<T.$$
Then $\tau<T$ a.s.  and consequently $Y_T=X_T$. Note that, since $\alpha<1$, the condition $|x-y|^{1-\alpha}e^{K_bT}<T$ holds as soon as $|x-y|$ is small enough. Set $W^*_t:=B_t+\int_0^{t\wedge\tau}|x-y|^\alpha \bar Z_s\,ds$ and
$$
\mathscr M_T:=\exp\Big(-\int_0^{T\wedge\tau}|x-y|^\alpha \bar Z_s\cdot dB_s-\frac12|x-y|^{2\alpha}(T\wedge\tau)\Big).
$$
Since $|\bar Z_s|=1$, $\mathscr M$ is a true martingale, $(W^*_t)_{t\in[0,T]}$ is a Brownian motion under $\mathscr M_T\mathbb P$ by Girsanov's theorem, and $\mathbb E[\mathscr M_T^2]\le e^{T|x-y|^{2\alpha}}$. Since $dY_t=b(Y_t)\,dt+dW^*_t+dK_t$, the law of $Y_T$ under $\mathscr M_T\mathbb P$ is the law of $X_T(y)$ under $\mathbb P$, by uniqueness in law. Using also $Y_T=X_T$ a.s., we get, for all $f\in b\mathcal B(\bar{\mathscr D})$,
$$
|P_Tf(x)-P_Tf(y)|=\big|\mathbb E[(1-\mathscr M_T)f(X_T)]\big|\le\|f\|_\infty\big(\mathbb E[\mathscr M_T^2]-1\big)^{1/2}\le\|f\|_\infty\big(e^{T|x-y|^{2\alpha}}-1\big)^{1/2},
$$
whenever $|x-y|^{1-\alpha}e^{K_bT}<T$. Hence $P_Tf$ is continuous on $\bar{\mathscr D}$ for every $T>0$, which is \ass{C1}.

\underline{Step 2}. Well-posedness, \ass{C1} and \ass{C2} in the non-convex case.

\noindent\textit{Truncated domains.} For $k\ge R_0+1$ in Example~3, set $V_k:=\bar{\mathscr D}\cap\mathscr B(0,k)$ and $\mathscr D_k:=\mathscr B(0,k+1)\setminus\bar U$, which is a bounded $C^3$ subdomain of $\mathscr D$. For $k\ge1$ in Example~2, set $V_k:=\bar{\mathscr D}\cap\{|x_d|<k\}$ and let $\mathscr D_k$ be a bounded $C^3$ subdomain of $\mathscr D$ such that $\mathscr D_k\cap\{|x_d|<k+1\}=\mathscr D\cap\{|x_d|<k+1\}$ \rev{(for instance, $\mathscr D_k$ can be obtained from the cylinder $\omega\times(-k-2,k+2)$ by smoothing its edges $\partial\omega\times\{\pm(k+2)\}$ inside $\{k+1<|x_d|<k+2\}$)}. 
In both cases, $V_k$ is open in $\bar{\mathscr D}$, $(V_k)_k$ is nondecreasing with $\bigcup_kV_k=\bar{\mathscr D}$, $\bar V_k$ is compact and contained in the region where $\mathscr D_k$ and $\mathscr D$ (and hence their boundaries and inward normals) coincide, and
$$
m_k:=\inf_{\bar{\mathscr D}\setminus V_k}W\ge e^{ak^2}\xrightarrow[k\to\infty]{}+\infty .
$$
For $x\in\bar{\mathscr D}_k$, let $(X^k_t(x),L^k_t(x))_{t\ge0}$ be the solution of \eqref{eq.LR} in $\bar{\mathscr D}_k$ with drift $b$ and $\sigma=\mathrm{Id}_{\mathbb R^d}$, driven by $(B)_{t\ge 0}$~\cite{lionsRSDE}. \rev{Proposition}~\ref{th.SF} and Lemma~\ref{le.CT} apply to $X^k(x)$. For $j\le k$, denote by $\tau^k_j $ the exit time of $X^k(x)$ from $V_j$.

\noindent\textit{Local uniqueness.} Let $j\le k\le k'$ and $x\in V_j$. Up to the exit times of $X^k(x)$ and $X^{k'}(x)$ from $V_j$, both processes live in $\bar V_j$, where $\mathscr D_k$, $\mathscr D_{k'}$ and $\mathscr D$ coincide. The pathwise uniqueness argument of~\cite{lionsRSDE}, performed up to the minimum of these two stopping times, shows that $X^k(x)$ and $X^{k'}(x)$ coincide up to this minimum. Hence the two exit times are equal, $\tau^k_j=\tau^{k'}_j$, and $X^k(x)=X^{k'}(x)$ on $[0,\tau^k_j]$.

\noindent\textit{Lyapunov estimate.} Let $x\in V_k$. For $s\le\tau^k_k$, $X^k_s(x)\in\bar V_k$, where $\mathscr D_k$ and $\mathscr D$ coincide. Hence It\=o's formula \eqref{eq.ItoW} holds for $W(X^k_t)$ up to time $\tau^k_k$, and the local-time term vanishes since $\partial_{\mathbf n_e}W=0$ on $\partial\mathscr D$. Since $\mathscr LW-c_WW\le0$, we get a.s.
$$
e^{-c_W(t\wedge\tau^k_k)}W(X^k_{t\wedge\tau^k_k})\le W(x)+\int_0^{t\wedge\tau^k_k}e^{-c_Ws}\,\nabla W(X^k_s)\cdot dB_s .
$$
Since $X^k_{\tau^k_k}\in\bar{\mathscr D}\setminus V_k$ on $\{\tau^k_k<+\infty\}$, so that $W(X^k_{\tau^k_k})\ge m_k$ there. Using $W\ge0$, we deduce
\begin{equation}\label{eq.locWk}
\mathbb P[\tau^k_k(x)\le t]\le \frac{e^{c_Wt}}{m_k}\,W(x),\qquad x\in V_k,\ t\ge0.
\end{equation}

\noindent\textit{Well-posedness.} By local uniqueness, $k\mapsto\tau^k_k$ is nondecreasing, since $\tau^k_k=\tau^{k'}_k\le\tau^{k'}_{k'}$ for $k\le k'$. By \eqref{eq.locWk}, $\tau^k_k\uparrow+\infty$ a.s.\ as $k\to+\infty$. Setting $X_t(x):=X^k_t(x)$ and $L_t(x):=L^k_t(x)$ for $t\le\tau^k_k$ thus consistently defines a solution of \eqref{eq.LR} in $\bar{\mathscr D}$ for all times. The same argument shows that any solution of \eqref{eq.LR} in $\bar{\mathscr D}$ driven by $B$ and starting at $x$ coincides with $X^k(x)$ on $[0,\tau^k_k]$ for all $k$, hence with $X(x)$. Thus \eqref{eq.LR} is pathwise well posed in $\bar{\mathscr D}$, and $(X_t)_{t\ge0}$ is a strong Markov process with continuous paths. Moreover, denoting by $\tau_k$ the exit time of $X$ from $V_k$, we have $\tau_k=\tau^k_k$,
\begin{equation}\label{eq.XXk}
X_t(x)=X^k_t(x) \ \text{ for all } t\in[0,\tau_k(x)], \ \text{ a.s.},
\end{equation}
and, by \eqref{eq.locWk} (the bound being trivial when $x\notin V_k$, since then $W(x)\ge m_k$),
\begin{equation}\label{eq.locW}
\mathbb P_x[\tau_k\le t]\le \frac{e^{c_Wt}}{m_k}\,W(x),\qquad x\in\bar{\mathscr D},\ t\ge0.
\end{equation}

\noindent\textit{Proof of \ass{C1}.} Let $t>0$ and $f\in b\mathcal B(\bar{\mathscr D})$. Define $f_k:=f\mathbf 1_{V_k}$, extended by $0$ to $\bar{\mathscr D}_k$, and let $P^k_t$ be the semigroup of $X^k$. By \eqref{eq.XXk}, $f(X_t(x))=f_k(X^k_t(x))$ on $\{t<\tau_k\}$, so that, by \eqref{eq.locW}, for all $x\in V_k$,
$$
|P_tf(x)-P^k_tf_k(x)|\le 2\|f\|_\infty\,\mathbb P_x[\tau_k\le t]\le 2\|f\|_\infty\,\frac{e^{c_Wt}}{m_k}\,W(x).
$$
Let $K$ be a compact subset of $\bar{\mathscr D}$. Then $K\subset V_k$ for $k$ large enough, and $P^k_tf_k$ is continuous on $K$ by \rev{Proposition}~\ref{th.SF}. Hence $P_tf$ is the uniform limit on $K$, as $k\to+\infty$, of continuous functions, and is therefore continuous on $K$. Since every point of $\bar{\mathscr D}$ has a compact neighborhood in $\bar{\mathscr D}$, $P_tf$ is continuous on $\bar{\mathscr D}$.

\noindent\textit{Proof of \ass{C2}.} Let $T>0$, $x_n\to x$ in $\bar{\mathscr D}$, and let $k$ be such that the compact set $K:=\{x\}\cup\{x_n,\ n\ge0\}$ is contained in $V_k$. By \eqref{eq.XXk}, $\sup_{t\in[0,T]}|X_t(x_n)-X_t(x)|=\sup_{t\in[0,T]}|X^k_t(x_n)-X^k_t(x)|$ on $\{\tau_k(x_n)>T\}\cap\{\tau_k(x)>T\}$. Hence, for every $\varepsilon>0$, by Markov's inequality, \eqref{eq.Ex} applied to $X^k$ (with $p=1$ and a constant $C_k$) and \eqref{eq.locW},
$$
\mathbb P\Big[\sup_{t\in[0,T]}|X_t(x_n)-X_t(x)|\ge\varepsilon\Big]\le \frac{C_k}{\varepsilon^2}\,|x_n-x|^2+\frac{2e^{c_WT}}{m_k}\,\sup_KW .
$$
Letting $n\to+\infty$ and then $k\to+\infty$, we get $X_{[0,T]}(x_n)\to X_{[0,T]}(x)$ in probability in $C([0,T],\bar{\mathscr D})$ endowed with the uniform distance. Hence the laws converge weakly, which is \ass{C2}.
 
 \underline{Step 3}. Proofs of \ass{C3}, \ass{C4}, and \ass{C5}.
 
 Condition \ass{C3} follows from the Lyapunov inequalities established in the examples, $W$ belonging to the extended domain by \eqref{eq.ItoW} (recall $\partial_{\mathbf n_e}W=0$) with the localizing times $\inf\{t\ge0,\ |X_t|\ge n\}$. The proof  Proposition~\ref{pr.C5}  and the third proof of Propositions~\ref{pr.C4} do not assumed $\mathscr D$ to be bounded and give respectively \ass{C5} and \ass{C4}.

Theorem~\ref{thm:main} then gives the conclusion.

\end{proof}

\section{Some useful results of operator theory}\label{sec.OperatorTheory}

A linear subspace $J \subseteq X$ of a Banach lattice $X$ is called an ideal if it satisfies the following property:
$$\text{If } f \in J \text{ and } g \in X \text{ satisfy } \vert g\vert  \le \vert f\vert , \text{ then } g \in J,$$
where $|f|=f^++f^-$ recalling $f^+ = \max(f,0)$ and denoting $f^- = \max(-f,0)$.
Let  $T \in \mathscr{L}(X)$ be a positive   linear operator. A closed ideal $J \subseteq X$ is $T$-invariant if, by definition, $T(J) \subseteq J$. The operator $T$ is called  irreducible if its only closed invariant ideals are the trivial ones, $\{0\}$ and $X$ (see  \cite{dePagter}).
The starting point of the proof of Theorem~\ref{th.3} is  the following form of the Krein-Rutman(-de Patger) theorem
(see \cite[Theorem~19.2]{Deimling} and \cite[Theorem~3]{dePagter}).

\begin{theorem}\label{KR}
Let $X$ a Banach Lattice (of dimension greater than one), $K \subset X$ a
total cone, and $T \in \mathcal{L}(X)$ a positive irreducible compact
operator, then the spectral radius satisfies $ \mathsf r_{sp}(T)>0$ and there exist
$u \in K \setminus \{0\}$, $u^* \in K^* \setminus \{0\}$ such that
$$
Tu = \mathsf r_{sp}(T)u
\qquad\text{and}\qquad
T^*u^* = \mathsf r_{sp}(T)u^*,
$$
where $K^* \subset X^*$ denotes the dual cone of $K$.
\end{theorem}

Note that we do not assume the cone to be solid, therefore, this  standard Krein-Rutman theorem does not yield uniqueness of the eigenvectors.

The following proposition is  well-known, we recall it for completeness. 

\begin{proposition}\label{pr.irr}
Let $X:=L^p(\mathscr D)$ for some $p\in(1,+\infty)$, and let $T\in\mathcal L(X)$
be a positive operator such that
\begin{equation}
  0\le f\in X,\ f\not\equiv0 \implies Tf>0 \text{ a.e. in }\mathscr D. \tag{\ref{eq.irr}}
\end{equation}
Then $T$ is irreducible: its only closed $T$-invariant ideals are $\{0\}$ and
$X$.
\end{proposition}

\begin{proof}
Let $J\subseteq X$ be a closed ideal with $T(J)\subseteq J$, and suppose
$J\ne\{0\}$. We show $J=X$.

\underline{Step 1}. $J$ contains a strictly positive element.

Since $J\ne\{0\}$, pick $g\in J$, $g\not\equiv0$. Because $J$ is an ideal and
$|g|\le|g|$, we have $|g|\in J$, and $J$ being $T$-invariant yields $h=T|g|\in J$. By \eqref{eq.irr} applied to $|g|$, we have $h>0$ a.e. in $\mathscr D$.

\underline{Step 2}. An ideal containing an a.e.\ strictly positive element is all
of $X$.

 Let $0\le\varphi\in X$ be arbitrary. For $n\in\mathbb N$ set
$$
  \varphi_n := \varphi \wedge (nh) = \min(\varphi, nh).
$$
Since $h\in J$ and $J$ is a linear subspace, $nh\in J$. Moreover
$0\le\varphi_n\le nh$, i.e.\ $|\varphi_n|\le|nh|$, so by the ideal property, we get $\varphi_n\in J$ for every $n$. 
Since $h>0$ a.e.\ in $\mathscr D$, for a.e.\ $x\in\mathscr D$ we have
$nh(x)>\varphi(x)$ as soon as $n>\varphi(x)/h(x)$, so $\varphi_n(x)=\varphi(x)$
for $n$ large (say $n\ge n_x$). Therefore, we have that  $\varphi_n\to\varphi$ a.e.\ as $n\to+\infty$. Moreover
$0\le\varphi-\varphi_n\le\varphi\in L^p(\mathscr D)$, so by dominated
convergence $\varphi_n\to\varphi$ in $L^p(\mathscr D)$-norm. Since $J$ is
closed and $\varphi_n\in J$ for every $n$, we conclude $\varphi\in J$.

This shows $\varphi\in J$ for every $0\le\varphi\in X$. For a general
$\varphi\in X$, write $\varphi=\varphi^+-\varphi^-$ and recall that
$\varphi^+,\varphi^-\ge0$; both belong to $J$ by the above, and $J$ is a linear
subspace, so $\varphi\in J$. Hence $J=X$.

\end{proof}

\begin{lemma} \label{lem:domination}
Let $p\in(1,+\infty)$ and let $T\in\mathcal L(L^p(\mathscr D,\mathbb C))$ be positive, i.e.\
$T(L^p_+(\mathscr D))\subset L^p_+(\mathscr D)$. Then
$$
|Tg|\le T|g|\qquad\text{a.e.\ in }\mathscr D,\qquad\forall g\in L^p(\mathscr D,\mathbb C).
$$
\end{lemma}

\begin{proof}
Write $g=g_1+ig_2$ with $g_1,g_2\in L^p(\mathscr D,\mathbb R)$, and set $h:=|g|=(g_1^2+g_2^2)^{1/2}\in L^p_+(\mathscr D)$.
For every $\theta\in\mathbb R$, the Cauchy-Schwarz inequality gives, pointwise a.e.,
$$
\cos\theta\, g_1+\sin\theta\, g_2\ \le\ h,
$$
so $h-(\cos\theta\, g_1+\sin\theta\, g_2)\in L^p_+(\mathscr D)$. Thus, by linearity and positivity of $T$,
\begin{equation}\label{eq.domtheta}
\cos\theta\, Tg_1+\sin\theta\, Tg_2\ \le\ Th\qquad\text{a.e.\ in }\mathscr D,
\end{equation}
for every fixed $\theta$, outside a (Lebesgues-)null set $N_\theta$. Since
$\theta\mapsto a\cos\theta+b\sin\theta$ is continuous for any fixed $a,b\in\mathbb R$, the
supremum over $\theta\in\mathbb R$ of the left-hand side of \eqref{eq.domtheta} equals the
supremum over $\theta\in\mathbb Q\cap[0,2\pi)$. Let $N:=\bigcup_{\theta\in\mathbb
Q\cap[0,2\pi)}N_\theta$, which is also a (Lebesgues-)null set. For $x\in\mathscr D\setminus N$, taking the supremum of
\eqref{eq.domtheta} over rational $\theta$ yields
$$
|Tg(x)|=\sup_{\theta\in\mathbb Q\cap[0,2\pi)}\big(Tg_1(x)\cos\theta+Tg_2(x)\sin\theta\big)\ \le\ Th(x).
$$
This proves $|Tg|\le T|g|$ a.e.\ in $\mathscr D$.

\end{proof}

\noindent
\textbf{Acknowledgements}.   Part of this work is funded by the Agence Nationale de la Recherche through the project DySLoS (ANR-25-CE40-6875).

\bibliographystyle{amsplain}
\bibliography{ref}

\end{document}